\documentclass[10pt, leqno]{article}

\usepackage{amsmath,amssymb,amsthm, epsfig}

\usepackage{hyperref}

\usepackage{amsmath}

\usepackage{amssymb}

\usepackage{hyperref}

\usepackage{color}

\usepackage{ulem}

\usepackage{epsfig}

\usepackage{mathrsfs}

\usepackage{wasysym}

\usepackage{esint}

\title{Universal regularity estimates for solutions to fully nonlinear elliptic equations with oblique boundary data}
\author{Junior da S. Bessa, Jo\~{a}o Vitor  da Silva \\$\&$ \\ Gleydson C. Ricarte}

\newlength{\hchng}
\newlength{\vchng}
\def \diam {\mathrm{diam}}

\newcommand{\pe}{P_{\varepsilon}}
\newcommand{\defeq}{\mathrel{\mathop:}=}

\newtheorem{theorem}{Theorem}[section]

\newtheorem{lemma}[theorem]{Lemma}

\newtheorem{proposition}[theorem]{Proposition}

\newtheorem{corollary}[theorem]{Corollary}

\theoremstyle{definition}

\newtheorem{statement}[]{Statement}

\newtheorem{definition}[theorem]{Definition}

\theoremstyle{remark}

\newtheorem{remark}[theorem]{Remark}

\numberwithin{equation}{section}

\newcommand{\intav}[1]{\mathchoice {\mathop{\vrule width 6pt height 3 pt depth  -2.5pt
\kern -8pt \intop}\nolimits_{\kern -6pt#1}} {\mathop{\vrule width
5pt height 3  pt depth -2.6pt \kern -6pt \intop}\nolimits_{#1}}
{\mathop{\vrule width 5pt height 3 pt depth -2.6pt \kern -6pt
\intop}\nolimits_{#1}} {\mathop{\vrule width 5pt height 3 pt depth
-2.6pt \kern -6pt \intop}\nolimits_{#1}}}

\begin{document}
\maketitle

\begin{abstract}
In this work, we establish universal moduli of continuity for viscosity solutions to fully nonlinear elliptic equations with oblique boundary conditions, whose general model is given by
$$
\left\{
\begin{array}{rcl}
F(D^2u,x) &=& f(x) \quad \mbox{in} \,\,   \Omega\\
\beta(x) \cdot Du(x) + \gamma(x) \, u(x)&=& g(x)  \quad \mbox{on} \,\,  \partial \Omega.
\end{array}
\right.
$$
Such regularity estimates are achieved by exploring the integrability properties of $f$ based on different scenarios, making a $\text{VMO}$ assumption on the coefficients of $F$, and by considering suitable smoothness properties on the boundary data $\beta, \gamma$ and $g$. Particularly, we derive sharp estimates for borderline cases where $f \in L^n(\Omega)$ and $f\in p-\textrm{BMO}(\Omega)$. Additionally, for source terms in $L^p(\Omega)$, for $p \in (n, \infty)$, we obtain sharp gradient estimates. Finally, we also address Schauder-type estimates for convex/concave operators and suitable H\"{o}lder data.

\medskip


\medskip

\end{abstract}

\tableofcontents
\newpage

\section{Introduction}

\hspace{0.5 cm}
In this work, we focus our attention on studying the sharp regularity for viscosity solutions of fully nonlinear elliptic models with oblique boundary conditions as follows
\begin{equation}\label{1}
\left\{
\begin{array}{rcl}
F(D^2u(x), x) &=& f(x) \quad \mbox{in} \,\,   \Omega\\
 \beta(x) \cdot Du(x) + \gamma(x) \, u(x)&=& g(x)  \quad \mbox{on} \,\,  \partial \Omega,
\end{array}
\right.
\end{equation}
where $\Omega \subset \mathbb{R}^n$ is a bounded and open domain with regular boundary $\partial \Omega$, $f \in L^p(\Omega)$, with $p\geq p_0(n,\lambda, \Lambda)=n- \varepsilon_0(n,\lambda,\Lambda) \in \left(\frac{n}{2},n \right)$, where $\varepsilon_0$ is the constant introduced by Escauriaza in \cite{Esc} and $0<\lambda \le\Lambda$ are the ellipticity constants of $F$. Additionally, we consider the case where $f \in p-\textrm{BMO}(\Omega)$ (to be defined soon). Moreover, the boundary data $\beta$, $\gamma$ and $g$ belong to suitable H\"{o}lder spaces, 
and $F:\textrm{Sym}(n) \times \Omega \to \mathbb{R}$ is a second-order, fully nonlinear, uniformly elliptic operator (to be clarified soon), where $\textrm{Sym}(n)$ denotes the set of $n \times n$ symmetric matrices. For more details on such assumptions and additional definitions, please refer to Section \ref{Sec02}, see assumptions $(\mathrm{A1})$ and $(\mathrm{A1})$.

In the succeeding part, we will summarize our optimal regularity estimates in the following classification table (see Theorems \ref{holderoptimal}, \ref{theorem4.2} and \ref{Holderoptimalgradientestimate} for more details):

\begin{table}[h]
\centering
\begin{tabular}{c|c|c}
{\it Source term} & {\it Boundary data}  & {\it Optimal regularity for solutions} \\
\hline
$f\in L^{p}(\Omega), \ p\in[n-\varepsilon_{0},n)$ & $\beta,\gamma,g\in C^{0,\alpha}(\partial \Omega)$ & $C^{0, 2-\frac{n}{p}}(\overline{\Omega})$ \\
\hline
$f\in L^{n}(\Omega)$ & $\beta,\gamma,g\in C^{0,\alpha}(\partial \Omega)$ & $C^{0, \text{Log-Lip}}(\overline{\Omega})$ \\
\hline
$f\in L^{p}(\Omega), \ n<p< \infty$ & $\beta,\gamma,g \in C^{0,\alpha}(\partial\Omega)$ & $C^{1,\min\left\{\alpha_{\text{Hom}}^{-},\frac{p-n}{p}\right\}}(\overline{\Omega})$,\\
\end{tabular}\caption{H\"{o}lder, Log-Lipschitz and gradient estimates}\label{Table01}
\end{table}
where the functional space $C^{0, \text{Log-Lip}}(\overline{\Omega})$ implies that $u$ satisfies the following bounds
\begin{equation}\label{Def-Log-Lip}
  \displaystyle [u]_{0, \text{Log-Lip}, \overline{\Omega}} \defeq \sup_{x, y \in \overline{\Omega} \atop{x \neq y}} \frac{|u(x)-u(y)|}{|x-y||\ln(|x-y|)|} \le \mathbf{C}<\infty.
\end{equation}
In such a context, $\alpha_{\text{Hom}} \in (0, 1]$ denotes the optimal H\"{o}lder regularity exponent for the first derivatives of solutions to the homogeneous problem with frozen coefficients, and $\kappa^{-}$ means that we can choose any value $\tau \in (0, \kappa)$.

Moreover, we must notice that the above definition \eqref{Def-Log-Lip} yields the inclusions
$$
 C^{0, \alpha}(\overline{\Omega}) \supset C^{0, \text{Log-Lip}}(\overline{\Omega}) \supset C^{0, 1}(\overline{\Omega}) \quad \text{for every}\,\,\, \alpha \in (0, 1).
$$

Finally, by assuming convexity/concavity conditions on the operator, we observe the following classification scenarios (see, Theorems \ref{theorem4.4.5} and \ref{ThmSchauder} for more details):

\begin{table}[h]
	\centering
\begin{tabular}{c|c|c}
		{\it Source term} & {\it Boundary data}  & {\it Optimal regularity for solutions} \\
		\hline
		$f\in p-\textrm{BMO}(\Omega)\cap L^{p}(\Omega), p\in[n-\varepsilon_{0},\infty]$ & $\beta,\gamma,g\in C^{1,\alpha}(\partial \Omega)$ & $C^{1, \text{Log-Lip}}(\overline{\Omega})$  \\
		\hline
		$f\in C^{0,\alpha}(\Omega)$ & $\beta,\gamma,g\in C^{1,\alpha}(\partial\Omega)$ & $C^{2,\alpha}(\overline{\Omega})$\\
	\end{tabular}\caption{Higher order estimates}
\end{table}
where the property that $u \in C^{1, \text{Log-Lip}}(\overline{\Omega})$ means that the following semi-norm is finite
\begin{equation}\label{BordEst02}
  \displaystyle  [u]_{1, \text{Log-Lip}, \overline{\Omega}} \defeq \sup_{x, y \in \overline{\Omega} \atop{x \neq y}} \frac{|u(x)-u(y)-D u(y)\cdot (x-y)|}{|x-y|^2|\ln(|x-y|)|}.
\end{equation}

Similarly to above argument, we have the following inclusions
$$
 C^{1, \alpha}(\overline{\Omega}) \supset C^{1, \text{Log-Lip}}(\overline{\Omega}) \supset C^{1, 1}(\overline{\Omega}) \quad \text{for every}\,\,\, \alpha \in (0, 1).
$$

One of the interesting features of such an approach in obtaining the borderline estimates \eqref{BordEst02}
is that it can be quite adjusted to obtain Hessian estimates for solutions of a particular class of nonlinear oblique problems, see final comments in Section \ref{Section06} for more details.

We must highlight that while we follow the outline of \cite{LiZhang} and \cite{ET}, new challenges arise due to the presence of tangential derivatives on the boundary condition in our scheme of iterative approximation. Additionally, in some setting, we improve and optimize the moduli of continuity within our geometric approach in each integrability scenario.

Therefore, as far as we know trying to classify the moduli of continuity for solutions of general nonlinear problems like \eqref{1} under suitable assumptions on data is an open and emerging issue in the modern area of elliptic regularity theory.

 We emphasize that the aforementioned assumptions (see, Table \ref{Table01}) could be relaxed further. Nevertheless, we have decided to state our estimates based on these conditions to uphold the manuscript's presentation cleaner. In particular, our findings are noteworthy even in the simplest model case driven by Bellman-type equations involving Neumann boundary conditions:
$$
\left\{
\begin{array}{rclcl}
	\displaystyle\sup_{\iota \in \mathcal{A}}\left\{\mbox{Tr}(\mathfrak{A}^{\iota}(x)D^{2}u)-f ^{\iota}(x)\right\} &=& 0& \mbox{in} & \Omega \\
	\overrightarrow{n}(x)\cdot Du(x)&=& \mathbf{c}_0 &\mbox{on}& \partial \Omega.
\end{array}
\right.
$$

We must highlight that the interest in studying models like \eqref{1} is justified by their numerous applications in various fields beyond mathematics. For example, these models find application in the theory of Markov processes (such as in the case of Brownian motion -  see \cite{Taira20}) the equation governing the oblique condition, i.e.,
	\begin{eqnarray}\label{2}
		\beta(x)\cdot Du(x)+\gamma(x) u(x)=g(x) \quad \text{on} \quad \partial \Omega
	\end{eqnarray}
naturally appears in such scenarios. Here, the first term on the left side of \eqref{2} describes the reflection process along the $\beta$ vector field, while the second one is related to the absorption phenomenon.

Another interesting scenario where a general condition like \eqref{2} occurs in a geometric context is in the long-time existence and convergence for the inverse mean curvature flow with a Neumann boundary condition driven by geometric evolution PDEs. For a modern compendium on mean curvature flow and related issues, refer to \cite{RitSin10}.

In addition to the above examples, models like \eqref{1} find applications in the theory of celestial bodies, shocks reflected in transonic flows, and stochastic control theory, among other contexts. This breadth of applications is described in more detail in Lieberman's fundamental work in \cite{Lieb01}.
	
By way of illustration, the simplest example of the regular oblique boundary condition is the Neumann condition, where $\beta = \overrightarrow{\textbf{n}}$ and $\gamma = 0$, with $\overrightarrow{\textbf{n}}$ representing the outward normal vector of $\partial \Omega$. Therefore, we can view the condition \eqref{2} as a sort of generalization of the Neumann boundary condition.

\subsection*{State-of-the-Art: form Dirichlet to oblique boundary condition}

\hspace{0.6cm}Before delving into our results, we will briefly outline some of the relevant literature on fully nonlinear models and their connections with interior/boundary regularity and general boundary conditions.

The study of optimal regularity in the scenario of fully nonlinear elliptic equations has been widely explored in recent years for a number of researchers in various contexts. For the start point of this mathematical journey, we must quote as our motivation the seminal work by Teixeira \cite{ET}, which brings a spotlight on the universal moduli continuity  for solutions of fully nonlinear elliptical PDEs of the form
$$
F(D^{2}u,x)=f(x) \quad \mbox{in} \quad \Omega.
$$
The optimal regularity achieved in such a manuscript is grounded in the integrability properties of the source term $f$ based on different scenarios. In this context, interior regularity, such as optimal $C_{\text{loc}}^{0, \alpha}$, $C_{\text{loc}}^{0, \text{Log-Lip}}$, $C_{\text{loc}}^{1,\alpha}$ and $C_{\text{loc}}^{1, \text{Log-Lip}}$ are addressed. We may summarize such results in the following table of moduli of continuity:

\begin{table}[h]
\centering
\begin{tabular}{|c|c|c|c}
\cline{1-3}
 ${\bf f \in L^{p}(B_1)}$ & {\bf Optimal regularity}  & {\bf Assumption on $F$}\\\cline{1-3}
 $n-\varepsilon_0< p < n$ & $C_{loc}^{0, \varsigma}(B_1)$ & \text{Uniformly elliptic}\\\cline{1-3}
 $p=n$ & $C_{loc}^{0,\textrm{Log-Lip}}(B_1)$ & \text{Uniformly elliptic} \\\cline{1-3}
 $n<p <\infty$& $C_{loc}^{1, \zeta}(B_1)$ & \text{Uniformly elliptic} \\\cline{1-3}
$\text{BMO} \supset L^{\infty}$ & $C_{loc}^{1, \textrm{Log-Lip}}(B_1)$ & \text{Uniformly elliptic and convex/concave} \\\cline{1-3}
\end{tabular}
\end{table}

Teixeira's results must be understood, to a certain extent, as an extension of Caffarelli's trailblazing work in \cite{Caff1} (see also \cite{CC} for an essay on these regularity issues).  We also recommend that readers refer to Da Silva-Teixeira's work \cite{daSilTei19} for the parabolic counterpart of these results. Furthermore, we also highlight that similar borderline regularity results to the ones in \cite{ET} were established by Daskalopoulos \textit{et al} in \cite{DKM14}, specifically in the context of Lorentz and Morrey spaces.

A few years later, in \cite{DaSilNorn21}, Da Silva and Nornberg developed a full regularity study along similar lines of \cite{Norn19} and \cite{ET}. Specifically, they considered a class of fully nonlinear elliptic operators admitting general Hamiltonian terms with unbounded ingredients in the following configuration:
$$
\mathcal{G}(D^{2}u,Du,x)=F(D^{2}u,x)+b(x)|Du(x)|+\mu(x)|Du(x)|^{m}= f(x) \quad \mbox{in} \quad \Omega,
$$
where $b\in L^{\varrho}(\Omega)$, $\mu\in L^{q}(\Omega)$ for $\varrho, q \in(n,\infty]$, and $m\in(0,2]$ with $m\neq 1$. One of the interesting aspect of this work lies in the dependence of the governing operator on a term of order one, specifically on the gradient term $Du$, which can enjoy a super-linear and sub-quadratic growth. Not only does this result in enhanced generality of the class of operators compared to \cite{ET}, but they also address Schauder-type estimates for such operators.

Recently, Amaral and Dos Prazeres in \cite{AP}, proved optimal regularity for fully nonlinear elliptic models under Dirichlet boundary conditions as follow
$$
\left\{
\begin{array}{rclcl}
	F(D^2u,Du,x) &=& f(x)& \mbox{in} & \Omega \\
	u(x)&=& \varphi(x) &\mbox{on}& \partial \Omega.
\end{array}
\right.
$$
In this context, the governing operator also depends on first-order terms of the solutions. Hence, unlike \cite{DaSilNorn21}, the oscillation of the operator in quest also depends on such first-order terms of the operator. They address a moduli of continuity's classification scheme as the one in \cite{ET}. Additionally, $C^{2,\alpha}$ type estimates were also obtained for this class of problems. We also refer to Lian-Zhang's work \cite{LiZhang20} for boundary point-wise $C^{1,\alpha}$ and $C^{2,\alpha}$ regularity for viscosity solutions of certain classes of fully nonlinear elliptic equations.

Turning back to the regularity results for general boundary data, we must mention that fully nonlinear elliptic problems with oblique boundary conditions have been extensively studied in the last few decades due to intrinsic connection with problems from nonlinear geometric PDEs in geometric analysis, stochastic control theory and homogenization processes just to mention a few (see \cite{Aris03},  \cite{CK23} and \cite{Lopez13} for related works). A pivotal concern in this type of problem revolves around the existence and uniqueness of solutions (in the viscosity sense). In this context, we must quote the Lieberman-Trudinger's pioneering work \cite{LieTru86} dating back 1986, where the authors study fully nonlinear second order uniformly elliptic equations with nonlinear oblique boundary conditions as follows
$$
\left\{
\begin{array}{rclcl}
F(x, u, Du, D^2u) &=& 0& \mbox{in} &   \Omega \\
\mathbf{G}(x, u, Du)&=& 0 &\mbox{on}& \partial \Omega,
\end{array}
\right.
$$
In such a context, under appropriate ``natural conditions'' on the nonlinearities, classical and H\"{o}lder estimates for second derivatives are obtained. Furthermore, an existence result was addressed. Such a work is a natural extension of several earlier works, such as Lieberman's work \cite{Lieb84} and \cite{Lieb86} on quasilinear equations and Trudinger's work \cite{Trud83} and \cite{Trud84} on fully nonlinear models with Dirichlet boundary conditions.

Subsequently, in 1991, Ishii, in \cite{Ishii91}, established, under certain assumptions, that the problem
$$
\left\{
\begin{array}{rclcl}
F(D^2u,Du,u,x) &=& 0& \mbox{in} &   \Omega \\
\mathcal{B}(Du,u,x)&=& 0 &\mbox{on}& \partial \Omega,
\end{array}
\right.
$$
has a unique viscosity solution, where the oblique condition in $\mathcal{B}$ is expressed in the condition
$$
D_{\overrightarrow{p}}\mathcal{B}(\overrightarrow{p},r,x)\cdot \overrightarrow{\textbf{n}}>0 \quad \text{ for} \quad (\overrightarrow{p},r,x)\in \mathbb{R}^{n}\times \mathbb{R} \times \partial \Omega.
$$
In addition to this existence/uniqueness result, a Comparison Principle was also established for solutions to such class of problems (see \cite[Theorems 2.1, 2.3 and 2.4]{Ishii91} for specific results).

Afterward, in 1995, Safonov, in \cite{Saf95}, established Schauder estimates for the following fully nonlinear Bellman problem with oblique boundary conditions
$$
\left\{
\begin{array}{rclcl}
	\displaystyle\sup_{ \iota  \in \mathcal{A}} \left\{\text{Tr}(A^{\iota}(x)D^{2}u)+\overrightarrow{b}^{\iota}(x)\cdot Du+ c^{\iota}(x)u-f^{\iota}(x)\right\} &=& 0& \mbox{in} & \Omega \\
	\beta(x)\cdot Du(x)+\gamma(x)u(x)&=& g(x) &\mbox{on}& \partial \Omega.
\end{array}
\right.
$$
Moreover, \textit{a priori} estimates were also addressed. It is noteworthy that Safonov's technique involved combining interior and boundary estimates to achieve the desired classical results.

Almost three decades after  Caffarelli's pioneering work \cite{Caff1}, in 2006, Milakis and Silvestre, in \cite{MilSil06}, developed up to the boundary $C^{0, \alpha}$, $C^{1,\alpha}$ and $C^{2,\alpha}$ estimates for solutions for fully nonlinear elliptic equations with constant coefficients under Neumann boundary conditions as follows
$$
\left\{
\begin{array}{rclcl}
	F(D^2u) &=& f(x)& \mbox{in} & \Omega \\
	\frac{\partial u}{\partial \overrightarrow{\textbf{n}}}(x)&=& g(x) &\mbox{on}& \partial \Omega.
\end{array}
\right.
$$
In this context, an extension of the A.B.P. estimate obtained for the fundamental class $\overline{\mathcal{S}}$ (see Definition \ref{FundClass}, or Chapter 3 of the Caffarelli-Cabr\'{e}'s book \cite{CC}) with Neumann boundary condition was also obtained. In such a scenario, the authors' strategy is to treat the Neumann condition as an integral part of the PDE under consideration. With this insight, their results must be understood as the counterpart to the corresponding interior regularity estimates addressed by Caffarelli's seminal work \cite{Caff1}.

More than one decade after the last developments on fully nonlinear elliptic models with Neumann boundary conditions, Li and Zhang, in 2018, in \cite{LiZhang}, addressed, along the same lines as Milakis and Silvestre, $C^{0, \alpha}$, $C^{1,\alpha} $ and $C^{2,\alpha}$  estimates in a broader context of equations with constant coefficients and non-homogeneous source terms under oblique boundary conditions
$$
\left\{
\begin{array}{rclcl}
	F(D^2u) &=& f(x)& \mbox{in} & \Omega \\
	\beta(x) \cdot Du(x) + \gamma(x) u(x)&=& g(x) &\mbox{on}& \partial \Omega.
\end{array}
\right.
$$
Furthermore, the A.B.P. estimate obtained by Milakis and Silvestre in \cite[Proposition 3.1]{MilSil06} was generalized for problems with a tangential oblique derivative.

Finally, in 2023, Bessa \textit{et al.}, in the work \cite{Bessa}, obtained, within the scope of the regularity theory for $L^{p}$-viscosity solutions, $W^{2,p}$ estimates for elliptic models of the form
$$
\left\{
\begin{array}{rclcl}
	F(D^2u,Du,u,x) &=& f(x)& \mbox{in} & \Omega \\
	\beta(x) \cdot Du(x) + \gamma(x) u(x)&=& g(x) &\mbox{on}& \partial \Omega,
\end{array}
\right.
$$
on a weaker convexity assumption with respect to the governing operator (cf. \cite{BH20} for similar results). In addition, $p-\textrm{BMO}$ type estimates for the Hessian of solutions to such a class of problems when $\gamma\equiv 0$ are also carried out. Last but not least, a study on $W^{2, p}$ estimates for the corresponding obstacle problem with oblique boundary conditions was addressed (cf. \cite{BHO22} for related results).

Despite the extensive literature on fully nonlinear models with Dirichlet and Neumann boundary conditions, quantitative/qualitative properties of solutions for models with general boundary conditions like \eqref{1} are far less investigated (cf. \cite{Bessa}, \cite{BH20} and \cite{LiZhang} as examples of such considerations). This has been our main impetus for the research presented in the current manuscript.

According to our knowledge, up to date, there has been no investigation into the moduli of continuity for viscosity solutions to fully nonlinear elliptic equations with oblique boundary data like \eqref{1}. Therefore, in this manuscript, we will focus on studying both the optimal lower and higher regularity estimates for such solutions. Furthermore, in some particular configurations of the nonlinearity, we can obtain an explicit and universal regularity exponent for the gradient of solutions.

In conclusion, we believe that our results can be useful to the study of a class of homogenization problems. More precisely, recently Choi and Kim in \cite{CK23} considered the family of bounded solutions $(u_{\varepsilon})_{\varepsilon>0}$ for the following problem:
\begin{equation*}\label{Equation Pe} \tag{$\pe$}
	\left\{
	\begin{array}{rclcl}
		F\left(D^2u_{\varepsilon}, \dfrac{x}{\varepsilon}\right) &=& 0 & \mbox{in} & \Pi,\\
		\partial_{\nu} u_{\varepsilon}(x)&=& \mathrm{G}\left(Du_{\varepsilon},\dfrac{x}{\varepsilon}\right) & \mbox{on} & \mathrm{H}_{-1}\\
		u_{\varepsilon}(x)&=&h(x) & \mbox{on} & \mathrm{H}_{0},
	\end{array}
	\right.
\end{equation*}
where
$$
\left\{
\begin{array}{rcl}
  \Pi & \defeq & \{x\in \mathbb{R}^{n};\,-1<(x-\tau)\cdot \nu<0\} \\
  \mathrm{H}_{-1} & \defeq & \{x\in \mathbb{R}^{n};\,(x-\tau)\cdot \nu=-1\} \\
  \mathrm{H}_{0} & \defeq & \{x\in \mathbb{R}^{n};\,(x-\tau)\cdot \nu=0\}
\end{array}
\right.
$$
for each $\tau\in\mathbb{R}^{n}$ and $\nu\in\mathbb{S}^{n-1}$, where $F$ is a uniformly elliptic and second-order operator, and $\mathrm{G}$ and $h$ are suitable given functions. Under certain conditions stated in \cite[Theorem 1.1]{CK23}, it was proved that the family $(u_{\varepsilon})_{\varepsilon>0}$ converges uniformly to the unique solution $\overline{u}$ of the following oblique boundary problem:
\begin{equation}\label{limitprofile}
	\left\{
	\begin{array}{rclcl}
		\overline{\mathbf{F}}(D^2\overline{u}) &=& 0 & \mbox{in} & \Pi,\\
		\partial_{\nu} \overline{u}(x)&=& g(\nu,\mathrm{D}_{\mathrm{T}}\overline{u}) & \mbox{on} & \mathrm{H}_{-1}\\
		\overline{u}(x)&=&h(x) & \mbox{on} & \mathrm{H}_{0}.
	\end{array}
	\right.
\end{equation}
where $\mathrm{D}_{\mathrm{T}}\overline{u}$ denotes the tangential derivative of $\overline{u}$ along the direction $\nu^{\perp}$, and $g=g(\overrightarrow{p}, \overrightarrow{q})$ is Lipschitz continuous in the variable $\overrightarrow{q}$, and if $\overline{\mathbf{F}}$ is rotation-invariant, then $g$ is $\alpha$-H\"{o}lder continuous over irrational directions $\nu$ for the exponent $\alpha=\frac{1}{5n}$.

Finally, since the associated oblique boundary problem in the limiting problem \eqref{limitprofile} enjoys $C^{1,\alpha_0}$ a priori estimates, we believe that, with certain adjustments, we can apply our strategy to obtain improved regularity for the limiting profile $\overline{u}$. We intend to revisit this topic in a forthcoming work.

\section{Assumptions and auxiliary results}\label{Sec02}

\hspace{0.5 cm} In this section, we introduce the structural conditions on which the results of this paper will rely, as well as useful notation and we collect some well-known facts. Throughout this manuscript we will be assuming the following structural conditions:

\begin{enumerate}
\item [\bf(A1)]\label{HypA1} ({\bf Structural conditions}) We assume that $F \in C^0(\text{Sym}(n), \Omega)$. Moreover, there are constants $0 < \lambda \le \Lambda$ such that
\begin{equation*}
\mathscr{P}^{-}_{\lambda,\Lambda}(\mathrm{M}-\mathrm{N})  \le F(\mathrm{M},x)-F(\mathrm{N},x) \le \mathscr{P}^{+}_{\lambda, \Lambda}(\mathrm{M}-\mathrm{N})
\end{equation*}
for any $x \in \Omega$ and $\mathrm{M}, \mathrm{N} \in \textrm{Sym}(n)$.
where
\begin{equation*}
\mathscr{P}^{+}_{\lambda,\Lambda}(\mathrm{X}) \defeq  \Lambda \sum_{e_i >0} e_i +\lambda \sum_{e_i <0} e_i \quad \text{and} \quad \mathscr{P}^{-}_{\lambda,\Lambda}(\mathrm{X}) \defeq \Lambda \sum_{e_i <0} e_i + \lambda \sum_{e_i >  0} e_i,
\end{equation*}
are the \textit{Pucci's extremal operators} and $e_i = e_i(\mathrm{X})$ ($1\leq i\leq n$) denote the eigenvalues of $\mathrm{X}$. For normalization reasons, we shall always assume: $ F(\mathcal{O}_{n\times n}, x) = 0 \quad \text{for all} \,\,\, x \in \Omega,$ which is not restrictive, because one can reduce the problem in order to check it.

\item[\bf(A2)] ({\bf Regularity of the data}) \, The data satisfy $f \in C^0(\Omega)\cap L^{p}(\Omega)$ for $ \frac{n}{2}\leq p<\infty$, $g, \gamma \in C^0(\partial \Omega)$ with $\gamma \le 0$ and $\beta \in C^0( \partial \Omega; \mathbb{R}^n)$ with $\|\beta\|_{L^{\infty}(\partial \Omega )} \le 1$ and there exists a positive constant $\mu_0$ such that $\beta\cdot \overrightarrow{\textbf{n}}\ge \mu_0$, where $\overrightarrow{\textbf{n}}$ is the outward normal vector of $\Omega$.
\end{enumerate}

\begin{remark}
From now on, we observe that an operator fulfilling $\text{\bf (A1)}$ will be referred to as a {\bf $(\lambda, \Lambda)$-elliptic operator}.
\end{remark}

Now, we recall the definition of viscosity solutions of \eqref{1}.

 \begin{definition}[{\bf $C^2$-viscosity solution}]

Let $F$ be continuous in all variables, and we assume $f \in C^0(\Omega \cup \Gamma)$, where $\Gamma \subset \partial \Omega$ (is a relatively open set). A function $u \in C^0(\Omega \cup \Gamma)$ is said to be a $C^{2}$-viscosity solution of \eqref{1} if the following conditions hold:
\begin{enumerate}
\item[a)] for all $\phi \in C^2(\Omega \cup \Gamma)$ touching $u$ by above at $x_0 \in \Omega \cup \Gamma$, then
$$
	F(D^2 \phi(x_0), x_0) \ge f(x_0)
$$
when $x_0 \in \Omega$ and
$$
\displaystyle \beta(x_0) \cdot D \phi(x_0) + \gamma(x_0) u(x_0) \ge g(x_0)
$$
 when $x_0 \in \Gamma$.
\item[b)] for all $\phi \in C^2(\Omega \cup \Gamma)$ touching $u$ by below at $x_0 \in \Omega \cup \Gamma$, then
$$
	F(D^2 \phi(x_0), x_0) \le f(x_0)
$$
when $x_0 \in \Omega$ and
$$
\displaystyle \beta(x_0) \cdot D \phi(x_0) + \gamma(x_0) u(x_0) \le g(x_0)
$$
 when $x_0 \in \Gamma$.
\end{enumerate}

 \end{definition}

Throughout this work, we assume that $\vec{0} \in \partial \Omega$, and we denote
 $$
 \Omega^{+}_R = \Omega \cap \mathrm{B}_R \quad \text{and} \quad \Omega^{0}_R = \partial \Omega \cap \mathrm{B}_R
 $$
  where
 $\mathrm{B}_R = \mathrm{B}_R(0) \subset \mathbb{R}^n$ is the ball centred at $\vec{0}$ with radius $R>0$. We also denote
 $$
 	\mathrm{B}^+_1 \defeq \{x=(x^{\prime},x_n) \in \mathbb{R}^n \, : \, \|x\| <1 \,\,\, \textrm{and} \,\,\, x_n >0\} \quad \textrm{and} \quad \mathrm{T}_1 \defeq \{x=(x^{\prime},0) \in \mathbb{R}^n \, : \, |x| <1\}.
 $$
Similarly, we can define $\mathrm{B}^+_r$ and $\mathrm{T}_r$ for a radius $r>0$.

Now, we define the following function, which measures the oscillation of the coefficients of $F$ around $x_0$:

\begin{eqnarray*}
\Phi_{F}(x,x_{0}) \defeq \sup_{\mathrm{M}\in \mbox{Sym}(n)\setminus\{0\}}\frac{|F(\mathrm{M},x)-F(\mathrm{M},x_{0})|}{\|\mathrm{M}\|}, \ x\in \mathrm{B}^{+}_{1}
\end{eqnarray*}

Moreover, when $x_{0}=\vec{0}$, we will use the notation $\Phi_{F}(x)=\Phi_{F}(x,0)$ for simplicity.

Next, we present the following stability result (see for instance \cite[Theorem 3.8]{CCKS} for a proof).

\begin{lemma}[{\bf Stability Lemma}]\label{Stability}
For $k \in \mathbb{N}$ let $\Omega_k \subset \Omega_{k+1}$ be an increasing sequence of domains and $\displaystyle \Omega \defeq \bigcup_{k=1}^{\infty} \Omega_k$. Let $F, F_k$ be $(\lambda, \Lambda)-$elliptic operators. Assume $f \in L^{p}(\Omega)$, $f_k \in L^p(\Omega_k)$ and that $u_k \in C^0(\Omega_k)$ are $C^{2}-$viscosity sub-solutions (resp. super-solutions) of
$$
F_k(D^2 u_k,x)=f_k(x) \quad \textrm{in} \quad \Omega_k.
$$
Suppose that $u_k \to u_{\infty}$ locally uniformly in $\Omega$ and that for $\mathrm{B}_r(x_0) \subset \Omega$ and $\varphi \in C^{2}(\mathrm{B}_r(x_0))$ we have
\begin{equation*}
\|(\hat{g}-\hat{g}_k)^+\|_{L^p(\mathrm{B}_r(x_0))} \to 0 \quad \left(\textrm{resp.} \,\,\, \|(\hat{g}-\hat{g}_k)^-\|_{L^p(\mathrm{B}_r(x_0))} \to 0 \right),
\end{equation*}
where $\hat{g}(x) \defeq F(D^2 \varphi,x)-f(x)$ and $\hat{g}_k(x) =  F_k(D^2 \varphi,x)-f_k(x)$.  Then $u$ is an $C^{2}-$viscosity sub-solution (resp. super-solution) of
$$
F(D^2 u,x)=f(x) \quad \textrm{in} \quad \Omega.
$$
\end{lemma}

We will also need the following terminology from the fundamental class of solutions (for more details, see \cite{CC}).

\begin{definition}\label{FundClass}
	We define the class $\overline{\mathcal{S}}\left(\lambda,\Lambda, f\right)$ and $\underline{\mathcal{S}}\left(\lambda,\Lambda, f\right)$ to be the set of all continuous functions $u: \Omega \to \mathbb{R}$ satisfying
$$
 \displaystyle \mathscr{P}^{+}_{\lambda,\Lambda}(D^{2}u) \ge f(x) \quad \text{in} \quad \Omega \quad (\text{resp.}\quad  \mathscr{P}^{-}_{\lambda,\Lambda}(D^{2}u) \le f(x))
$$
in the viscosity sense. Thus, we define
	$$
	\mathcal{S}\left(\lambda, \Lambda, f\right) \defeq  \overline{\mathcal{S}}\left(\lambda, \Lambda,f\right) \cap \underline{\mathcal{S}}\left(\lambda, \Lambda, f\right)\,\,\text{and}\,\,
	\mathcal{S}^{\star}\left(\lambda, \Lambda, f\right) \defeq  \overline{\mathcal{S}}\left(\lambda, \Lambda,|f|\right) \cap \underline{\mathcal{S}}\left(\lambda, \Lambda, -|f|\right).
	$$
\end{definition}

We also present a Maximum Principle that ensures universal boundedness (see \cite{BH20} for more details).

\begin{lemma}[ {\bf A.B.P. Maximum Principle}]\label{ABP} Let $\Omega\subset \mathrm{B}_{1}$ and $u\in C^{0}(\overline{\Omega})$ be satisfying
\begin{equation*}
\left\{
\begin{array}{rclcl}
u\in \mathcal{S}(\lambda,\Lambda,f) &\mbox{in}& \Omega \\
\beta\cdot Du+\gamma u=g &\mbox{on}& \Gamma.
\end{array}
\right.
\end{equation*}
Suppose there is $\varsigma\in \partial \mathrm{B}_{1}$ such that $\beta\cdot\varsigma\geq \mu_0$ and $\gamma\le 0$ on $\Gamma$. Then, there exists  $\varepsilon_{0}=\varepsilon_{0}(n,\lambda,\Lambda,\mu_0)\in \left(0,\frac{n}{2}\right)$ such that
\begin{eqnarray*}
\|u\|_{L^{\infty}(\Omega)}\leq \|u\|_{L^{\infty}(\partial \Omega\setminus \Gamma)}+\mathrm{C} (n, \lambda, \Lambda,b, \mu_0)(\| g\|_{L^{\infty}(\Gamma)}+\| f\|_{L^{n-\varepsilon_{0 }}(\Omega)}).
\end{eqnarray*}
\end{lemma}

Next, we will remember the definition of some functionals spaces.

\begin{definition}
Let $\alpha\in (0,1]$. We say that a function $u\in C^{0}(\overline{\Omega})$ is $\alpha$-H\"{o}lder continuous in $\overline{\Omega}$, if
\begin{eqnarray*}
[u]_{0,\alpha,\overline{\Omega}}=\sup_{x, y \in \overline{\Omega} \atop{x \neq y}}\frac{|u(x)-u(y)|}{|x-y|^{\alpha}}<\infty.
\end{eqnarray*}
The set of functions $\alpha$-H\"{o}lder continuous in $\overline{\Omega}$ is denoted by $C^{0,\alpha}(\overline{\Omega})$. Furthermore, $C^{0,\alpha}(\overline{\Omega})$ is a Banach space equipped with the following norm
\begin{eqnarray*}
\|u\|_{C^{0,\alpha}(\overline{\Omega})}=:\|u\|_{L^{\infty}(\overline{\Omega})}+[u]_{0,\alpha,\overline{\Omega}}.
\end{eqnarray*}
\end{definition}

We can also define the concept of H\"{o}lder continuity in the $L^{p}$-sense.

\begin{definition}
Let $\alpha\in(0,1)$ and $p \in [1, \infty)$. We say that $u:\mathrm{B}^{+}_{1}\cup\mathrm{T}_{1}\longrightarrow\mathbb{R}$ is $\alpha$-H\"{o}lder continuous at the origin in $L^{p}$-sense if,
\begin{eqnarray*}
[u]_{C^{0,\alpha}_{p}(0)} \defeq \sup_{0<r<1} \frac{1}{r^{\alpha}}\left(\intav{\mathrm{B}^{+}_{r}}|u(x)-(u)_{r}|^{p}dx\right)^{\frac{1}{p}}<\infty,
\end{eqnarray*}
where
\begin{eqnarray*}
(u)_{r} \defeq \intav{\mathrm{B}^{+}_{r}}u(x)dx
\end{eqnarray*}
\end{definition}

More generally, we can also define higher order H\"{o}lder spaces.

\begin{definition}
Let $\alpha\in (0,1]$ and $k$ an positive integer. We define the H\"{o}lder space $C^{k,\alpha}(\overline{\Omega})$ of the functions $C^{k}(\overline{\Omega})$ such that its partial derivatives $D^{\kappa}u$ belongs to $C^{0,\alpha}(\overline{\Omega})$ for any multi-index of order $|\kappa|=k$. In this case, we can equip $C^{k,\alpha}(\overline{\Omega})$ with the following norm:
\begin{eqnarray*}
\|u\|_{C^{k,\alpha}(\overline{\Omega})}=:\|u\|_{C^{k}(\overline{\Omega})}+[D^{k}u]_{\alpha,\overline{\Omega}},
\end{eqnarray*}
where
\begin{eqnarray*}
[D^{k}u]_{\alpha,\overline{\Omega}}=:\sum_{|\kappa|=k}[D^{\kappa} u]_{0,\alpha,\overline{\Omega}}
\end{eqnarray*}
which makes it a Banach space.
\end{definition}

\begin{remark}
In particular, we can define in $C^{2,\alpha}(\overline{\mathrm{B}^{+}_{r}})$, the ``adimensional norm''
\begin{eqnarray*}
\|u\|^{*}_{C^{2,\alpha}(\overline{B^+_{r}})}= \|u\|_{L^{\infty}(\overline{\mathrm{B}^{+}_{r}})}+r\|Du\|_{L^{\infty}(\overline{\mathrm{B}^{+}_{r}})}+r^{2}\|D^{2}u\|_{L^{\infty}(\overline{\mathrm{B}^{+}_{r}})}+r^{2+\alpha}\sup_{x, y \in \overline{\mathrm{B}^{+}_{r}} \atop{x \neq y}}\frac{\|D^{2}u(x)-D^{2}u(y)\|}{|x-y|^{\alpha}},
\end{eqnarray*}
which will play an essential role in proving the Schauder-type estimates in Section \ref{SecSchauder}.
\end{remark}

Now, we define also the Morrey spaces (see \cite{DKM14}), which we will revisit them in the next sections.

\begin{definition}[{\bf Morrey spaces}]
Let $\mathrm{E} \subset \mathbb{R}^{n}$ be a bounded open set, and let $1\leq p<\infty$ and $0\leq\theta\leq n$. By $L^{p,\theta}(\mathrm{E})$, we denote the \textit{Morrey space} of functions $h\in L^{p}_{\text{loc}}(\mathrm{E})$ such that
\begin{eqnarray*}
\|h\|_{L^{p,\theta}(\mathrm{E})}=\sup_{x_{0}\in \mathrm{E} \atop \ 0<r\leq \diam(\mathrm{E})}\left(r^{\theta -n}\int_{\mathrm{E}(x_{0},r)}|h(y)|^{p}dy\right)^{\frac{1}{p}}<\infty,
\end{eqnarray*}
where $\mathrm{E}(x_{0},r)=\mathrm{E}\cap\mathrm{B}(x_{0},r)$.
\end{definition}

It is not difficult to verify that $L^{p,\theta}(\mathrm{E})\subset L^{p}(\mathrm{E})$. Furthermore, we have the following inequality
\begin{eqnarray*}
\|h\|_{L^{p}(\mathrm{E})}\leq \left(\diam(\mathrm{E})\right)^{\frac{n-\theta}{p}}\|h\|_{L^{p,\theta}(\mathrm{E})}, \quad \forall h\in L^{p,\theta}(\mathrm{E}).
\end{eqnarray*}

We also need the definition of \textbf{Bounded Mean Oscillation} functions, which will be useful in Section \ref{Section06}. Specifically,

\begin{definition}\label{DefBMO}
We recall that a function $f\in L^{1}_{\text{loc}}(\Omega)$ is said to be \textbf{$p$-bounded mean oscillation} in $\Omega$ for $p\in[1,\infty)$, i.e., $f\in p-\textrm{BMO}(\Omega)$ if
\begin{equation*}\label{p-BMOnorm}
\|f\|_{p-\textrm{BMO}(\Omega)} \defeq \sup_{x_{0}\in \Omega, \rho>0} \left(\intav{\mathrm{B}{\rho}(x_0) \cap \Omega} |f(x) - (f)_{x_0, \rho}|^p ,dx\right)^{\frac{1}{p}} <\infty,
\end{equation*}
where, for each $x_{0}\in \Omega$ and $\rho>0$, we have that
$$
(f)_{x_0, \rho} \defeq \intav{\mathrm{B}_{\rho}(x_0) \cap \Omega} f(x) dx
$$
and for simplicity, we use the abbreviated notation $(f)_{\rho}$ when $x_{0}=0$.
\end{definition}

From the $p$-BMO spaces, we can define a subclass of functions that we will need later.
\begin{definition}
We recall that a function $f\in \textrm{BMO}$ is said to be \textbf{vanishing mean oscillation} in $\Omega$, i.e., $f\in\textrm{VMO}(\Omega)$ if
\begin{equation*}
 \lim_{r\to 0^{+}}\sup_{|\mathrm{B}|\leq r} \left(\intav{\mathrm{B} \cap \Omega} |f(x) - (f)_{\mathrm{B}}| dx\right) =0,
\end{equation*}
where the supremum is taken over all balls $\mathrm{B}\subset \mathbb{R}^{n}$ of the measure at most $r$ and for each ball $\mathrm{B}\subset\mathbb{R}^{n}$, and
$$
(f)_{\mathrm{B}} \defeq \intav{\mathrm{B} \cap \Omega} f(x) dx.
$$
\end{definition}

\subsection*{Approximation devices for viscosity solutions}

In this section, we will present a key ingredient in accessing the sharp regularity estimates available for our model PDEs. For this purpose, we need to prove some Approximation Lemmas for viscosity solutions.

\begin{lemma}[{\bf Approximation Lemma I}]\label{approx}

Let $u$ be a viscosity solution of $\eqref{1}$ with $u=\varphi$ on $\partial \mathrm{B}^{+}_{1}\setminus\mathrm{T}_{1}$ for $\varphi\in C^{0}(\partial \mathrm{B}^{+}_{1}\setminus\mathrm{T}_{1})$ such that $\|\varphi\|_{L^{\infty}(\partial \mathrm{B}^{+}_{1}\setminus\mathrm{T}_{1})}\leq \mathrm{C}_{1}$ for a constant  $\mathrm{C}_{1}>0$.  Suppose $f \in L^p(B_1^+)$ and $g\in C^{0,\alpha^{\prime}}(\mathrm{T}_{1})$, for some fixed $\alpha^{\prime} \in (0, 1]$ with $\|g\|_{C^{0,\alpha^{\prime}}(\mathrm{T}_{1})}\leq \mathrm{C}_{2}$ for $\mathrm{C}_{2}>0$ and $p\in[n-\varepsilon_{0},\infty)$. Given $\delta>0$, there exists $\eta>0$ depending only on $\delta$, $n$, $\lambda$, $\Lambda$, $\mathrm{C}_{1}$, $\mathrm{C}_{2}$ and $p$ such that if
\begin{eqnarray*}
\max\left[\left(\intav{\mathrm{B}^{+}_{1}}|\Phi_{F}(x)|^{p}dx\right)^{\frac{1}{p}}, \|f\|_{L^{p}(\mathrm{B}^{+}_{1})}\right]\leq \eta
\end{eqnarray*}
then $\mathfrak{h}$, the viscosity solution of
\begin{eqnarray*}
	\left\{
	\begin{array}{rclcl}
		F(D^{2}\mathfrak{h},0)=0 &\mbox{in}& \mathrm{B}^{+}_{\frac{7}{8}} \\
		\beta(x)\cdot D\mathfrak{h}(x)+\gamma(x) \mathfrak{h}(x)=g(x) &\mbox{on}& \mathrm{T}_{\frac{7}{8}}\\
		\mathfrak{h}(x)= u(x) &\mbox{on}& \partial \mathrm{B}^{+}_{\frac{7}{8}}\setminus \mathrm{T}_{\frac{7}{8}}
	\end{array}
	\right.
\end{eqnarray*}
satisfies
\begin{eqnarray*}
\sup_{\mathrm{B}^{+}_{\frac{7}{8}}}|u-\mathfrak{h}|\leq \delta.
\end{eqnarray*}
\end{lemma}

\begin{proof}
The proof will be carried out by a \textit{Reductio ad Absurdum} argument. Specifically, suppose, for the sake of contradiction, that there exists a $\delta_{0}>0$ such that the thesis of the Lemma does not hold. Thus, we could find sequences of functions $(F_{k})_{k\in\mathbb{N}}$, $(f_{k})_{k\in \mathbb{N}}$, $(u_{k})_{k\in\mathbb{N}}$, $(\mathfrak{h}_{k})_{k\in\mathbb{N}}$ and $(g_{k})_{k\in\mathbb{N}}$ such that $u_{k}$ and $\mathfrak{h}_{k}$ are viscosity solutions of
$$
	\left\{
	\begin{array}{rclcl}
		F_k(D^{2}u_{k},x) & = & f_{k}(x) &\mbox{in}& \mathrm{B}^{+}_{1} \\
		\beta(x)\cdot Du_{k}(x)+\gamma(x) u_{k}(x) & = & g_{k}(x) &\mbox{on}& \mathrm{T}_{1}\\
		u_{k}(x) & = & \varphi_{k}(x)
		&\mbox{on}& \partial \mathrm{B}^{+}_{1}\setminus \mathrm{T}_{1}
	\end{array}
	\right.
$$
and
$$
	\left\{
	\begin{array}{rclcl}
		F_k(D^{2}\mathfrak{h}_{k},0) & = & 0 &\mbox{in}& \mathrm{B}^{+}_{\frac{7}{8}} \\
		\beta(x)\cdot D\mathfrak{h}_{k}(x)+\gamma(x) \mathfrak{h}_{k}(x) & = & g_{k}(x) &\mbox{on}& \mathrm{T}_{\frac{7}{8}}\\
		\mathfrak{h}_{k}(x)& = &  u_{k}(x)
		&\mbox{on}& \partial \mathrm{B}^{+}_{\frac{7}{8}}\setminus \mathrm{T}_{\frac{7}{8}},
	\end{array}
	\right.
$$
where $\|\varphi_{k}\|_{L^{\infty}(\partial \mathrm{B}_{1}^{+}\setminus \mathrm{T}_{1})}\leq \mathrm{C}_{1}$, $\|g_{k}\|_{C^{0,\alpha}(\overline{\mathrm{T}_{1}})}\leq \mathrm{C}_{2}$ and
$$\intav{\mathrm{B}^{+}_{1}} |\Phi_{F_{k}}(x)|^{p}dx\leq \frac{1}{k^{p} } \ \ \ \mbox{and} \ \ \ \int_{\mathrm{B}^{+}_{1}}|f_{k}(x)|^{p}dx\leq \frac{1} {k^{p}},
$$
however,
\begin{eqnarray}\label{4.2}
	\sup_{\mathrm{B}^{+}_{\frac{7}{8}}}|u_{k}-\mathfrak{h}_{k}|>\delta_{0},\forall k\in\mathbb{ N}.
\end{eqnarray}

Now, by the A.B.P. Maximum Principle (Lemma \ref{ABP}), it follows that
\begin{eqnarray}\label{ukestimate}
	\|u_{k}\|_{L^{\infty}(\mathrm{B}^{+}_{1})}&\leq& \|\varphi_{k}\|_{L^{\infty}(\partial \mathrm{B}^{+}_{1}\setminus\mathrm{T}_{1})}+\mathrm{C}(n,\lambda,\Lambda,\mu_{0})\cdot(\|f_{k}\|_{L^{p}(\mathrm{B}^{+}_{1})}+\|g_{k}\|_{L^{\infty}(\mathrm {T}_{1})})\nonumber\\
	&\leq& \mathrm{C}_{1}+\mathrm{C}(n,\lambda,\Lambda,\mu_{0},p,\varepsilon_{0})\cdot(\|f_{k} \|_{L^{p}(\mathrm{B}^{+}_{1})}+\|g_{k}\|_{C^{0,\alpha^{\prime}}(\overline{\mathrm{T}_{1}})})\nonumber\\
	&\leq& \mathrm{C}(n,\varepsilon_{0},p,\lambda,\Lambda,\mu_{0},\mathrm{C}_{1},\mathrm{C}_{2}), \ \forall k\in\mathbb{N},
\end{eqnarray}

On the other hand, by H\"{o}lder estimates from \cite[Theorem 1.1]{LiZhang}, there exists $\hat{\alpha}\in(0,1)$ depending only on $n$,$\lambda$, $\Lambda$ and $\mu_{0}$ such that $u\in C^{0,\hat{\alpha}}\left(\overline{\mathrm{B}^{+}_{\frac{7}{8}}}\right)$, and the following estimate holds
\begin{eqnarray}\label{4.4}
	\|u_{k}\|_{C^{0,\hat{\alpha}}\left(\overline{\mathrm{B}^{+}_{\frac{7}{8}}}\right)} &\leq& \mathrm{C}(n,\lambda,\Lambda,\mu_{0},\|\gamma\|_{L^{\infty}(\mathrm{T}_{1})}) (\|u_{k}\|_{L^{\infty}(\mathrm{B}^{+}_{1})}+\|f_{k}\|_{L^{p}(\mathrm{B}^{+}_{1})}+\|g_{k}\|_{L^{\infty}(\mathrm{T}_{1})} )\nonumber\\
	&\leq&\mathrm{C}(\|u_{k}\|_{L^{\infty}(\mathrm{B}^{+}_{1})}+\|f_{k}\| _{L^{p}(\mathrm{B}^{+}_{1})}+\|g_{k}\|_{L^{\infty}(\mathrm{T}_{1} )}),
\end{eqnarray}
where $\mathrm{C}=\mathrm{C}(n,\lambda,\Lambda,\mu_{0},p,\varepsilon_{0},\|\gamma\|_{L^{\infty} (\mathrm{T}_{1})})$. Thus, from \eqref{ukestimate} and \eqref{4.4} it follows that
\begin{equation*}
\|u_{k}\|_{C^{0,\hat{\alpha}}\left(\overline{\mathrm{B}^{+}_{\frac{7}{8}}}\right)} \leq C(n,\lambda,\Lambda,\mu_{0},p,\varepsilon_{0}, \|\gamma\|_{L^{\infty}(\mathrm{T}_{1})}, \mathrm{C}_{1},\mathrm{C}_{2}), \forall k\in\mathbb{N}.
\end{equation*}

Thus the sequence $(u_{k})_{k\in\mathbb{N}}$ is uniformly bounded in $C^{0,\hat{\alpha}}\left(\overline{\mathrm{B}^{+} _{\frac{7}{8}}}\right)$. Thus, such a sequence is equicontinuous and equibounded. Similarly, we also reach the same conclusion for the sequence  $(g_{k})_{k\in\mathbb{N}}$ and since the sequence of operators $(F_{k})_{k\in\mathbb{N}}$ is $(\lambda,\Lambda)$-elliptic, it follows that the sequence $(F_{k}(\cdot,0))_{k \in \mathbb{N}}$ is equicontinuous and equibounded on compact sets of $\mbox{Sym}(n)$. Thus, from the Ascoli-Arzelà compactness criterium, we obtain subsequences of functions $(u_{k_{j}})_{j\in\mathbb{N}}$ and $(g_{k_{j}})_{ j\in\mathbb{N}}$ and operators $(F_{k_{j}})_{j\in\mathbb{N}}$ such that
$$
u_{k_{j}}\to u_{\infty} \quad  \text{in} \quad L^{\infty}\left(\overline{\mathrm{B}^{+}_{\frac{7}{8}}}\right) \quad \text{and} \quad g_{k_{j} }\to g_{\infty} \quad  \text{in} \quad L^{\infty}(\overline{\mathrm{T}_{1}})
$$
uniformly, and $F_{k_{j}}(\cdot,x)\to F_{\infty}(\cdot,0)$ uniformly on compact sets of $\mbox{Sym}(n)$, where $F_{\infty}$ is a $(\lambda,\Lambda)$-elliptic operator. Furthermore, for every $\varphi\in C^{2}(\mathrm{B}_{r}(x_{0}))$ to $\mathrm{B}_{r}(x_{0})\subset \mathrm{B}^{+}_{\frac {7}{8}}$ we get

\begin{eqnarray*}
|F_{k_{j}}(D^{2}\varphi(x),x)-f_{k_{j}}(x)-F_{\infty}(D^{2}\varphi(x) ,0)|&\leq&|F_{k_{j}}(D^{2}\varphi(x),x)-F_{k_{j}}(D^{2}\varphi(x),0 )|+\\
&+&|f_{k_{j}}(x)|+\\
&+&|F_{k_{j}}(D^{2}\varphi(x),0)-F_{\infty}(D^{2}\varphi(x),0)|\\
&\leq&|D^{2}\varphi(x)||\Phi_{F_{k_{j}}}(x)|+|f_{k_{j}}(x)|+\\
&+&|F_{k_{j}}(D^{2}\varphi(x),0)-F_{\infty}(D^{2}\varphi(x),0)|,
\end{eqnarray*}
where, by the assumptions on $\Phi_{k_{j}}$ and $f_{k_{j}}$, it follows that
\begin{eqnarray*}
\lim_{j\to \infty}\|F_{k_{j}}(D^{2}\varphi(\cdot),\cdot)-f_{k_{j}}(\cdot)-F_{\infty}(D^{2}\varphi(\cdot),0)\|_{L^{p}(\mathrm{B}_{r}(x_{0}))}=0.
\end{eqnarray*}
Therefore, by the Stability Lemma \ref{Stability} we may conclude that $u_{\infty}$ is a viscosity solution of
\begin{eqnarray*}
	\left\{
	\begin{array}{rclcl}
		F_{\infty}(D^{2}u_{\infty},0)=0 &\mbox{in}& \mathrm{B}^{+}_{\frac{7}{8}} \\
		\beta\cdot Du_{\infty}+\gamma u_{\infty}=g_{\infty} &\mbox{on}& \mathrm{T}_{\frac{7}{8}}.
	\end{array}
	\right.
\end{eqnarray*}

Finally, by defining $w_{k{j}}\defeq u_{\infty}-\mathfrak{h}_{k_{j}}$, we obtain in the viscosity sense that
\begin{eqnarray*}
	\left\{
	\begin{array}{rclcl}
		w_{k_{j}}\in S\left(\frac{\lambda}{n},\Lambda,0\right) &\mbox{in}& \mathrm{B}^{+}_\frac{ 7}{8} \\
		\beta\cdot Dw_{k_{j}}+\gamma w_{k_{j}}=g_{\infty}-g_{k_{j}} &\mbox{on}& \mathrm{T}_{\frac{7}{8}}\\
		w_{k_{j}}=u_{\infty}-u_{k_{j}} &\mbox{on}& \partial\mathrm{B}^{+}_{\frac{7}{8}} \setminus \mathrm{T}_{\frac{7}{8}}.
	\end{array}
	\right.
\end{eqnarray*}
Once again, by the A.B.P. Maximum Principle (Lemma \ref{ABP}), we obtain that
\begin{equation*}
	\|w_{k_{j}}\|_{L^{\infty}\left(\mathrm{B}^{+}_{\frac{7}{8}}\right)}\leq \| u_{\infty}-u_{k_{j}}\|_{L^{\infty}\left(\partial\mathrm{B}^{+}_{\frac{7}{8}}\setminus \mathrm{T}_{\frac{7}{8}}\right)}+\|g_{\infty}-g_{k_{j}}\|_{L^{\infty}\left(\mathrm{T}_{\frac{7}{8}}\right)}\to 0 \quad \text{as} \quad j\to \infty.
\end{equation*}

Therefore, $\mathfrak{h}_{k{j}}\to u_{\infty}$ in $\overline{\mathrm{B}^{+}_{\frac{7}{8}}}$ uniformly, which yields a contradiction to the condition \eqref{4.2} for $j$ large enough, thereby completing the proof.
\end{proof}


Next, as in the previous result, we will need a version of the approximation Lemma similar to \ref{approx}. However, in this case, due to the presence of the source term belonging to $p-\textrm{BMO}$, we must assume that the semi-norm $p-\textrm{BMO}$ is small. In summary, we have the following Lemma.

\begin{lemma}[{\bf Approximation Lemma II - Frozen coefficients case}]\label{lemma4.4.3}
Let $u$ be a normalized viscosity solution of
\begin{eqnarray*}
	\left\{
	\begin{array}{rclcl}
		F(D^{2}u)= f(x) &\mbox{in}&   \mathrm{B}^{+}_{1} \\
		\beta(x)\cdot Du(x)+\gamma(x) u(x)=g(x)  &\mbox{on}&  \mathrm{T}_{1},
	\end{array}
	\right.
\end{eqnarray*}
where $u=\varphi$ on $\partial\mathrm{B}^{+}_{1}\setminus \mathrm{T}_{1}$ and $\varphi\in C^{0}(\partial\mathrm{B}^{+}_{1}\setminus \mathrm{T}_{1})$ such that $\|\varphi\ |_{L^{\infty}(\partial\mathrm{B}^{+}_{1}\setminus \mathrm{T}_{1})}\leq \mathrm{C}_{1}$ for some positive constant $\mathrm{C}_{1}$ and $g\in C^{0,\alpha}(\overline{\mathrm{T}_{1}})$ for some $\alpha\ \in(0,1]$ such that $\|g\|_{C^{0,\alpha}(\overline{\mathrm{T}_{1}})}\leq \mathrm{C}_ {2}$ where $\mathrm{C}_{2}>0$ and $f \in p-\textrm{BMO}(\mathrm{B}^{+}_{1})$ for $p\in[n-\varepsilon_{0},\infty)$ ($\varepsilon_{0}$ is the Escauriazia's constant). Thus, given $\delta>0$, there exists $\eta>0$ depending only on $n$, $\lambda$, $\Lambda$, $p$, $\delta$ such that if
$$
\|f\|_{p-\textrm{BMO}(\mathrm{B}^{+}_{1})}\leq \eta,
$$
then, $\mathfrak{h}$ the viscosity solution of
\begin{eqnarray*}
\left\{
\begin{array}{rclcl}
F(D^{2} \mathfrak{h}) & = & (f)_{1} & \mbox{in} & \mathrm{B}^{+}_{\frac{7}{8}} \\
\beta(x)\cdot D\mathfrak{h}(x) +\gamma(x) \mathfrak{h}(x) &  = & g(x) & \mbox{on} & \mathrm{T}_{\frac{7}{8}}\\
\mathfrak{h}(x) & = & u(x) &\mbox{on}& \partial \mathrm{B}^{+}_{\frac{7}{8}}\setminus \mathrm{T}_{\frac{7}{8}},
\end{array}
\right.
\end{eqnarray*}
then
\begin{eqnarray*}
\sup_{\mathrm{B}^{+}_{\frac{7}{8}}}|u-\mathfrak{h}|\leq \delta.
\end{eqnarray*}
\end{lemma}

\begin{proof}
The proof is closely similar to the one in Lemma \eqref{approx}, but we will present the changes with respect to Lemma \ref{approx} for the reader's convenience. For this, let us assume, for the sake of contradiction, that the thesis of the lemma is false. Then, there exists a positive constant $\delta_{0}$ and sequences of functions  $(F_{k})_{k\in\mathbb{N}}$, $(f_{k})_{k\in \mathbb{N}}$, $(u_{k})_{k\in\mathbb{N}}$, $(\mathfrak{h}_{k})_{k\in\mathbb {N}}$ and $(g_{k})_{k\in\mathbb{N}}$ such that $u_{k}$ and $v_{k}$ satisfy in the viscosity sense
\begin{eqnarray*}
\left\{
\begin{array}{rclcl}
F(D^{2}u_{k}) & = & f_{k}(x) &\mbox{in}& \mathrm{B}^{+}_{1} \\
\beta(x)\cdot Du_{k}(x)+\gamma(x) u_{k}(x) & = & g_{k}(x) &\mbox{on}& \mathrm{T}_{1}\\
u_{k}(x) & = & \varphi_{k}(x) &\mbox{in}& \partial \mathrm{B}^{+}_{1}\setminus \mathrm{T}_{1}
\end{array}
\right.
\end{eqnarray*}
and
\begin{eqnarray*}
\left\{
\begin{array}{rclcl}
F(D^{2}h_{k}) & = & (f_{k})_{1} & \mbox{in} & \mathrm{B}^{+}_{\frac{7}{8}} \\
\beta(x)\cdot Dh_{k}(x)+\gamma(x) h_{k}(x) & = & g_{k}(x) &\mbox{on}& \mathrm{T}_{\frac{7}{8}}\\
h_{k}(x) & = & u_{k}(x) & \mbox{on} & \partial \mathrm{B}^{+}_{\frac{7}{8}}\setminus \mathrm{T}_{\frac{7}{8}},
\end{array}
\right.
\end{eqnarray*}
where $\|\varphi_{k}\|_{L^{\infty}(\partial \mathrm{B}_{1}^{+}\setminus \mathrm{T}_{1})}\leq \mathrm{C}_{1}$, $\|g_{k}\|_{C^{0,\alpha}(\overline{\mathrm{T}_{1}})}\leq \mathrm{C}_{2}$ and $\|f_{k}\|_{p-\textrm{BMO}(\mathrm{B}^{+}_{1})}\leq \frac{1}{k}$, however
\begin{eqnarray}\label{(87)}
\sup_{\mathrm{B}^{+}_{\frac{7}{8}}} |u_{k}-v_{k}|>\delta_{0}, \,\,\forall k\in\mathbb{N}.
\end{eqnarray}

Since we have control over the $p-\textrm{BMO}$ semi-norm, we also control the $L^{p}$-norm of $f$ (cf. \cite{STEIN} and \cite{LLO}), such that
\begin{eqnarray*}
\|f_{k}\|_{L^{p}(\mathrm{B}^{+}_{1})}\leq \mathrm{C}_{p}\|f_{k}\| _{p-\textrm{BMO}(\mathrm{B}^{+}_{1})}\leq \frac{1}{k}\mathrm{C}_{p},
\end{eqnarray*}
for some positive constant $\mathrm{C}_{p}$ that depends only on $p$ (cf. Stein's Book \cite[Chapter IV]{STEIN} and the fact that $p-\textrm{BMO}$ semi-norms are equivalent, see \cite{LLO}). Furthermore, by the A.B.P. estimate (Lemma \ref{ABP}) and \cite[Theorem 1.1]{LiZhang}, we conclude that
\begin{equation*}
\|u_{k}\|_{C^{0,\alpha^{\prime}}\left(\overline{\mathrm{B}^{+}_{\frac{7}{8}}}\right)} \leq \mathrm{C}(n,\lambda,\Lambda,\mu_{0},p,\varepsilon_{0}, \|\gamma\|_{L^{\infty}(\mathrm{T}_{1} )}, \mathrm{C}_{1},\mathrm{C}_{2}), \,\,\,\forall k\in\mathbb{N}.
\end{equation*}

In this case, we find that $(u_{k})_{k\in\mathbb{N}}$ is a bounded sequence in $C^{0,\alpha^{\prime}}\left(\overline{\mathrm{B} ^{+}_{\frac{7}{8}}}\right)$. With this bound, it follows that such a sequence is equicontinuous and point-wise bounded. Similarly, as in the proof of Lemma \ref{approx}, up to a subsequence,  $F_{k_{j}}$ converges to a $(\lambda, \Lambda)$-elliptic operator $F_{\infty}$ , $g_{k_{j}}\to g_{\infty}$ in $L^{\infty}(\overline{\mathrm{T}_{1}})$ and $u_{k_{j}}\to u_{\infty}$ in $L^{\infty}(\overline{\mathrm{B}_{\frac{7}{8}}^{+}})$ when $j\to \infty$. Furthermore, since $\|f_{k_{j}}\|_{L^p(\mathrm{B}^{+}_{1})}\leq \frac{1}{k_{j}}$ for every $k\in\mathbb{N}$, it follows that the sequence $(f_{k_{j}})$ is Cauchy in $L^{p}(\mathrm{B}^{+}_{1} )$. Therefore, there is $f_{\infty}\in L^{p}(\mathrm{B}^{+}_{1})$ such that $f_{k_{j}}\to f_{\infty } = 0$ in $L^{p}(\mathrm{B}^{+}_{1})$. Hence, for every $\varphi\in C^{2}(\mathrm{B}_{r}(x_{0}))$ with $\mathrm{B}_{r}(x_{0})\subset \mathrm{B}^{+}_{\frac{7}{8}}$, we obtain
\begin{eqnarray*}
|F_{k_{j}}(D^{2}\varphi(x))-f_{k_{j}}(x)-F_{\infty}(D^{2}\varphi(x))- (f_{\infty})_{1}|&\leq&|F_{k_{j}}(D^{2}\varphi(x))-F_{\infty}(D^{2}\varphi( x))|+\\
&+&|f_{k_{j}}(x)-(f_{k_{j}})_{1}|+\\
&+&|(f_{k_{j}})_{1}-(f_{\infty})_{1}|,
\end{eqnarray*}
where, by the above assumptions on $f_{k_{j}}$ and the convergence $F_{k_{j}}\to F_{\infty}$ on compact sets of $\mathrm{Sym}(n)$, it follows that
\begin{eqnarray*}
\lim_{j\to \infty}\|F_{k_{j}}(D^{2}\varphi(\cdot))-f_{k_{j}}(\cdot)-F_{\infty}( D^{2}\varphi(\cdot))-(f_{\infty})_{1}\|_{L^{p}(\mathrm{B}_{r}(x_{0}))}=0.
\end{eqnarray*}
Therefore, by Stability Lemma $\ref{Stability}$, we can conclude that $u_{\infty}$ is a viscosity solution of
$$
\left\{
\begin{array}{rclcl}
F_{\infty}(D^{2}u_{\infty}) & = & (f_{\infty})_{1} &\mbox{in}& \mathrm{B}^{+}_{\frac{7}{8}} \\
\beta\cdot Du_{\infty}+\gamma u_{\infty} & = & g_{\infty} &\mbox{on}& \mathrm{T}_{\frac{7}{8}}.
\end{array}
\right.
$$
Thus, by defining $w_{k{j}}\defeq u_{\infty}-\mathfrak{h}_{k_{j}}$ we have, analogously to Lemma \ref{approx}, that $w_{k_{j}}\to 0$ in $\overline{\mathrm{B}^{+}_{\frac{7}{8}}}$ as $j\to \infty$, and so $\mathfrak{h}_{k_{j}}\to u_{\infty}$, which is a contradiction.
\end{proof}


\section{Optimal H\"{o}lder estimates}

This section is devoted to proving H\"{o}lder regularity estimates for solutions of \eqref{1} under suitable assumptions on the problem's data. For such a purpose, the next result constitutes the first step in a sophisticated geometric approximation process, which will yield in the desired H\"{o}lder estimate.

\begin{proposition}\label{compI}
Let $u$ be a viscosity solution of \eqref{1}, where $\beta$, $\gamma$, $g\in C^{0,\alpha}(\overline{\mathrm{T}_{1}})$ for some $\alpha\in(0,1)$. Given $\overline{\alpha}\in (0,\alpha)$, there exist constants $\eta>0$ and $\rho\in\left(0,\frac{1}{2}\right]$ depending only on  $n$, $p$, $\lambda$, $\Lambda$, $\mu_{0}$,  $\alpha$, $\overline{\alpha}$, $\|\beta\|_{C^{0,\alpha}(\overline{\mathrm{T}_{1}})}$, $\|\gamma\|_{C^{0,\alpha}(\overline{\mathrm{T}_{1}})}$ and $\|g\|_{C^{0,\alpha}(\overline{\mathrm{T}_{1}})}$ such that if
$$
\max\left\{ \left(\intav{\mathrm{B}^{+}_{1}} |\Phi_{F}(x)|^{n}dx\right)^{1/n}, \, \left( \int_{\mathrm{B}^{+}_{1}}|f(x)|^{p}dx \right)^{1/p}\right\}\leq \eta^{n},
$$
for $p\in[n-\varepsilon_{0},n)$, then there exists a constant $\mu \in \mathbb{R}$, universality bounded in following sense
$$
|\mu|\leq \mathrm{C}(n,\lambda,\Lambda,\mu_{0},\alpha,\|\beta\|_{C^{0,\alpha}(\overline{\mathrm{T}_{1}})},\|\gamma\|_{C^{0,\alpha}(\overline{\mathrm{T}_{1}})},\|g\|_{C^{0,\alpha}(\overline{\mathrm{T}_{1}})}),
$$
such that
$$
\sup_{\mathrm{B}^{+}_{\rho}}|u-\mu|\leq \rho^{\overline{\alpha}}.
$$
\end{proposition}

\begin{proof}
Fix $\delta>0$ to be chosen \textit{a posteriori}. We apply the Lemma \ref{approx} to find a function $\mathfrak{h}$ such that
\begin{eqnarray*}
\left\{
\begin{array}{rclcl}
F(D^{2}\mathfrak{h},0)=0 &\mbox{in}&   \mathrm{B}^{+}_{\frac{7}{8}} \\
\beta\cdot D \mathfrak{h}+\gamma \mathfrak{h}=g  &\mbox{on}&  \mathrm{T}_{\frac{7}{8}}\\
\mathfrak{h}= u &\mbox{on}& \partial \mathrm{B}^{+}_{\frac{7}{8}}\setminus \mathrm{T}_{\frac{7}{8}},
\end{array}
\right.
\end{eqnarray*}
such that
\begin{eqnarray}\label{4.5}
\sup_{\mathrm{B}^{+}_{\frac{7}{8}}}|u-\mathfrak{h}|\leq \delta
\end{eqnarray}

Now, by the $C^{1,\alpha}$ regularity theory for equations with frozen coefficients and oblique boundary conditions (see \cite[Theorem 1.2]{LiZhang} for details), it follows that $\mathfrak{h}\in C^{1,\alpha}\left(\overline{\mathrm{B}^{+}_{\frac{2}{3}}}\right)$ and
\begin{eqnarray*}
\|\mathfrak{h}\|_{C^{1,\alpha}\left(\overline{\mathrm{B}^{+}_{\frac{2}{3}}}\right)}\leq \mathrm{C}\left(\|\mathfrak{h}\|_{L^{\infty}\left(\mathrm{B}^{+}_{\frac{7}{8}}\right)}+\|g\|_{C^{0,\alpha}\left(\overline{\mathrm{T}_{\frac{7}{8}}}\right)}\right),
\end{eqnarray*}
where $\mathrm{C}>0$ is a constant depending only on $n$, $\lambda$, $\Lambda$, $\mu_{0}$, $\alpha$, $\|\beta\|_{C^{0,\alpha}(\overline{\mathrm{T}_{1}})}$, and $\|\gamma\|_{C^{0,\alpha}(\overline{\mathrm{T}_{1}})}$. Thus, the last estimate, together with the fact that $u$ is normalized, ensures that
\begin{eqnarray}\label{4.6}
\|\mathfrak{h}\|_{C^{1,\alpha}\left(\overline{\mathrm{B}^{+}_{\frac{2}{3}}}\right)}\leq\tilde{\mathrm{C}}=\tilde{\mathrm{C}}(n,\lambda,\Lambda,\mu_{0},\alpha,\|\beta\|_{C^{0,\alpha}(\overline{\mathrm{T}_{1}})},\|\gamma\|_{C^{0,\alpha}(\overline{\mathrm{T}_{1}})},\|g\|_{C^{0,\alpha}(\overline{\mathrm{T}_{1}})}).
\end{eqnarray}
In particular, it follows by the Value Mean Inequality for all $x\in \mathrm{B}^{+}{r}$ and $r\in\left(0,\frac{1}{2}\right)$, the following 
\begin{eqnarray}\label{4.7}
\displaystyle \sup_{x\in \mathrm{B}^{+}_{r}}|\mathfrak{h}(x)-\mathfrak{h}(0)|\leq \tilde{\mathrm{C}}r.
\end{eqnarray}

Now, for a given $\tilde{\alpha}\in(0,1)$, we make the following choices
\begin{equation}\label{EqRadius}
\rho \defeq \min\left\{\frac{1}{2}, \,\left(\frac{1}{2\tilde{\mathrm{C}}}\right)^{\frac{1}{1-\tilde{\alpha}}}\right\} \ \ \mbox{and} \ \ \delta \defeq \frac{1}{2}\rho^{\tilde{\alpha}}.
\end{equation}
Such choices determine the constant $\eta$, liked through the Approximation Lemma \ref{approx}. Finally, choosing $\mu=\mathfrak{h}(0)$, it follows by $\eqref{4.6}$ that $|\mu|\leq \tilde{\mathrm{C}}$.

Furthermore, by $\eqref{4.5}$, $\eqref{4.7}$ and \eqref{EqRadius},	we have
\begin{eqnarray*}
\sup_{\mathrm{B}^{+}_{\rho}}|u-\mu|&\leq& \sup_{\mathrm{B}^{+}_{\rho}}|u-\mathfrak{h}|+\sup_{\mathrm{B}^{+}_{\rho}}|\mathfrak{h}-\mu|\\
&\stackrel{\rho<\frac{7}{8}}{\leq}&\sup_{\mathrm{B}^{+}_{\frac{7}{8}}}|u-\mathfrak{h}|+\sup_{\mathrm{B}^{+}_{\rho}}|\mathfrak{h}-\mu|\\
&\leq& \delta +\mathrm{C}\rho\\
&\le&\rho^{\tilde{\alpha}}.
\end{eqnarray*}
which concludes the proof of the Lemma.
\end{proof}

In the next results, $\varepsilon_{0}\in\left(0,\frac{n}{2}\right)$ denotes the Escauriaza's constant (see \cite{Esc}).

\begin{theorem}\label{holderoptimal}
Let $u$ be a viscosity solution of $\eqref{1}$, where $\beta,\gamma, g\in C^{0,\alpha}(\overline{\mathrm{T}_{1}} )$ for some $\alpha\in(0,1)$ and $f\in L^{p}(\mathrm{B}^{+}_{1})\cap C^{0 }(\mathrm{B}^{+}_{1})$ to $p\in[n-\varepsilon_{0},n)$. There exists a constant $\eta_{0}>0$, depending only on $n$, $p$, $\lambda$, $\Lambda$, $\mu_{0}$, $\alpha$, $\|\beta \|_{C^{0,\alpha}(\overline{\mathrm{T}_{1}})}$ and $\|\gamma\|_{C^{0,\alpha}(\overline {\mathrm{T}_{1}})}$, such that if
\begin{eqnarray}\label{hypothesis1theorem4.1.2}
\intav{\mathrm{B}^{+}_{r}}|\Phi_{F}(x, y)|^{n}dx\leq \eta_{0}^{n}, \,\, \forall y\in \mathrm{B}^{+}_{\frac{1}{2}}, \forall r\in \left(0,\frac{1}{2}\right),
\end{eqnarray}
then $u\in C^{0,\frac{2p-n}{p}}\left(\overline{\mathrm{B}^{+}_{\frac{1}{2}}}\right)$. Moreover, the following estimate holds
\begin{eqnarray*}
\|u\|_{C^{0,\frac{2p-n}{p}}\left(\overline{\mathrm{B}^{+}_{\frac{1}{2}}}\right)} \leq \mathrm{C}(\|u\|_{L^{\infty}(\mathrm{B}^{+}_{1})}+\|f\|_{L^{p} (\mathrm{B}^{+}_{1})}+\|g\|_{C^{0,\alpha}(\overline{\mathrm{T}_{1}})})
\end{eqnarray*}
where $\mathrm{C}=\mathrm{C}(n,\lambda,\Lambda,\mu_{0},p,\alpha,\|\beta\|_{C^{0,\alpha}( \overline{\mathrm{T}_{1}})},\|\gamma\|_{C^{0,\alpha}(\overline{\mathrm{T}_{1}})})$ is a positive constant.
\end{theorem}

\begin{proof}
Initially, we can assume, without loss of generality, that
$$
\|u\|_{L^{\infty}(\mathrm{B}^{+}_{1})}\leq 1,\,\, \|f\|_{ L^{p}(\mathrm{B}^{+}_{1})}\leq \eta \quad  \text{and}\quad  \|g\|_{C^{0,\alpha}(\overline{\mathrm{T}_{1}})}\leq 1,
$$
where $\eta>0$ is the constant from Approximation Lemma \ref{approx} when we set $\tilde{\alpha}=\frac{2p-n}{ p}$. In fact, if such conditions do not occur, we define the constant
\begin{eqnarray*}
\kappa \defeq \frac{\eta}{\eta\|u\|_{L^{\infty}(\mathrm{B}^{+}_{1})}+\|f\|_{L ^{p}(\mathrm{B}^{+}_{1})}+\eta \|g\|_{C^{0,\alpha}(\overline{\mathrm{T}_{1}})}}
\end{eqnarray*}
and we see that $\tilde{u}(x)=\kappa u(x)$ is a viscosity solution of
$$
\left\{
\begin{array}{rclcl}
\tilde{F}(D^{2}\tilde{u},x) & = & \tilde{f}(x) &\mbox{in}&   \mathrm{B}^{+}_{1} \\
\tilde{\beta}\cdot D\tilde{u}+\tilde{\gamma} \tilde{u} & = & \tilde{g}(x)  &\mbox{on}&  \mathrm{T}_{1},
\end{array}
\right.
$$
where
$$
\left\{
\begin{array}{rcl}
\tilde{F}(\mathrm{X},x) & \defeq & \kappa F\left(\frac{1}{\kappa} \mathrm{X}, x\right) \\
\tilde{f}(x) & \defeq & \kappa f(x)\\
\tilde{\beta}(x) & \defeq &  \beta(x)\\
\tilde{\gamma}(x) & \defeq & \gamma(x)\\
\tilde{g}(x)& \defeq & \kappa g(x),
\end{array}
\right.
$$
Thus, it is easy to check that
$$
\|\tilde{u}\|_{L^{\infty}(\mathrm{B}^{+}_{1})}\leq 1,\quad \|\tilde{f }\|_{L^{p}(\mathrm{B}^{+}_{1})}\leq \eta \quad  \text{and} \quad \|\tilde{g}\|_{C^{0,\alpha}(\overline{\mathrm{T}_{1}})}\leq 1
$$
and thus, assuming that the theorem is valid under these conditions, it follows that $\tilde{u}\in C^{0,\frac{ 2p-n}{p}}\left(\overline{\mathrm{B}^{+}_{\frac{1}{2}}}\right)$, with the following estimates
\begin{eqnarray}\label{estimate1theorem4.1.2}
\|\tilde{u}\|_{C^{0,\frac{2p-n}{p}}\left(\overline{\mathrm{B}^{+}_{\frac{1}{2}}}\right)} \leq \mathrm{C}\left(n,\lambda,\Lambda,\mu_{0},p,\alpha, \|\beta\|_{C^{0,\alpha}(\overline{\mathrm{T}_{1}})},\|\gamma\|_{C^{0,\alpha}(\overline{\mathrm{T}_{1}})}\right),
\end{eqnarray}
which clearly yields that $u\in C^{0,\frac{2p-n}{p}}\left(\overline{\mathrm{B}^{+}_{\frac{1}{2}}}\right)$, and re-scaling the scalar factor in the estimate \eqref{estimate1theorem4.1.2} implies the desired estimate in the theorem's thesis.

Thus we can, in fact, make the assumption stated above at the beginning of the proof. In this case, choose $\eta_{0}=\eta$. To prove the desired result, for a fixed  $y\in\mathrm{T}_{\frac{1}{2}}$, we assert that there exists a sequence of real numbers $(\mu_{k})_{ k\in\mathbb{N}}$ such that for all $k\in\mathbb{N}$,
\begin{eqnarray}\label{estimate2theorem4.1.2}
\sup_{z\in\mathrm{B}^{+}_{\rho^{k}}(y)}|u(z)-\mu_{k}|\leq \rho^{k\frac{ 2p-n}{p}},
\end{eqnarray}
where $\rho\in\left(0,\frac{1}{2}\right]$ is the radius of the semi-ball obtained in Proposition \ref{compI}. Moreover,  such sequence satisfies the following approximation rate
\begin{eqnarray}\label{estimate4theorem4.1.2}
|\mu_{k+1}-\mu_{k}|\leq \mathrm{C}\rho^{k\frac{2p-n}{p}} \quad \forall\,\, k \in \mathbb{N}.
\end{eqnarray}

It is worth a quick digression to note that because $\rho\in\left(0,\frac{1}{2}\right]$ and $y\in\mathrm{T}_{\frac{1}{2}}$, it follows that $\mathrm{T}_{\rho^{k}}(y)\subset \mathrm{T}_{1}$ and $\mathrm{B}^{+}_{\rho^{k}}(y)\subset \mathrm{B}^{+}_{1}\cup \mathrm{T}_{1}$ for all $k\in\mathbb{N}$. With such an observation in hand, we prove the statement by induction in $k\in \mathbb{N}$. Indeed, when $k=1$, we already have the existence of the constant $\mu_{1}$ guaranteed by Proposition \ref{compI} together with the estimate $\eqref{estimate2theorem4.1.2}$. Now, assuming, by the induction hypothesis, that the statement holds for $k\in\mathbb{N}$, let us define the auxiliary function
\begin{eqnarray*}
v_{k}(x) \defeq \frac{u(y+\rho^{k}x)-\mu_{k}}{\rho^{k}\frac{2p-n}{p}}, \  x\in\mathrm{B}^{+}_{1}\cup \mathrm{T}_{1}.
\end{eqnarray*}
Thus, we see that $v_{k}$ is a viscosity solution to
$$
	\left\{
	\begin{array}{rclcl}
		F_{k}(D^{2}v_{k},x) & = & f_{k}(x) &\mbox{in}& \mathrm{B}^{+}_{1} \\
		\beta_{k}(x)\cdot Dv_{k}(x)+\gamma_{k}(x) v_{k}(x) & = & g_{k}(x) &\mbox{on}& \mathrm{T}_{1},
	\end{array}
	\right.
$$
where
$$
\left\{
\begin{array}{rcl}
	F_{k}(\mathrm{X},x) & \defeq & \rho^{k\frac{n}{p}}F\left(\frac{1}{\rho^{k\frac{n }{p}}} \mathrm{X},y+\rho^{k}x\right) \\
	f_k(x) & \defeq & \rho^{k\frac{n}{p}} f(y+\rho^{k}x)\\
	\beta_{k}(x) & \defeq & \beta(y+\rho^{k}x)\\
	\gamma_{k}(x) & \defeq & \rho^{k}\gamma(y+\rho^{k}x)\\
	g_{k}(x)&\defeq & \rho^{k\left(-1+\frac{n}{p}\right)}(g(y+\rho^{k}x)-\mu_{ k}\gamma(y+\rho^{k}x)).
\end{array}
\right.
$$

In this context, we claim that $v_{k}$ falls into the hypotheses of Proposition \ref{compI}. In fact, by the induction hypothesis, it follows from the estimate $\eqref{estimate2theorem4.1.2}$ for all $k \ge 1$ that $\|v_{k}\|_{L^{\infty}(\mathrm{B}^{+}_{1})}\leq 1$. Furthermore, we clearly have, by $\beta,\gamma \in C^{0,\alpha}(\overline{\mathrm{T}_{1}})$ and $\rho\in\left(0,\frac{1}{2}\right)$, that
\begin{eqnarray*}
[\beta_{k}]_{0,\alpha,\mathrm{T}_{1}}& \defeq &\sup_{\stackrel{x,z\in \mathrm{T}_{1}}{ x\neq z}}\frac{|\beta_{k}(x)-\beta_{k}(z)|}{|x-z|^{\alpha}}=\rho^{k\alpha}\sup_ {\stackrel{x,z\in \mathrm{T}_{1}}{x\neq z}}\frac{|\beta(y+\rho^{k}x)-\beta(y+\rho^ {k}z)|}{|(y+\rho^{k}x)-(y+\rho^{k}z)|^{\alpha}}\\
&=&\rho^{k\alpha}[\beta]_{0,\alpha,\mathrm{T}_{\rho^{k}}(y)}\\
&\leq& [\beta]_{0, \alpha,\mathrm{T}_{1}}<\infty,
\end{eqnarray*}
hence $\beta_{k}\in C^{0,\alpha}(\overline{\mathrm{T}_{1}})$, and
\begin{eqnarray*}
[\gamma_{k}]_{0,\alpha,\mathrm{T}_{1}}& \defeq &\sup_{\stackrel{x,z\in \mathrm{T}_{1}}{ x\neq z}}\frac{|\gamma_{k}(x)-\gamma_{k}(z)|}{|x-z|^{\alpha}}=\rho^{k(1+\alpha )}\sup_{\stackrel{x,z\in \mathrm{T}_{1}}{x\neq z}}\frac{|\gamma(y+\rho^{k}x)-\gamma( y+\rho^{k}z)|}{|(y+\rho^{k}x)-(y+\rho^{k}z)|^{\alpha}}\\
&=&\rho^{k(1+\alpha)}[\gamma]_{0,\alpha,\mathrm{T}_{\rho^{k}}(y)}\\
&\leq& [\gamma] _{0,\alpha,\mathrm{T}_{1}}<\infty,
\end{eqnarray*}
also guaranteeing the $\alpha$-H\"{o}lder regularity of $\gamma_{k}$ on $\overline{\mathrm{T}_{1}}$.

Additionally, we also see that $g_{k}\in C^{0,\alpha}(\overline{\mathrm{T}_{1}})$. In fact,
\begin{eqnarray*}
[g_{k}]_{0,\alpha,\mathrm{T}_{1}}&\leq& \rho^{k\left(-1+\frac{n}{p}\right)}( \rho^{k\alpha}[g]_{0,\alpha,\mathrm{T}_{\rho^{k}}(y)}+|\mu_{k}|\rho ^{k\alpha}[\gamma]_{0,\alpha,\mathrm{T}_{\rho^{k}}(y)})\\
&\leq&[g]_{0,\alpha,\mathrm{T}_{1}}+|\mu_{k}|[\gamma]_{0,\alpha,\mathrm{T}_{1}}<\infty,
\end{eqnarray*}
since $\rho\in\left(0,\frac{1}{2}\right]$, $0<\frac{n}{p}-1\leq 1$ (remember that $p\in[n-\varepsilon_{0},n)$), and by continuity of $u$ and $v_{k}$, it follows (by induction hypothesis) that
\begin{eqnarray*}
\|u-\mu_{ k}\|_{L^{\infty}(\overline{\mathrm{B}^{+}_{\rho^{k}} (y)})}\leq \rho^{k\frac{2p-n}{p}} \quad \mbox{and} \quad  \|u\|_{L^{\infty}(\mathrm{B}_{1}^{+}\cup \mathrm{T}_{1})}\leq 1
\end{eqnarray*}
which implies $|\mu_{ k}|\leq \frac{3}{2}$. Furthermore, by the definition of $f_{k}$, we have that
\begin{eqnarray*}
\int_{\mathrm{B}^{+}_{1}}|f_{k}(x)|^{p}dx=\rho^{kn}\int_{\mathrm{B}^{+}_{1}}|f(y+\rho^{k}x)|^{p}dx=\int_{\mathrm{B}^{+}_{\rho^{k}}(y)}|f(z)|^{p}dz\leq\int_{\mathrm{B}^{+}_{1}}|f(z)|^{p}dz\leq \eta^{p}.
\end{eqnarray*}

Finally, note that, by the definition of $F_{k}$, it is clearly a $(\lambda,\Lambda)$-elliptic operator, and by the hypothesis \eqref{hypothesis1theorem4.1.2}, it follows that
\begin{eqnarray*}
\intav{\mathrm{B}^{+}_{1}}|\Phi_{F_{k}}(x)|^{n}dx=\intav{\mathrm{B}^{+}_{1}}|\Phi_{F}(y+\rho^{k}x,y)|^{n}dx=\intav{\mathrm{B}^{+}_{\rho^{k}}( y)}|\Phi_{F}(z,y)|^{n}dz\leq \eta_{0}^{n}=\eta^{n}.
\end{eqnarray*}

Therefore, we can invoke Proposition \ref{compI} for $v_{k}$ and obtain a constant $\tilde{\mu}$ such that $|\tilde{\mu}|\leq \mathrm{C}$, and
\begin{eqnarray}\label{estimate3theorem4.1.2}
\sup_{\mathrm{B}^{+}_{\rho}}|v_{k}-\tilde{\mu}|\leq \rho^{\frac{2p-n}{p}}.
\end{eqnarray}

Now, we define $\mu_{k+1}=\mu_{k}+\rho^{k\frac{2p-n}{p}}\tilde{\mu}$. Thus, by the definition of $v_{k}$ together with \eqref{estimate2theorem4.1.2} and \eqref{estimate3theorem4.1.2}, we have
\begin{eqnarray*}
\sup_{\mathrm{B}^{+}_{\rho^{k+1}}(y)}|u-\mu_{k+1}|\leq \rho^{(k+1)\frac{2p-n}{p}},
\end{eqnarray*}
which proves the $(k+1)^{\text{th}}$ step of induction.

Thus, by induction, the assertion of the existence of the sequence $(\mu_{k})_{k\in\mathbb{N}}$ in $\mathbb{R}$ satisfying \eqref{estimate2theorem4.1.2} follows. From the estimate \eqref{estimate4theorem4.1.2}, we can conclude that $(\mu_{k})_{k\in\mathbb{N}}$ is a Cauchy sequence in $\mathbb{R}$, and therefore there exists a real number $\displaystyle \mu_{\infty} =\lim_{k \to \infty}\mu_{k}$. We claim that $\mu_{\infty}=u(y)$. Indeed, we fix $x_{0}\in\mathrm{B}^{+}{1}$ and define for each $k\in\mathbb{N}$, $z_{k}=y+\rho^ {k}x_{0}$. We clearly see that $z_{k}\in\mathrm{B}^{+}{\rho^{k}}(y)$ for all $k\in\mathbb{N}$ and that $z_{k}\to y$ when $k\to\infty$, since,
\begin{eqnarray*}
|z_{k}-y|=\rho^{k}|x_{0}|<\rho^{k}\to 0, \ \mbox{when} \ k\to \infty.
\end{eqnarray*}

Hence, by the continuity of the function $u$, we also have $u(z_{k})\to u(y)$. Thus, using such convergences, \eqref{estimate3theorem4.1.2}, and \eqref{estimate4theorem4.1.2}, we have
\begin{eqnarray*}
|u(y)-\mu_{\infty}|&\leq& |u(z_{k})-u(y)|+|u(z_{k})-\mu_{k}|+|\mu_ {k}-\mu_{\infty}|\\
&\stackrel{z_{k}\in\mathrm{B}^{+}_{\rho^{k}}(y)}{\leq}&|u(z_{k})-u(y) |+\sup_{\mathrm{B}^{+}_{\rho^{k}}(y)}|u-\mu_{k}|+|\mu_{k}-\mu_{\infty} |\\
&\leq&|u(z_{k})-u(y)|+\rho^{k\frac{2p-n}{p}}+|\mu_{k}-\mu_{\infty}|\to 0,\ \mbox{when} \ k\to\infty.
\end{eqnarray*}

Thus, $\mu_{\infty}=u(y)$. On the other hand, given any natural $k<m$, we have, by the condition \eqref{estimate4theorem4.1.2}, that
\begin{eqnarray*}
|\mu_{k}-\mu_{m}|\leq \sum_{j=k}^{m-1}|\mu_{j+1}-\mu_{j}|\leq \mathrm{C} \sum_{j=k}^{m-1}\rho^{j\frac{2p-n}{p}}=\mathrm{C}\frac{\rho^{k\frac{2p-n} {p}}\left(\rho^{(m-k)\frac{2p-n}{p}}-1\right)}{\rho^{\frac{2p-n}{p}}-1}.
\end{eqnarray*}
Moreover, fixed $k\in\mathbb{N}$ as above, and letting $m\to \infty$, we obtain, by the convergence of $\mu_{m}\to u(y)$ the following
\begin{eqnarray}\label{estimate5theorem4.1.2}
|u(y)-\mu_{k}|\leq \frac{\mathrm{C}}{1-\rho^{\frac{2p-n}{p}}}\rho^{k\frac{ 2p-n}{p}}.
\end{eqnarray}

Finally, we fix $0<r<\rho$, and choose $k\in\mathbb{N}$, such that $\rho^{k+1}<r\leq \rho^{k}$. From \eqref{estimate2theorem4.1.2} and \eqref{estimate5theorem4.1.2}, we have that
\begin{eqnarray}\label{estimate6theorem4.1.2}
\sup_{x\in\mathrm{B}^{+}_{r}(y)}|u(x)-u(y)|&\stackrel{r\leq\rho^{k}}{\leq}& \sup_{x\in \mathrm{B}^{+}_{\rho^{k}}(y)}|u(x)-\mu_{k}|+ |\mu_{k} -u(y)|\\
&\leq& \rho^{k\frac{2p-n}{p}}+\frac{\mathrm{C}}{1-\rho^{\frac{2p-n}{p }}}\rho^{k\frac{2p-n}{p}}\nonumber\\
&=&\frac{1}{\rho}\left(1+\frac{\mathrm{C}}{1-\rho^{\frac{2p-n}{p}}}\right)\rho ^{(k+1)\frac{2p-n}{p}}\\
&\stackrel{\rho^{k+1}<r}{\leq}&\mathrm{C}r^{\frac{2p-n }{p}}.
\end{eqnarray}

Now, we prove that indeed $u\in C^{0,\frac{2p-n}{p}}\left(\overline{\mathrm{B}^{+}_{\frac{1}{2}}}\right )$. For this purpose, given $x\in\mathrm{B}^{+}_{\frac{1}{2}}$ and $y\in\mathrm{T}_{\frac{1}{2}}$, we have two possible cases to analyze:

\begin{enumerate}
\item[\checkmark] {\bf Case 1:} $r=|x-y|\geq\rho$: In this case, it immediately follows that
\begin{eqnarray*}
\frac{|u(x)-u(y)|}{|x-y|^{\frac{2p-n}{p}}}\leq \frac{\overbrace{2\|u\|_{L^{\infty}(\overline{\mathrm{B}^{+}_{1/2}})}}^{\leq 2}}{\rho^{\frac{2p-n}{p}}}\leq \mathrm{C}.
\end{eqnarray*}
\item[\checkmark] {\bf Case 2:} $r=|x-y|< \rho$: Note that the estimate $\eqref{estimate6theorem4.1.2}$ also holds in $\overline{\mathrm{B}^{+}_{r}(y)}$, and hence, as $x\in\overline{\mathrm{B}^{+}_ {r}(y)}$, it follows from such an estimate that
\begin{eqnarray*}
\frac{|u(x)-u(y)|}{|x-y|^{\frac{2p-n}{p}}}\leq\frac{\mathrm{C} r^{\frac{2p-n}{p}}}{|x-y|^{\frac{2p-n}{p}}}=\mathrm{C}.
\end{eqnarray*}
\end{enumerate}

Finally, from these cases above and the optimal H\"{o}lder interior estimates (cf. \cite[Remark 2]{ET}), we obtain that $[u]_{0,\frac{2p-n}{p},\mathrm{B}^{+}_{\frac{1}{2}}}<\infty$. Moreover, we obtain the following estimate
\begin{eqnarray*}
\|u\|_{C^{0,\frac{2p-n}{p}}\left(\overline{\mathrm{B}^{+}_{\frac{1}{2}}}\right)} \leq \mathrm{C}(\|u\|_{L^{\infty}(\mathrm{B}^{+}_{1})}+\|f\|_{L^{p} (\mathrm{B}^{+}_{1})}+\|g\|_{C^{0,\alpha}(\overline{\mathrm{T}_{1}})}),
\end{eqnarray*}
which concludes the proof of this Theorem.
\end{proof}

An application of Theorem \ref{holderoptimal} is presented in the following result.

\begin{corollary}\label{Corollary3.3}
Consider $u$ to be a viscosity solution of $\eqref{1}$, where $\beta,\gamma,g\in C^{0,\alpha}(\overline{\mathrm{T}_{1}})$ for some $\alpha\in(0,1)$, and suppose that $f\in L^{p}(\mathrm{B}^{+}_{1})\cap C^{0}(\mathrm{B}^{+}_{1})$ for
$$p \defeq \max\left\{\frac{n}{2-\alpha},n-\varepsilon_{0}\right\}.$$ Then, there exists constant a $\eta_{0}>0$ which depends only on $n$, $\Lambda$, $\mu_{0}$, $\alpha$, $\|\beta \|_{C^{0,\alpha}(\overline{\mathrm{T}_{1}})}$ and $\|\gamma\|_{C^{0,\alpha}(\overline {\mathrm{T}_{1}})}$ such that if,
\begin{eqnarray*}
\intav{\mathrm{B}^{+}_{r}}|\Phi_{F}(x, y)|^{n}dx\leq \eta_{0}^{n}, \ \forall y\in \mathrm{B}^{+}_{\frac{1}{2}}, \forall r\in \left(0,\frac{1}{2}\right),
\end{eqnarray*}
then $u\in C^{0,2-\frac{n}{p(n, \alpha, \varepsilon_0)}}\left(\overline{\mathrm{B}^{+}_{\frac{1}{2}}}\right)$ and the following estimate holds
\begin{eqnarray*}
\|u\|_{C^{0,2-\frac{n}{p(n, \alpha, \varepsilon_0)}}\left(\overline{\mathrm{B}^{+}_{\frac{1}{2}}}\right)} \leq \mathrm{C}\left(\|u\|_{L^{\infty}(\mathrm{B}^{+}_{1})}+\|f\|_{L^{p} (\mathrm{B}^{+}_{1})}+\|g\|_{C^{0,\alpha}(\overline{\mathrm{T}_{1}})}\right),
\end{eqnarray*}
where $\mathrm{C}=\mathrm{C}\left(n,\lambda,\Lambda,\mu_{0},\alpha,\|\beta\|_{C^{0,\alpha}( \overline{\mathrm{T}_{1}})},\|\gamma\|_{C^{0,\alpha}(\overline{\mathrm{T}_{1}})}\right)>0$, and $\varepsilon_{0}\in\left(0,\frac{n}{2}\right)$ is Escauriaza's constant.
\end{corollary}

\begin{proof}
In effect, by definition of parameter $p$ it follows that $p\in [n-\varepsilon_{0},n)$. Thus, we can apply Theorem \ref{holderoptimal} to guarantee the existence of a universal constant $\eta_{0}>0$ such that $u\in C^{0,2-\frac{n}{p(n, \alpha, \varepsilon_0)}}\left(\overline{\mathrm{B}^{+}_{\frac{1}{2}}}\right)$ with the following estimate
$$
\|u\|_{C^{0,2-\frac{n}{p(n, \alpha, \varepsilon_0)}}\left(\overline{\mathrm{B}^{+}_{\frac{1}{2}}}\right)}\leq \mathrm{C}(\verb"universal")\left(\|u\|_{L^{\infty}(\mathrm{B}^{+}_{1})}+\|f\|_{L^{p}(\mathrm{B}^{+}_{1})}+\|g\|_{C^{0,\alpha}(\overline{\mathrm{T}_{1}})}\right),
$$
thereby concluding the proof.
\end{proof}

\begin{remark} We must highlight that in the Corollary \ref{Corollary3.3}, we have
$$
\alpha_p \defeq 2-\frac{n}{p(n, \alpha, \varepsilon_0)} \to 1 \quad \text{as} \quad \alpha \to 1^{-} \quad \text{and} \quad \alpha_p \defeq 2-\frac{n}{p(n, \alpha, \varepsilon_0)} \to \frac{n-2\varepsilon_0}{n-\varepsilon_0} \quad \text{as} \quad \alpha \to 0^{+}.
$$
\end{remark}

\section{Log-Lipschitz regularity estimates}

This section is devoted to address borderline estimates for solutions of \eqref{1} under suitable assumptions on the problem's data and $L^n$-integrability on source term. In this context, the next result provides the first step of an affine approximation scheme, which will establish the desired Log-Lipschitz estimate.

\begin{proposition}
\label{compII}
Let $u$ be a solution of \eqref{1}, where $\beta$, $\gamma$, $g\in C^{0,\alpha}(\overline{\mathrm{T}_{1}} )$ to some $\alpha\in(0,1)$. There are $\eta>0$ and $\rho\in\left(0,\frac{1}{2}\right)$ depending only on $n$, $p$, $\lambda$, $\Lambda$ , $\mu_{0}$, $\alpha$, $\|\beta\|_{C^{0,\alpha}(\overline{\mathrm{T}_{1}})}$, $ \|\gamma\|_{C^{0,\alpha}(\overline{\mathrm{T}_{1}})}$ and $\|g\|_{C^{0,\alpha} (\overline{\mathrm{T}_{1}})}$ such that, if
$$
\max\left\{\intav{\mathrm{B}^{+}_{1}} |\Phi_{F}(x)|^{n}dx, \,\, \int_{\mathrm{B}^{+}_{1}}|f(x)|^{n}dx\right\}\leq \eta^{n},
$$
then there exists an affine function $\mathfrak{l}(x)=\mathbf{a}+\mathbf{b}\cdot x$, with universally bounded coefficients in the following sense
$$
|\mathbf{a}|+\|\mathbf{b}\|\leq \mathrm{C}(n,\lambda,\Lambda,\mu_{0},\alpha,\|\beta\|_{C^{0,\alpha}(\overline{\mathrm{T}_{1}})},\|\gamma\|_{C^{0,\alpha}(\overline{\mathrm{T}_{1}})},\|g\|_{C^{0,\alpha}(\overline{\mathrm{T}_{1}})}),
$$
such that
$$
\sup_{\mathrm{B}^{+}_{\rho}}|u-\mathfrak{l}|\leq \rho.
$$
\end{proposition}
\begin{proof}
	Let $\delta>0$ be a constant, which we will determine \textit{a posteriori}. Thus, by \eqref{approx} we can consider $\mathfrak{h}$ viscosity solution for
	\begin{eqnarray*}
		\left\{
		\begin{array}{rclcl}
			F(D^{2}\mathfrak{h},0)& = & 0 &\mbox{in}& \mathrm{B}^{+}_{\frac{7}{8}} \\
			\beta\cdot D \mathfrak{h}+\gamma \mathfrak{h} & = & g(x) &\mbox{on}& \mathrm{T}_{\frac{7}{8}}\\
			\mathfrak{h} & = & u &\mbox{on}& \partial \mathrm{B}^{+}_{\frac{7}{8}}\setminus \mathrm{T}_{\frac{7}{8}},
		\end{array}
		\right.
	\end{eqnarray*}
	such that
	\begin{eqnarray}\label{estimate1lemma4.2.1}
		\sup_{\mathrm{B}^{+}_{\frac{7}{8}}}|u-\mathfrak{h}|\leq \delta.
	\end{eqnarray}
Hence, by \cite[Theorem1.2]{LiZhang}, it follows that $\mathfrak{h}\in C^{1,\alpha}(\overline{\mathrm{B}^{+}_{\frac{2}{3}}})$ and
\begin{eqnarray}\label{estimate2lemma4.2.1}
\|\mathfrak{h}\|_{C^{1,\alpha}\left(\overline{\mathrm{B}^{+}_{\frac{2}{3}}}\right)}\leq\tilde {\mathrm{C}}=\tilde{\mathrm{C}}\left(n,\lambda,\Lambda,\mu_{0},\alpha,\|\beta\|_{C^{0,\alpha }(\overline{\mathrm{T}_{1}})},\|\gamma\|_{C^{0,\alpha}(\overline{\mathrm{T}_{1}})} ,\|g\|_{C^{0,\alpha}(\overline{\mathrm{T}_{1}})}\right).
\end{eqnarray}
Now, we define $\mathbf{a}=\mathfrak{h}(0)$, $\mathbf{b}=D\mathfrak{h}(0)$ and $\mathfrak{l}(x) = \mathbf{a}+\mathbf{b}\cdot x$. Thus, by \eqref{estimate2lemma4.2.1}, we have
\begin{eqnarray}\label{estimate3lemma4.2.1}
\sup_{\mathrm{B}^{+}_{r}}|\mathfrak{h}-\mathfrak{l}|\leq \tilde{\mathrm{C}}r^{1+\alpha}, \,\,\,\forall r\in \left(0, \frac{2}{3}\right).
\end{eqnarray}

Finally, we can choose $\rho$ and $\delta$ by setting
\begin{equation}\label{EqRadius_2}
\rho= \min\left\{\left(\frac{1}{2\tilde{\mathrm{C}}}\right)^{\frac{1}{\alpha}}, \frac{1}{e}\right\} \qquad \mbox{and} \qquad \delta= \frac{1}{2}\rho.
\end{equation}
Thus, note that such choices determine the constant $\eta>0$ due to Lemma \ref{approx}. The universal bound of the constants $\mathbf{a}$ and $\mathbf{b}$ follows directly from \eqref{estimate2lemma4.2.1}. For the remainder, from \eqref{estimate1lemma4.2.1}, \eqref{estimate3lemma4.2.1}, and \eqref{EqRadius_2}, we obtain that
\begin{eqnarray*}
\sup_{\mathrm{B}^{+}_{\rho}}|u-\mathfrak{l}|&\leq& \sup_{\mathrm{B}^{+}_{\rho}}|u-\mathfrak{h}|+\sup_{\mathrm{B}^{+}_{\rho}}|\mathfrak{h}-\mathfrak{l}|\leq \sup_{\mathrm{B}^{+}_{\frac{7}{8}}}|u-\mathfrak{h}|+\sup_{ \mathrm{B}^{+}_{\rho}}|\mathfrak{h}-\mathfrak{l}|\leq \delta+\tilde{\mathrm{C}}\rho^{1+\alpha}\\
&\leq&\frac{1}{2}\rho+\left(\frac{1}{2\rho^{\alpha}}\right)\rho^{1+\alpha}=\frac{1}{ 2}\rho+\frac{1}{2}\rho=\rho.
\end{eqnarray*}
\end{proof}

\begin{theorem}\label{theorem4.2}
Let $u$ be a viscosity solution of \eqref{1}, where $\beta,\gamma,g\in C^{0,\alpha}(\overline{\mathrm{T}_{1}})$ (for $\alpha\in(0,1)$) and $f\in L^{n}(\mathrm{B}^{+}_{1})\cap C^{0 }(\mathrm{B}^{+}_{1})$. There exists a constant $\eta_{0}>0$, which depends only on $n$, $p$, $\lambda$, $\Lambda$, $\mu_{0}$, $\alpha$, $\|\beta \|_{C^{0,\alpha}(\overline{\mathrm{T}_{1}})}$ and $\|\gamma\|_{C^{0,\alpha}(\overline {\mathrm{T}_{1}})}$ such that if,
\begin{eqnarray}\label{hypothesis1theorem4.2.5}
\intav{\mathrm{B}^{+}_{r}}|\Phi_{F}(x, y)|^{n}dx\leq \eta_{0}^{n}, \ \forall y\in \mathrm{B}^{+}_{\frac{1}{2}}, \forall r\in \left(0,\frac{1}{2}\right),
\end{eqnarray}
then $u\in C^{0, \text{Log-Lip}}\left(\overline{\mathrm{B}^{+}_{\frac{1}{2}}}\right)$ and the following estimate holds
\begin{eqnarray*}
\displaystyle \sup_{x,y\in \overline{\mathrm{B}^{+}_{\frac{1}{2}}} \atop{x\neq y}} \frac{|u(x)-u(y)|}{|x-y|\ln(|x-y|^{-1})}\leq \mathrm{C}\left(\|u\|_{L^{\infty}(\mathrm{B}^{+}_{1})}+ \|f\|_{L^{n}(\mathrm{B}^{+}_{1})}+\|g\|_{C^{0,\alpha}(\overline{\mathrm {T}_1})}\right),
\end{eqnarray*}
where $\mathrm{C}>0$ depends only on $n$, $\lambda$, $\Lambda$, $\mu_{0}$, $\|\beta\|_{C^{0,\alpha}(\overline{\mathrm{T}_{1}} )}$ and $\|\gamma\|_{C^{0,\alpha}(\overline{\mathrm{T}_1})}$.
\end{theorem}

\begin{proof}
As in the proof of Theorem \ref{holderoptimal}, we can assume, without loss of generality, that $u$ is normalized ( i.e., $\|u\|_{L^{\infty}(\mathrm{B}^{+}_{1})} \leq 1$), $\|f\|_{L^{n}(\mathrm{B}^{+}_{1})}\leq \eta$ and $\|g\|_{C^ {0,\alpha}(\overline{\mathrm{T}_{1}})}\leq 1$, where $\eta>0$ is the constant of Lemma \eqref{compII}. Now, we take $\eta_{0}=\eta$. For a fixed $y\in\mathrm{T}_{\frac{1}{2}}$, we assert that there is a sequence of affine functions $(\mathfrak{l}_{k})_{k\in\mathbb{N }}$ of the form $\mathfrak{l}_{k}(x)=\mathbf{a}_{k}+\mathbf{b}_{k}\cdot(x-y)$ satisfying
\begin{eqnarray}\label{estimate1theorem4.2.5}
\sup_{\mathrm{B}^{+}_{\rho^{k}}(y)}|u-\mathfrak{l}_{k}|\leq \rho^{k},
\end{eqnarray}
where $\rho>0$ is the radius of the half ball found in Lemma \ref{compII}. Furthermore, such a sequence must satisfy, for all $k\in\mathbb{N}$,
\begin{eqnarray}\label{estimate2theorem4.2.5}
|\mathbf{a}_{k+1}-\mathbf{a}_{k}|\leq \mathrm{C}(\verb"universal")\rho^{k} \ \ \mbox{and} \ \ \|\mathbf{b}_{k+1}-\mathbf{b}_{k}\|\leq \mathrm{C}(\verb"universal"),
\end{eqnarray}
Indeed, we prove this statement by induction on $k$. The first case, i.e. $k=1$, corresponds to Lemma \ref{compII}. Now, assuming it holds for $k$, we may define the function
\begin{eqnarray*}
v_{k}(x) \defeq \frac{(u-\mathfrak{l}_{k})(y+\rho^{k}x)}{\rho^{k}}, \ x\in\mathrm{B}^ {+}_{1}\cup \mathrm{T}_{1},
\end{eqnarray*}
Thus, we see that $v_{k}$ is a viscosity solution to
$$
\left\{
\begin{array}{rclcl}
F_{k}(D^{2}v_{k},x)& = & f_{k}(x) &\mbox{in}& \mathrm{B}^{+}_{1} \\
\beta_{k}(x)\cdot Dv_{k}(x)+\gamma_{k}(x) v_{k}(x) & = & g_{k}(x) &\mbox{on}& \mathrm{T}_{1},
\end{array}
\right.
$$
here
$$
\left\{
\begin{array}{rcl}
F_{k}(\mathrm{X},x) & \defeq & \rho^{k}F\left(\frac{1}{\rho^{k}} \mathrm{X},y+\rho^ {k}x\right) \\
\tilde{f}(x) & \defeq & \rho^{k} f(y+\rho^{k}x)\\
\beta_{k}(x) & \defeq & \beta(y+\rho^{k}x)\\
\gamma_{k}(x) & \defeq & \rho^{k}\gamma(y+\rho^{k}x)\\
g_{k}(x)& \defeq & g(y+\rho^{k}x)-\beta(y+\rho^{k}x)\cdot \mathbf{b}_{k}-\gamma(y+\rho^{ k}x)\mathfrak{l}_{k}(y+\rho^{k}x).
\end{array}
\right.
$$

Now, we note that $F_{k}$ is a $(\lambda,\Lambda)$-elliptic operator and by the hypothesis \eqref{hypothesis1theorem4.2.5}, we check that
\begin{eqnarray*}
\intav{\mathrm{B}^{+}_{1}}|\Phi_{F_{k}}(x)|^{n}dx=\intav{\mathrm{B}^{+}_{ \rho^{k}}(y)}|\Phi_{F}(x,y)|^{n}dx\leq \eta^{n}.
\end{eqnarray*}
Furthermore, we verify that
\begin{eqnarray*}
\int_{\mathrm{B}^{+}_{1}}|f_{k}|^{n}dx=\int_{\mathrm{B}^{+}_{\rho^{k}} (y)}|f(z)|^{n}dz\leq \|f\|_{L^{n}(\mathrm{B}^{+}_{1})}^{n}\leq \eta^{n}.
\end{eqnarray*}
Additionally, we also see that
$$
  [\beta_k]_{0,\alpha,\mathrm{T}_{1}}=\rho^{k\alpha}[\beta]_{0,\alpha,\mathrm{T}_{\rho^{k}}(y)}\leq [\beta]_{0,\alpha,\mathrm{T}_{1}}<\infty,\\
$$
and
$$
[\gamma_k]_{0,\alpha,\mathrm{T}_{1}} = \rho^{k(1+\alpha)}[\gamma]_{0,\alpha, \mathrm{T}_{\rho^{k}}(y)}\leq [\gamma]_{0,\alpha,\mathrm{T}_{1}}<\infty,
$$
since $\beta,\gamma\in C^{0,\alpha}(\overline{ \mathrm{T}_{1}})$. Moreover, we also have that
\begin{eqnarray*}
[g_{k}]_{0,\alpha,\mathrm{T}_{1}}&\leq& \rho^{k\alpha}[g]_{0,\alpha,\mathrm{T}_ {\rho^{k}}(y)}+\rho^{k\alpha}\|\mathbf{b}_{k}\|[\beta]_{0,\alpha,\mathrm{T}_{\rho^{k }}(y)}+\rho^{k\alpha}\|\mathfrak{l}_{k}\|_{L^{\infty}(\overline{\mathrm{T}_{\rho^{k}}(y)})}[\gamma ]_{0,\alpha,\mathrm{T}_{\rho^{k}}(y)}+\\
&+& \rho^{k\alpha}\|\gamma\|_{L^{\infty}(\mathrm{T}_{\rho^{k}}(y))}[\mathfrak{l}_{k}(y+\rho^{k}\cdot)]_{0,\alpha,\mathrm{T}_{1}}\\
&\leq&\|g\|_{C^{0,\alpha}(\overline{\mathrm{T}_{1}})}+\rho^{k\alpha}\|\mathbf{b}_{k}\|\|\beta\|_{C^{0,\alpha}(\overline{\mathrm{T}_{1}})}+\|\mathfrak{l}_{k}\|_{L^{\infty}(\overline{\mathrm{T}_{\rho^{k}}(y)})}\|\gamma\|_{C^{0,\alpha}(\overline{\mathrm{T}_{1}})}+\\
&+& (\text{diam}(\mathrm{T}_{\frac{1}{2}}))^{1-\alpha}\rho^{k}\|\mathbf{b}_{k}\|\|\gamma\|_{C^{0,\alpha}(\overline{\mathrm{T}_{1}})}<\infty,
\end{eqnarray*}
since  $\beta,\gamma,g\in C^{0,\alpha}(\overline{\mathrm{T}_{1}})$, $\|\mathfrak{l}_{k}\|_{L^{\infty}(\overline{\mathrm{T}_{\rho^{k}}(y)})}\leq \frac{3}{2}$ (by \eqref{estimate1theorem4.2.5}), and $\|\mathbf{b}_{k}\|\leq \|\mathbf{b}_{1}\|+\mathrm{C}(k-1)$ (by \eqref{estimate2theorem4.2.5}). Thus,
$$
\rho^{k\alpha}\|\mathbf{b}_{k}\|=\text{o}(k) \quad \text{as} \quad  k\to\infty.
$$

Hence, by the induction hypothesis \eqref{estimate1theorem4.2.5}, we have  $\|v_{k}\|_{L^{\infty}(\mathrm{B}^{+}_{1}) }\leq 1$. Then, we fall into the hypotheses of Lemma  \ref{compII}. Thus, we can find an affine function $\tilde{\mathfrak{l}}(x)=\tilde{\mathbf{a}}+\tilde{\mathbf{b}} \cdot x$ such that
\begin{eqnarray}\label{estimate3theorem4.2.5}
\sup_{\mathrm{B}^{+}_{\rho}}|v_{k}-\tilde{\mathfrak{l}}| \leq \rho,
\end{eqnarray}
where $|\tilde{\mathbf{a}}|, |\tilde{\mathbf{b}}| \le \mathrm{C}(\verb"universal")$. Now, by defining
$$
\mathbf{a}_{k+1}=\mathbf{a}_{k}+\rho^{k}\tilde{\mathbf{a}} \quad \text{and} \quad  \mathbf{b}_{k+1}=\mathbf{b}_{k}+\tilde{\mathbf{b}},
$$
we can see, based on the universal bounds of the constants $\tilde{\mathbf{a}}$ and $\tilde{\mathbf{b}}$, that
\begin{eqnarray*}
|\mathbf{a}_{k+1}-\mathbf{a}_{k}|\leq \mathrm{C}(\verb"universal")\rho^{k} \ \ \mbox{and} \ \ \|\mathbf{b}_{k+1}-\mathbf{b}_{k}\|\leq \mathrm{C}(\verb"universal").
\end{eqnarray*}
Thus, putting $\mathfrak{l}_{k+1}(x)=\mathbf{a}_{k+1}+\mathbf{b}_{k+1}\cdot (x-y)$, we can rescale the inequality  \eqref{estimate3theorem4.2.5},
\begin{eqnarray*}
\sup_{\mathrm{B}^{+}_{\rho^{k+1}}(y)}|u-\mathfrak{l}_{k+1}|\leq \rho^{k+1},
\end{eqnarray*}
thereby completing the $(k+1)^{\text{th}}-$step of induction.

Now, note that the sequence $(\mathbf{a}_{k})_{k\in \mathbb{N}}$ is Cauchy in $\mathbb{R}$, and consequently, there exists $\displaystyle \mathbf{a}_{\infty}=\lim_{k \to \infty}\mathbf{a}_{k}$.  On the other hand, analogous to the proof of Theorem  \ref{holderoptimal}, the  sentence \eqref{estimate2theorem4.2.5} implies the following rate of convergence for all $k\in\mathbb{N}$
\begin{eqnarray}\label{estimate4theorem4.2.5}
|u(y)-\mathbf{a}_{k}|\leq \frac{\mathrm{C}}{1-\rho}\rho^{k}.
\end{eqnarray}
Moreover, once again by \eqref{estimate2theorem4.2.5}, putting $\mathbf{b}_{0}=0$, we have for all $k\in\mathbb{N}$,
\begin{eqnarray}\label{estimate5theorem4.2.5}
\|\mathbf{b}_{k}\|\leq \sum_{j=0}^{k-1}\|\mathbf{b}_{j+1}-\mathbf{b}_{j}\|\leq \mathrm{C}k.
\end{eqnarray}

It is worth noting that in the above construction, we have no guarantee about the convergence of the sequence $(\mathbf{b}_{k})_{k\in\mathbb{N}}$, and thus, such a sequence could not be convergent

Finally, given $r\in(0,\rho)$, with $\rho \le e^{-1}$, we can find $k\in\mathbb{N}$ such that $\rho^{k+1}<r\leq \rho^{k} $. Hence, by \eqref{estimate1theorem4.2.5}, \eqref{estimate4theorem4.2.5} and \eqref{estimate5theorem4.2.5}, we obtain
\begin{eqnarray}\label{estimate6theorem4.2.5}
\sup_{x\in\mathrm{B}^{+}_{r}(y)}|u(x)-u(y)|&\leq& \sup_{\mathrm{B}^{+}_ {r}(y)}|u-\mathfrak{l}_{k}|+|u(y)-\mathbf{a}_{k}|+\|\mathbf{b}_{k}\|\sup_{x\in\mathrm{B}^{+}_ {r}(y)}|x-y|\nonumber\\
&\leq& \sup_{\mathrm{B}^{+}_{\rho^{k}}(y)}|u-\mathfrak{l}_{k}|+\frac{\mathrm{C}}{1-\rho}\rho^{k}+\mathrm{C}kr\leq \rho^{k}+\frac{\mathrm{C}}{1-\rho}\rho^{k}+\mathrm{C}k\rho^{k}\nonumber\\
&=&\left(1+\frac{\mathrm{C}}{1-\rho}\right)\rho^{k}+\mathrm{C}k\rho^{k}\leq\mathrm{C}(\rho^{k}+k\rho^{k})=\frac{\mathrm{C}}{\rho}\left(\frac{1}{k}+1\right)k\rho^{k}\nonumber\\
&\stackrel{r>\rho^{k+1}}{\leq}& \mathrm{C}kr\leq \mathrm{C}r\frac{\ln (r)}{\ln (\rho)}= \frac {-\mathrm{C}}{-\ln (\rho)}r\ln(r)=-\mathrm{C}r\ln(r),
\end{eqnarray}
where we must remember that $t\longmapsto \ln (t)$ is an increasing function, that is, $\frac {\ln (r)}{\ln (\rho)}\geq k$, since $\rho<\frac{1}{2}$, and thus $\log\rho<0$.

Now, we prove that $u\in C^{0,\text{Log-Lip}}\left(\overline{\mathrm{B}^{+}_{\frac{1}{2}}}\right)$. Indeed, let $x\in\mathrm{B}^{+}_{\frac{1}{2}}$ and $y\in\mathrm{T}_{\frac{1}{2}}$. We have two cases to analyze:
\begin{enumerate}
\item[\checkmark] {\bf Case 1:} $e^{-1} \ge r^{\prime}=|x-y|\geq\rho$.\\

In this case, $r^{\prime}\ln ({r^{\prime}}^{-1})\geq \rho\ln (\rho^{-1})$, and thus by inequality $\ln(\rho^{-1})\geq 1-\rho$ we obtain that
\begin{eqnarray*}
\frac{|u(x)-u(y)|}{|x-y|\ln(|x-y|^{-1})}\leq \frac{\overbrace{2\|u\|_{L^{\infty}(\overline{\mathrm {B}^{+}_{1/2}})}}^{\leq 2}}{\rho\ln(\rho^{-1})}\leq \frac{2}{\rho(1-\rho)}=\mathrm{C}.
\end{eqnarray*}

\item[\checkmark] {\bf Case 2:} $r^{\prime}=|x-y|< \rho$.\\

Note that the estimate \eqref{estimate6theorem4.2.5} also holds in $\overline{\mathrm{B}^{+}_{r^{\prime}}(y)}$, and hence as $x\in\overline{\mathrm{B}^{+}_ {r^{\prime}}(y)}$, it follows from such an estimate that
\begin{eqnarray*}
\frac{|u(x)-u(y)|}{|x-y|\ln(|x-y|^{-1})}\leq\frac{-\mathrm{C} r^{\prime}\log r^{\prime}}{|x-y|\log|x-y|^{-1} }=\mathrm{C}.
\end{eqnarray*}
\end{enumerate}

Therefore, from the above cases, and \cite[Theorem 2]{ET}, we can conclude that $u\in C^{0,\text{Log-Lip}}\left(\overline{\mathrm{B}^{+}_{\frac{1}{2}}}\right)$, and the desired estimate holds.
\end{proof}

\section{Optimal gradient regularity}

In this section, we will develop the study of the optimal regularity of solutions for $\eqref{1}$ when $f \in L^{p}(\mathrm{B}^{+}_{1})$ for $n<p<\infty$, and the boundary data are regular enough (to be clarified \textit{a posteriori}). The regularity estimates addressed in this scenario will be $C^{1,\nu}$, where $\nu\in(0,1)$ is an optimal constant depending on the optimal regularity exponent $C^{1,\alpha}$ of the homogeneous problem with frozen coefficients and oblique boundary conditions (see, \cite[Theorem 1.2]{LiZhang} for details), as well as the range of integrability of the source term.

Now, we will explain a little more about the constant $\nu$. Such a constant will be
\begin{eqnarray*}
\nu \defeq \min\left\{\alpha^{-},\frac{p-n}{p}\right\}
\end{eqnarray*}
where the notation above on $\nu$ should be understood as follows
\begin{eqnarray*}
\left\{
\begin{array}{rlcl}
\mbox{If} & \frac{p-n}{p}<\alpha &\mbox{then}& \ u\in C^{1,\frac{p-n}{p}}    \\
\mbox{If} &  \frac{p-n}{p}\geq \alpha &\mbox{then}& \ u\in C^{1,\varsigma} \ \mbox{for any} \ 0<\varsigma<\alpha\\
	\end{array}
	\right.
\end{eqnarray*}
With such observations in mind, we establish the following first step of an affine approximation scheme of $(1+\overline{\alpha})$-order.

\begin{proposition}
\label{compIII}
Let $u$ be a normalized viscosity solution of \eqref{1} for $f \in L^p(B_1^{+})$ for some $p \in (n, \infty)$, where $\beta$, $\gamma$, $g\in C^{0,\alpha}(\overline{\mathrm{T}_{1}} )$ to some $\alpha\in(0,1)$. Given $\overline{\alpha}\in (0,\alpha)$, there exist $\eta>0$ and $\rho\in\left(0,\frac{1}{2}\right]$, depending only on $n$, $p$, $\lambda$, $\Lambda$, $\mu_{0}$, $\alpha$, $\overline{\alpha}$, $\|\beta\|_{C^{0,\alpha }(\overline{\mathrm{T}_{1}})}$, $\|\gamma\|_{C^{0,\alpha}(\overline{\mathrm{T}_{1}} )}$ and $\|g\|_{C^{0,\alpha}(\overline{\mathrm{T}_{1}})}$ such that, if
$$
\max\left\{  \left(\intav{\mathrm{B}^{+}_{1}} |\Phi_{F}(x)|^{n}dx\right)^{1/n}, \,\, \left(\int_{\mathrm{B}^{+}_{1}}|f(x)|^{p}dx\right)^{1/p}\right\}\leq \eta,
$$
 then there exists an affine function $\mathfrak{l}(x)=\mathbf{a}+\mathbf{b}\cdot x$ with universally bounded coefficients, in the following sense
$$
|\mathbf{a}|+\|\mathbf{b}\|\leq \mathrm{C}\left(n,\lambda,\Lambda,\mu_{0},\alpha,\|\beta\|_{C^{0,\alpha}(\overline{ \mathrm{T}_{1}})},\|\gamma\|_{C^{0,\alpha}(\overline{\mathrm{T}_{1}})},\|g\|_{C^{0,\alpha}(\overline{\mathrm{T}_{1}})}\right),
$$
such that
$$
\sup_{\mathrm{B}^{+}_{\rho}}|u-\mathfrak{l}|\leq \rho^{1+\overline{\alpha}}.
$$
\end{proposition}

\begin{proof}
The proof follows the same lines as the Proposition \ref{compII}. First, we set $\delta>0$, which we will determine \text{a posteriori}. By Lemma \eqref{approx},  we can consider $\mathfrak{h}$ as the viscosity solution to
\begin{eqnarray*}
\left\{
\begin{array}{rclcl}
F(D^{2}\mathfrak{h},0)& = & 0 &\mbox{in}& \mathrm{B}^{+}_{\frac{7}{8}} \\
\beta\cdot D\mathfrak{h}+\gamma \mathfrak{h} & = & g &\mbox{on}& \mathrm{T}_{\frac{7}{8}}\\
\mathfrak{h} & = & u
&\mbox{on}& \partial \mathrm{B}^{+}_{\frac{7}{8}}\setminus \mathrm{T}_{\frac{7}{8}},
	\end{array}
	\right.
\end{eqnarray*}
such that
\begin{eqnarray}\label{estimate1lemma4.3.1}
	\sup_{\mathrm{B}^{+}_{\frac{7}{8}}}|u-\mathfrak{h}|\leq \delta.
\end{eqnarray}
Moreover, by the Theorem \cite[Theorem 1.2]{LiZhang}, it follows that $\mathfrak{h}\in C^{1,\alpha}(\overline{\mathrm{B}^{+}_{\frac{2}{3}} })$, and
\begin{eqnarray}\label{estimate2lemma4.3.1}
\|\mathfrak{h}\|_{C^{1,\alpha}\left(\overline{\mathrm{B}^{+}_{\frac{2}{3}}}\right)}\leq\mathrm{C}=\mathrm{C}\left(n,\lambda,\Lambda,\mu_{0},\alpha,\|\beta\|_{C^{0,\alpha}( \overline{\mathrm{T}_{1}})},\|\gamma\|_{C^{0,\alpha}(\overline{\mathrm{T}_{1}})},\ |g\|_{C^{0,\alpha}(\overline{\mathrm{T}_{1}})}\right).
\end{eqnarray}
In this context, we can  define $\mathbf{a}=\mathfrak{h}(0)$ and $\mathbf{b}=D\mathfrak{h}(0)$. Hence, by \eqref{estimate2lemma4.3.1}, it follows that
\begin{eqnarray}\label{estimate3lemma4.3.1}
\sup_{\mathrm{B}^{+}_{r}}|\mathfrak{h}-\mathfrak{l}|\leq \mathrm{C}r^{1+\alpha}, \,\,\forall r\in \left(0,\frac{2}{3}\right).
\end{eqnarray}
Finally, we may choice $\rho$ and $\delta$ in such a way
\begin{equation}\label{EqRadius_3}
  \rho \defeq \min\left\{\left(\frac{1}{2\mathrm{C}}\right)^{\frac{1}{\alpha-\overline{\alpha}}}, \frac{1}{2}\right\} \ \ \mbox{and } \ \ \delta \defeq \frac{1}{2}\rho^{1+\overline{\alpha}}.
\end{equation}
Note that such choices determine the constant $\eta>0$ from Lemma \eqref{approx}. Thus, the universal bound of the constants $\mathbf{a}$ and $\mathbf{b}$ follows directly from \eqref{estimate2lemma4.3.1}.

In conclusion, from \eqref{estimate1lemma4.3.1}, \eqref{EqRadius_3} and \eqref{estimate3lemma4.3.1}, we obtain
\begin{eqnarray*}
\sup_{\mathrm{B}^{+}_{\rho}}|u-\mathfrak{l}|&\leq& \sup_{\mathrm{B}^{+}_{\rho}}|u-\mathfrak{h}|+\sup_{\mathrm{B}^{+}_{\rho}}|\mathfrak{h}-\mathfrak{l}|\leq \sup_{\mathrm{B}^{+}_{\frac{7}{8}}}|u-\mathfrak{h}|+\sup_{ \mathrm{B}^{+}_{\rho}}|\mathfrak{h}-\mathfrak{l}|\leq \delta+\mathrm{C}\rho^{1+\alpha}\\
&\leq&\frac{1}{2}\rho^{1+\overline{\alpha}}+\left(\frac{1}{2\rho^{\alpha-\overline{\alpha }}}\right)\rho^{1+\alpha}=\frac{1}{2}\rho^{1+\overline{\alpha}}+\frac{1}{2}\rho^{1+\overline{\alpha}}=\rho^{1+\overline{\alpha}},
\end{eqnarray*}
which finishes the proof of the Lemma.
\end{proof}

Finally, we are in a position to address the main result of this section.

\begin{theorem}\label{Holderoptimalgradientestimate}
Let $u$ be a viscosity solution of \eqref{1}, where $\beta,\gamma,g\in C^{0,\alpha}(\overline{\mathrm{T}_{1}})$, $f\in L^{p}(\mathrm{B}^{+}_{1})\cap C^{0}(\mathrm{B}^{+}_{1})$ for $ p \in (n, \infty)$, and
\begin{eqnarray*}
\nu=\min\left\{\alpha^{-},\frac{p-n}{p}\right\}
\end{eqnarray*}
Then, there exists a positive constant $\eta_{0}$, depending only on $n$, $\lambda$, $\Lambda$, $\mu_{0}$, $p$, $\alpha$, $\ |\beta\|_{C^{0,\alpha}(\overline{\mathrm{T}_{1}})}$, and $\|\gamma\|_{C^{0,\alpha} (\overline{\mathrm{T}_{1}})}$ such that if,
\begin{eqnarray}\label{hypothesis1theorem4.3.2}
\intav{\mathrm{B}^{+}_{r}}|\Phi_{F}(x,y)|^{n}dx\leq \eta_{0}^{n}, \,\, \forall y\in\mathrm{B}^{+}_{\frac{1}{2}},\ \forall r\in\left(0,\frac{1}{2}\right),
\end{eqnarray}
then, $u\in C^{1,\nu}\left(\overline{\mathrm{B}^{+}_{\frac{1}{2}}}\right)$. Moreover, the following estimate holds
\begin{eqnarray*}
\|u\|_{C^{1,\nu}\left(\overline{\mathrm{B}^{+}_{\frac{1}{2}}}\right)}\leq\mathrm{C}(\|u\| _{L^{\infty}(\mathrm{B}^{+}_{1})}+\|f\|_{L^{p}(\mathrm{B}^{+}_{1 })}+\|g\|_{C^{0,\alpha}(\overline{\mathrm{T}_{1}})}),
\end{eqnarray*}
where $\mathrm{C}>0$ depends only on $n$, $\lambda$, $\Lambda$, $\mu_{0}$, $p$, $\alpha$, $\| \beta\|_{C^{0,\alpha}(\overline{\mathrm{T}_{1}})}$ and $\|\gamma\|_{C^{0,\alpha}( \overline{\mathrm{T}_{1}})}$.
\end{theorem}

\begin{proof}
As in the previous sections, we can assume, without loss of generality, that
$$
\|u\|_{L^{\infty}(\mathrm{B}^{+}_{1})}\leq 1,\quad  \|f \|_{L^{p}(\mathrm{B}^{+}_{1})}\leq \eta \quad \text{and} \quad  \|g\|_{C^{0,\alpha}(\overline {\mathrm{T}_{1}})}\leq 1,
$$
where $\eta>0$ is the constant from Lemma \ref{compIII} when we put $\eta_{0}=\eta$. Fixed $y\in\mathrm{T}_{\frac{1}{2}}$, we assert that there exists a sequence of affine functions $(\mathfrak{l}_{k})_{k\in\mathbb{N }}$ of the form $\mathfrak{l}_{k}(x)=\mathbf{a}_{k}+\mathbf{b}_{k}(x-y)$, such that
\begin{eqnarray}\label{estimate1theorem4.3.2}
\sup_{\mathrm{B}^{+}_{\rho^{k}}(y)}|u-\mathfrak{l}_{k}|\leq \rho^{k(1+\nu)} \quad \forall\,\, k\in\mathbb{N},
\end{eqnarray}
and
\begin{eqnarray}\label{estimate2theorem4.3.2}
|\mathbf{a}_{k+1}-\mathbf{a}_{k}|\leq \mathrm{C}(\verb"universal")\rho^{k(1+\nu)} \,\, \mbox{and} \,\, \|\mathbf{b}_{k+1}-\mathbf{b}_ {k}\|\leq\mathrm{C}(\verb"universal")\rho^{k\nu} \,\, \forall\,\, k\in\mathbb{N},
\end{eqnarray}
where $\rho$ is the universal radius obtained from Lemma \ref{compIII}.

In fact, we will prove such statements by induction on $k$. The case $k=1$ is precisely the statement of Lemma \ref{compIII}. Now, assuming that it holds for some $k\in\mathbb{N}$, we define the following auxiliary function
\begin{eqnarray*}
v_{k}(x) \defeq \frac{(u-\mathfrak{l}_{k})(y+\rho^{k}x)}{\rho^{k(1+\nu)}}, x\in\mathrm{B}^{+}_{1}\cup\mathrm{T}_{1}.
\end{eqnarray*}
Thus, by the induction hypothesis, namely \eqref{estimate1theorem4.3.2}, it follows that $\|v_{k}\|_{L^{\infty}(\mathrm{B}^{+}_{1}) }\leq 1$. Moreover, $v_{k}$ is a viscosity solution for
\begin{eqnarray*}
\left\{
\begin{array}{rclcl}
F_{k}(D^{2}v_{k},x) & = & f_{k}(x) &\mbox{in}& \mathrm{B}^{+}_{1} \\
\beta_{k}(x)\cdot Dv_{k}(x)+\gamma_{k}(x) v_{k}(x) & = & g_{k}(x) &\mbox{on}& \mathrm{T}_{1},
\end{array}
	\right.
\end{eqnarray*}
where
$$
\left\{
\begin{array}{rcl}
F_{k}(\mathrm{X},x) & \defeq & \rho^{k(1-\nu)}F\left(\frac{1}{\rho^{k(1-\nu) }} \mathrm{X},y+\rho^{k}x\right) \\
f_k(x) & \defeq & \rho^{k(1-\nu)} f(y+\rho^{k}x)\\
\beta_{k}(x) & \defeq & \beta(y+\rho^{k}x)\\
\gamma_{k}(x) & \defeq & \rho^{k}\gamma(y+\rho^{k}x)\\
g_{k}(x)& \defeq &\rho^{-k\nu}( g(y+\rho^{k}x)-\beta(y+\rho^{k}x)\cdot \mathbf{b}_{k }-\gamma(y+\rho^{k}x)\mathfrak{l}_{k}(y+\rho^{k}x)).
\end{array}
\right.
$$

It is easy to check that $F_{k}$ is a $(\lambda,\Lambda)$-elliptic operator, and by the hypothesis \eqref{hypothesis1theorem4.3.2}, we have
\begin{eqnarray*}
\intav{\mathrm{B}^{+}_{1}}|\Phi_{F_{k}}(x)|^{n}dx=\intav{\mathrm{B}^{+}_{ \rho^{k}}(y)}|\Phi_{F}(x,y)|^{n}dx\leq \eta^{n}.
\end{eqnarray*}
Furthermore, it follows by the definition of $f_{k}$ that
\begin{eqnarray*}
	\int_{\mathrm{B}^{+}_{1}}|f_{k}|^{p}dx=\rho^{k[(1-\nu)p-n]}\int_{\mathrm{B}^{+}_{\rho^{k}} (y)}|f(z)|^{p}dz\leq \|f\|_{L^{p}(\mathrm{B}^{+}_{1})}^{p}\leq \eta^{p}.
\end{eqnarray*}
We also see that
$$
[\beta_{k}]_{0,\alpha,\mathrm{T}_{1}}=\rho^{k\alpha}[\beta]_{0,\alpha,\mathrm{T}_{\rho^{k}}(y)}\leq [\beta]_{0,\alpha,\mathrm{T}_{1}}<\infty,$$
$$[\gamma_{k}]_{0,\alpha,\mathrm{T}_{1}}=\rho^{k(1+\alpha)}[\gamma]_{0,\alpha, \mathrm{T}_{\rho^{k}}(y)}\leq [\gamma]_{0,\alpha,\mathrm{T}_{1}}<\infty.
$$
Additionally, by $\alpha> \nu$, it follows that
\begin{eqnarray*}
[g_{k}]_{0,\alpha,\mathrm{T}_{1}}&\leq& \rho^{k(\alpha-\nu)}([g]_{0,\alpha, \mathrm{T}_{\rho^{k}}(y)}+\|\mathbf{b}_{k}\|[\beta]_{0,\alpha,\mathrm{T}_{\rho^{k}} (y)}+\|\mathfrak{l}_{k}\|_{L^{\infty}(\overline{\mathrm{T}_{\rho^{k}}(y)})}[\gamma]_{0,\alpha,\mathrm {T}_{\rho^{k}}(y)}+\\
&+&2^{1-\alpha}\|\mathbf{b}_{k}\|\|\gamma\|_{L^{\infty}(\mathrm{T}_{\rho^{k}}(y))})\\
&\leq&\|g\|_{C^{0,\alpha}(\overline{\mathrm{T}_{1}})}+\rho^{k(\alpha-\nu)}\|\mathbf{b}_{k}\|\|\beta\|_{C^{0,\alpha}(\overline{\mathrm{T}_{1}})}+\|\mathfrak{l}_{k}\|_{L^{\infty}(\overline{\mathrm{T}_{\rho^{k}}(y)})}\|\gamma\|_{C^{0,\alpha}(\overline{\mathrm{T}_{1}})}+\\
&+&2^{1-\alpha}\rho^{k(\alpha-\nu)}\|\mathbf{b}_{k}\|\|\gamma\|_{C^{0,\alpha}(\overline{\mathrm{T}_{1}})}<\infty,
\end{eqnarray*}
since as in the proof of Theorem \eqref{theorem4.2},
$$
\|\mathfrak{l}_{k}\|_{L^{\infty}(\overline{\mathrm{T}_{\rho^{k}}(y)})}\leq \frac{3}{2} \quad \text{and} \quad \rho^{k(\alpha-\nu)}\|\mathbf{b}_{k}\|= \text{o}(k) \quad \text{as} \quad k\to \infty.
$$
Thus, we have ensured that $\beta_{k},\gamma_{k},g_{k}\in C^{0,\alpha}(\overline{\mathrm{T}_{1}})$. Therefore, we can invoke Lemma \ref{compIII} to guarantee the existence of an affine function $\tilde{\mathfrak{l}}(x)=\tilde{a}+\tilde{b}\cdot x$ in such a way that
\begin{eqnarray}\label{estimate3theorem4.3.2}
\sup_{\mathrm{B}^{+}_{\rho}}|v_{k}-\tilde{\mathfrak{l}}|\leq \rho^{1+\nu},
\end{eqnarray}
where $|\tilde{\mathbf{a}}|, \|\tilde{\mathbf{b}}\| \le \mathrm{C}(\verb"universal")$. Now, by defining $$
\mathbf{a}_{k+1} \defeq \mathbf{a}_{k}+\rho^{k(1+\nu)}\tilde{\mathbf{a}} \quad  \text{and} \quad  \mathbf{b}_{k+1} \defeq \mathbf{b}_{k}+\rho^{ k\nu}\tilde{\mathbf{b}},
$$
we see  that
\begin{eqnarray*}
|\mathbf{a}_{k+1}-\mathbf{a}_{k}|\leq \mathrm{C}(\verb"universal")\rho^{k(1+\nu)} \ \ \mbox{and} \ \ \|\mathbf{b}_{k+1}-\mathbf{b}_ {k}\|\leq \mathrm{C}(\verb"universal")\rho^{k\nu}.
\end{eqnarray*}
Now, setting $\mathfrak{l}_{k+1}(x)=\mathbf{a}_{k+1}+\mathbf{b}_{k+1}\cdot (x-y)$, we have by scaling back the inequality \eqref{estimate3theorem4.3.2},
\begin{eqnarray*}
\sup_{\mathrm{B}^{+}_{\rho^{k+1}}(y)}|u-\mathbf{l}_{k+1}|\leq \rho^{(k+1)(1 +\nu)}.
\end{eqnarray*}
which completes the statement for $k+1$.

Now, by \eqref{estimate2theorem4.3.2}, we have that $(\mathbf{a}_{k})_{k\in\mathbb{N}}$ and $(\mathbf{b}_{k})_{k\in\mathbb{N}}$ are Cauchy sequences. Thus, there are limits
$$
\displaystyle \mathbf{a}_{\infty} \defeq \lim_{k \to \infty} \mathbf{a}_{k} \quad  \text{and} \quad  \mathbf{b}_{\infty}=\lim_{k \to \infty } \mathbf{b}_{k}.
$$
As in Theorem \ref{theorem4.2}, it is possible to see that $\mathfrak{a}_{\infty}=u(y)$. Furthermore,
we see that for all $k\in\mathbb{N}$ the following rate of convergence of the sequences $(\mathbf{a}_{k})_{k\in\mathbb{N}}$ and $(\mathbf{b}_{k})_{k\in\mathbb{N}}$ hold
\begin{eqnarray}\label{estimate4theorem4.3.2}
|u(y)-\mathbf{a}_{k}|\leq \frac{\mathrm{C}}{1-\rho^{1+\nu}}\rho^{k(1+\nu)} \ \ \mbox{and} \ \ \|\mathbf{b}_{\infty}-\mathbf{b}_{k}\|\leq \frac{\mathrm{C}}{1-\rho^{\nu}}\rho^{k\nu}.
\end{eqnarray}

Finally, given $r\in(0,\rho)$, we can choose $k\in\mathbb{N}$ so that $\rho^{k+1}<r\leq\rho^{k}$. By \eqref{estimate1theorem4.3.2} and \eqref{estimate4theorem4.3.2}, we obtain
\begin{eqnarray}\label{estimate5theorem4.3.2}
\sup_{x\in \mathrm{B}^{+}_{r}(y)}|u(x)-u(y)-\mathbf{b}_{\infty}(x-y)|&\leq& \sup_{\mathrm{B}^{+}_{r}(y)}|u-\mathfrak{l}_{k}|+|u(y)-\mathbf{a}_{k}|+\nonumber\\
&+&\sup_{x\in \mathrm{B}^{+}_{r}(y)}|(\mathbf{b}_{k}-\mathbf{b}_{\infty})\cdot(x-y)|\nonumber\\
&\leq&\sup_{\mathrm{B}^{+}_{\rho^{k}}(y)}|u-\mathfrak{l}_{k}|+\frac{\mathrm{C}}{1-\rho^{1+\nu}}\rho^{\nu}+\|\mathbf{b}_{k}-\mathbf{b}_{\infty}\|r\nonumber\\
&\leq& \rho^{k(1+\nu)}+\frac{\mathrm{C}}{1-\rho^{1+\nu}}\rho^{\nu}+\frac{\mathrm{C}}{1-\rho^{\nu}}\rho^{k\nu}\rho^{k}\nonumber\\
&\leq&\left(1+\frac{\mathrm{C}}{1-\rho^{1+\nu}}+\frac{\mathrm{C}}{1-\rho^{\nu} }\right)\rho^{k(1+\nu)}\nonumber\\
&\leq&\frac{1}{\rho^{1+\nu}}\left(1+\frac{2\mathrm{C}}{1-\rho^{1+\nu}}\right)\rho^{(k+1)(1+\nu)}\nonumber \\
&\leq&\mathrm{C}r^{1+\nu}.
\end{eqnarray}

In conclusion, as such an estimate is valid for each $y\in\mathrm{T}_{\frac{1}{2}}$ (and being more precise  $\mathbf{b}_{\infty}=\mathbf{b}_{\infty}(y)$), then, from \eqref{estimate5theorem4.3.2}, and interior estimates (see \cite[Section 4]{ET}), it follows that $u\in C^{1,\nu}\left(\overline{\mathrm{B}^{+}_{ \frac{1}{2}}}\right)$ obtained in a similar way as the proof of Theorem \ref{holderoptimal} with the desired estimate.
\end{proof}

\section{ $C^{1, Log-Lip}$ regularity estimates}\label{Section06}

In this section, we will deal with the limiting integrability case, i.e., when the source term $f$ has \textbf{bounded mean oscillation}, see Definition \ref{DefBMO} for details.

In this part, we will work initially with the problem with constant coefficients, that is, with the following problem
\begin{eqnarray}\label{coeficientesconstantes}
\left\{
\begin{array}{rclcl}
F(D^{2}u) & = &  f(x) &\mbox{in}&   \mathrm{B}^{+}_{1} \\
\beta(x)\cdot Du(x)+\gamma(x) u(x) & = & g(x)  &\mbox{on}&  \mathrm{T}_{1},
\end{array}
\right.
\end{eqnarray}
where we will assume the following regularity assumption:

\begin{statement}[{\textbf{(RA)}}]\label{RA}

For any matrix $\mathrm{M}\in Sym(n)$ such that $F(\mathrm{M})=0$, the translated problem
\begin{eqnarray}
\left\{
\begin{array}{rclcl}
F(D^{2}\mathfrak{h}+\mathrm{M})= 0 &\mbox{in}& \mathrm{B}^{+}_{\frac{7}{8}} \\
\beta\cdot D\mathfrak{h}+\gamma \mathfrak{h}=g &\mbox{on}& \mathrm{T}_{\frac{7}{8}},
\end{array}
\right.
\end{eqnarray}
admits solutions $\mathfrak{h}\in C^{2,\tilde{\alpha}}\left(\overline{\mathrm{B}^{+}_{\frac{2}{3}}}\right)\cap C^{0}\left(\mathrm{B}^{+}_{\frac{7}{8}}\cup \mathrm{T}_{\frac{7}{8}}\right)$, for some $\tilde{\alpha}\in(0,\alpha]$ when $\beta,\gamma,g\in C^{1,\alpha}(\overline{ \mathrm{T}_{1}})$, and the following estimate holds
\begin{eqnarray*}
\|\mathfrak{h}\|_{C^{2,\tilde{\alpha}}\left(\overline{\mathrm{B}^{+}_{\frac{2}{3}}}\right)} \leq\mathrm{C}^{\ast}\left(\|\mathfrak{h}\|_{L^{\infty}\left(\mathrm{B}^{+}_{\frac{7}{8}} \right)}+\|g\|_{C^{1,\alpha}\left(\overline{\mathrm{T}_{\frac{7}{8}}}\right)}\right).
\end{eqnarray*}

We emphasize that this assumption holds true whenever we assume that $F$ is a convex operator (cf. \cite[Theorem 1.3]{LiZhang}).
\end{statement}

More precisely, the source term $f\in p$-BMO($\mathrm{B}^{+}_{1}$)$\cap L^{p}(\mathrm{B}^{+}_{1})$ for $p\geq n-\varepsilon_{0}$, where $\varepsilon_{0}$ is the Escauriaza's constant. In this scenario, we can conclude that solutions of \eqref{coeficientesconstantes} are $C^{1, \text{Log-Lip}}$.

Furthermore, using the approximation Lemma \ref{lemma4.4.3}, we can ensure the existence of a quadratic approximation for normalized solutions with a small semi-norm in $p-\text{BMO}$. This constitutes the focus of the next result.

\begin{lemma}\label{lemma4.4.4}
Let $u$ be a normalized viscosity solution of
\begin{eqnarray*}
\left\{
\begin{array}{rclcl}
F(D^{2}u+\tilde{\mathrm{M}}) & = & f(x) &\mbox{in}& \mathrm{B}^{+}_{1} \\
\beta(x)\cdot Du(x)+\gamma(x) u(x) & = & g(x) &\mbox{on}& \mathrm{T}_{1},
\end{array}
\right.
\end{eqnarray*}
where $\beta$, $\gamma$, $g\in C^{1,\alpha}(\overline{\mathrm{T}_{1}})$ for some constant $\alpha\in(0, 1)$, $\tilde{\mathrm{M}}\in Sym(n)$ is such that $F(\tilde{\mathrm{M}})=0$. Suppose further the Statement \eqref{RA} holds. There are $\eta>0$ and $\rho\in\left(0,\frac{1}{2}\right]$, depending only on $n$, $p$, $\lambda$, $\Lambda $, $\mu_{0}$, $\mathrm{C}^{*}$, $\|\beta\|_{C^{1,\alpha}(\overline{\mathrm{T}_{1}})}$, $\|\gamma\|_{C^{1,\alpha}(\overline{\mathrm{T}_{1}})}$ and $\|g\|_{C^{1,\alpha}(\overline{\mathrm{T}_{1}})}$ such that, if
$$
\|f\|_{p-BMO(\mathrm{B}^{+}_{1})}\leq \eta,
$$
for some $p\geq n-\varepsilon$. Then, there exists a quadratic polynomial $\mathfrak{p}(x)=\mathbf{a}+\mathbf{b}\cdot x+\frac{1}{2}x^{t}\mathrm{M}x$ with universally bounded coefficients, in the following sense
$$
|\mathbf{a}|+\|\mathbf{b}\|+\|\mathrm{M}\|\leq \mathrm{C}\left(n,\lambda,\Lambda,\mu_{0},\alpha,\|\beta\|_{C^{1,\alpha}(\overline{\mathrm{T}_{1}})},\|\gamma\|_{C^{1,\alpha}(\overline{\mathrm{T}_{1}})},\|g\|_{C^{1,\alpha}(\overline{\mathrm{T}_{1}})}\right),
$$
such that
$$
\sup_{\mathrm{B}^{+}_{\rho}}|u-\mathfrak{p}|\leq \rho^{2}.
$$
Furthermore, we still have that $F(\mathrm{M}+\tilde{\mathrm{M}})=(f)_{1}$.
\end{lemma}

\begin{proof}
First, we set $\delta>0$, which we will be choose later. By Approximation Lemma II \ref{lemma4.4.3}, we can consider $\mathfrak{h}$, the viscosity solution to
\begin{eqnarray*}
\left\{
\begin{array}{rclcl}
F(D^{2}\mathfrak{h}+\tilde{\mathrm{M}}) & = & (f)_{1} &\mbox{in}& \mathrm{B}^{+}_{\frac{7}{8}} \\
\beta(x)\cdot D\mathfrak{h}(x)+\gamma(x) \mathfrak{h}(x) & = & g(x) &\mbox{on}& \mathrm{T}_{\frac{7}{8}}\\
\mathfrak{h}= u
&\mbox{in}& \partial \mathrm{B}^{+}_{\frac{7}{8}}\setminus \mathrm{T}_{\frac{7}{8}},
\end{array}
\right.
\end{eqnarray*}
such that
\begin{eqnarray}\label{estimate1lemma4.4.4}
\sup_{\mathrm{B}^{+}_{\frac{7}{8}}}|u-\mathfrak{h}|\leq \delta.
\end{eqnarray}
We observe that Statement \eqref{RA} guarantees that the problem
\begin{eqnarray} \label{probobs}
\left\{
\begin{array}{rclcl}
F(D^{2}\mathfrak{h}+\mathrm{M}) & = & \mathrm{c} &\mbox{in}&   \mathrm{B}^{+}_{\frac{7}{8}} \\
\beta\cdot D\mathfrak{h}+\gamma \mathfrak{h} & = & g  &\mbox{on}&  \mathrm{T}_{\frac{7}{8}},
\end{array}
\right.
\end{eqnarray}
also admits $C^{2,\tilde{\alpha}}$ estimates with $\tilde{\mathrm{C}^{\ast}}$ depending only on $\mathrm{C}^{\ast}$ and $|\mathrm{c}|$, for any $\mathrm{M}\in Sym(n)$ such that $F(\mathrm{M})=\mathrm{c}$, for more details see \cite{ET}.

Thus, it follows that $\mathfrak{h}\in C^{2,\tilde{\alpha}}(\overline{\mathrm{B}^{+}_{\frac{2}{3}}})$, and
\begin{eqnarray}\label{estimate2lemma4.4.4}
\|\mathfrak{h}\|_{C^{2,\tilde{\alpha}}\left(\overline{\mathrm{B}^{+}_{\frac{2}{3}}}\right)} \leq\mathrm{C}=\mathrm{C}(n,\lambda,\Lambda,\mu_{0},\tilde{\alpha},\|\beta\|_{C^{0,\alpha} (\overline{\mathrm{T}_{1}})},\|\gamma\|_{C^{0,\alpha}(\overline{\mathrm{T}_{1}})}, \|g\|_{C^{0,\alpha}(\overline{\mathrm{T}_{1}})}).
\end{eqnarray}

Now, we define
$$
\mathbf{a}=\mathfrak{h}(0),\,\, \mathbf{b}=D\mathfrak{h}(0)\,\, \text{and} \,\,\mathrm{M}=D^{2}\mathfrak{h}(0), \quad \text{and} \quad. \mathfrak{p}(x)= \mathbf{a}+\mathbf{b}\cdot x+\frac{1}{2}x^{t}\mathrm{M}x.
$$
Thus, by \eqref{estimate2lemma4.4.4}. it follows that
\begin{eqnarray}\label{estimate3lemma4.4.4}
\sup_{\mathrm{B}^{+}_{r}}|\mathfrak{h}-\mathfrak{p}|\leq \mathrm{C}r^{2+\tilde{\alpha}}, \,\,\forall r\in \left(0, \frac{2}{3}\right).
\end{eqnarray}
In this point, we make the following universal choices of the constants
\begin{equation}\label{EqRadius_4}
  \rho\defeq \min\left\{ \left(\frac{1}{2\mathrm{C}}\right)^{\frac{1}{\tilde{\alpha}}}, \,\frac{1}{2}\right\} \quad \mbox{and} \quad  \delta \defeq \frac{1}{2}\rho^{2}.
\end{equation}
With such choices, the constant $\eta>0$ is determined due to Lemma \ref{lemma4.4.3}. Moreover, the universal bound of the constants $\mathbf{a}$, $\mathbf{b}$ and $\mathrm{M}$, it follows directly from \eqref{estimate2lemma4.4.4}.

Finally, from \eqref{estimate1lemma4.4.4}, \eqref{EqRadius_4} and \eqref{estimate3lemma4.4.4}, we obtain that
\begin{eqnarray*}
\sup_{\mathrm{B}^{+}_{\rho}}|u-\mathfrak{p}|&\leq& \sup_{\mathrm{B}^{+}_{\rho}}|u-\mathfrak{h}|+\sup_{\mathrm{B}^{+}_{\rho}}|\mathfrak{h}-\mathfrak{p}|\leq \sup_{\mathrm{B}^{+}_{\frac{7}{8}}}|u-\mathfrak{h}|+\sup_{ \mathrm{B}^{+}_{\rho}}|\mathfrak{h}-\mathfrak{p}|\leq \delta+\mathrm{C}\rho^{2+\tilde{\alpha}}\\
&\leq&\frac{1}{2}\rho^{2}+\left(\frac{1}{2\rho^{\tilde{\alpha}}}\right)\rho^{2+\tilde{\alpha}}=\frac{1}{2}\rho^{2}+\frac{1}{2}\rho^{2}=\rho^{2},
\end{eqnarray*}
thereby obtaining the desired estimate.
\end{proof}

Different from interior borderline estimates addressed by Teixeira in \cite{ET}, the scenario with oblique boundary conditions in \eqref{1} imposes a substantial challenge in dealing with the tangential derivative. Thus, to overcome such an obstacle, we must suppose a suitable behavior of the data.

Therefore, for the main theorem this section we need the following hypothesis:

\begin{itemize}
\item [({\bf A})]({\bf Regularity of the data}) We assume in the problem \eqref{coeficientesconstantes} that the source term $f$ belongs to $p-BMO(\mathrm{B}^{+}_{1})\cap L^{p}(\mathrm{B}^{+}_{1})\cap C^{ 0}(\mathrm{B}^{+}_{1})$ for $p\in [n-\varepsilon_{0},\infty)$.
\item[({\bf B})]({\bf Regularity of boundary terms})  Also we assume that $\beta,\gamma,g\in C^{1,\alpha}(\overline{\mathrm{T}_{1}})$ and there exist constants $\alpha_{\beta},\alpha_{\gamma}\in (0,\alpha]$ such that
\begin{eqnarray}\label{hypothesisunderbetaandgamma}
\sup_{x,z\in \mathrm{T}_{r}(y)\atop x\neq z}\frac{|D\beta(x)|}{|x-z|^{\alpha}}\leq \mathrm{O}(r^{-\alpha_{\beta}}) \quad \mbox{and} \quad \sup_{x,z\in \mathrm{T}_{r}(y)\atop x\neq z}\frac{|D\gamma(x)|}{|x-z|^{\alpha}} \leq \mathrm{O}(r^{-\alpha_{\gamma}}),\quad \forall y\in \mathrm{T}_{\frac{1}{2}}.
\end{eqnarray}
\end{itemize}

Finally, we are in a position to present the main result of this section.

\begin{theorem}[{\bf Regularity $C^{1,Log-Lip}$}]\label{theorem4.4.5}
Let $u$ be a viscosity solution for \eqref{coeficientesconstantes}. Suppose further the Statement \eqref{RA}, $({\bf A})$ and $({\bf B})$ hold. Then, $u\in C^{1,Log-Lip}\left(\overline{\mathrm{B}^{+}_{\frac{1}{2}}}\right)$. Moreover, the following estimate holds
\begin{eqnarray*}
\displaystyle \sup_{x,y\in \overline{\mathrm{B}^{+}_{\frac{1}{2}}} \atop{x \neq y}} \frac{|u(x)-u(y)-Du(y)\cdot(x-y)|}{|x-y|^{2}\ln|x-y|^{-1}} \leq \mathrm{C}\left(\|u\|_{L^{\infty}(\mathrm{B}^ {+}_{1})}+\|f\|_{p-BMO(\mathrm{B}^{+}_{1})}+\|g\|_{C^{1,\alpha}(\overline{\mathrm{T}_{1}})}\right),
\end{eqnarray*}
where $\mathrm{C}>0$ is a constant depending only on $n$, $\lambda$, $\Lambda$, $\mu_{0}$, $\alpha_{0}$, $\mathrm{C}_{\beta\gamma}$, $p$, $\mathrm{C}^{*}$, $\|\beta\|_{ C^{1,\alpha}(\overline{\mathrm{T}_{1}})}$ and $\|\gamma\|_{C^{1,\alpha}(\overline{\mathrm{T}_{1}})}$.
\end{theorem}

\begin{proof}
We can assume, without loss of generality, that
$$
\|u\|_{L^{\infty}(\mathrm{B}^{+}_{1})}\leq 1,\,\,  \|f\|_{p-\text{BMO}(\mathrm{B}^{+}_{1})}\leq \eta \quad \text{and}  \quad \|g\|_{C^{1,\alpha}(\overline{\mathrm{T}_{1}})}\leq 1,
$$
where $\eta>0$ is the constant from Lemma  \ref{lemma4.4.4}. Now, fixed $y\in\mathrm{T}_{\frac{1}{2}}$, we assert that there exists a sequence of quadratic polynomials $(\mathfrak{p}_{k})_{k\in\mathbb{N} }$ of the form $\mathfrak{p}_{k}(x)=\mathbf{a}_{k}+\mathbf{b}_{k}\cdot (x-y)+\frac{1}{2}(x-y)^{t}\mathrm{M}_{k}(x-y) $ satisfying the following properties:
\begin{enumerate}
\item [(i)] $F(\mathrm{M}_{k})=(f)_{1}$,
\item[(ii)] $\displaystyle\sup_{\mathrm{B}^{+}_{\rho^{k}}(y)}|u-\mathfrak{p}_{k}|\leq \rho^{2k }$,
\item[(iii)] $|\mathbf{a}_{k-1}-\mathbf{a}_{k}|+\rho^{k-1}|\mathbf{b}_{k-1}-\mathbf{b}_{k}|+\rho^{2( k-1)}|\mathrm{M}_{k-1}-\mathrm{M}_{k}|\leq \mathrm{C}\rho^{2(k-1)}$,
\end{enumerate}
for all $k\geq 0$, where $\mathfrak{p}_{-1}=\mathfrak{p}_{0}=\dfrac{1}{2}(x-y)^{t}\mathrm{M}_{0}(x-y)$ for $\mathrm{M}_{0 }\in Sym(n)$ in  such a way that $F(\mathrm{M}_{0})=(f)_{1}$ and $\rho$ is the radius coming from Lemma \ref{lemma4.4.4}.

We will prove such claim via induction on $k$. Note that, the first case, i.e.  $k=0$, it  is clearly satisfied. Now, we assume that the statement holds for some $k$, and we define the following auxiliary function:
\begin{eqnarray*}
v_{k}(x) \defeq \frac{(u-\mathfrak{p}_{k})(y+\rho^{k}x)}{\rho^{2k}}, \ x\in \mathrm{B}^{+}_{1}\cup\mathrm{T}_{1}.
\end{eqnarray*}
Thus, it is easy to check that $v_{k}$ is a viscosity solution of
\begin{eqnarray*}
\left\{
\begin{array}{rclcl}
F(D^{2}v_{k}+\mathrm{M}_{k}) & = & f_{k}(x) &\mbox{in}& \mathrm{B}^{+}_{1} \\
\beta_{k}(x)\cdot Dv_{k}(x)+\gamma_{k}(x) v_{k}(x) & = & g_{k}(x) &\mbox{on}& \mathrm{T}_{1},
\end{array}
\right.
\end{eqnarray*}
where
$$
\left\{
\begin{array}{rcl}
  f_{k}(x) & \defeq & f(y+\rho^{k}x) \\
  \beta_{k}(x) & \defeq & \beta(y+\rho^{k}x) \\
  \gamma_{k}(x) & \defeq & \rho^{k}\gamma(y+\rho^{k}x) \\
  g_{k}(x) & \defeq &\rho^{-k}( g(y+\rho^{k}x)-\beta(y+\rho^{k}x)\cdot D\mathfrak{p}_{k}( y+\rho^{k}x)-\gamma(y+\rho^{k}x)\mathfrak{p}_{k}(y+\rho^{k}x))
\end{array}
\right.
$$

Now, note that, by the induction hypothesis, it follows from (ii)  that $\|v_{k}\|_{L^{\infty}(\mathrm{B}^{+}_{1})}\leq 1 $. Moreover, by the definition of $f_{k}$, we have
\begin{eqnarray*}
\|f_{k}\|_{p-BMO(\mathrm{B}^{+}_{1})}&=&\sup_{x_{0}\in \Omega, r>0} \left (\intav{\mathrm{B}_{r}(x_0) \cap \mathrm{B}^{+}_{1}} |f_{k}(x) - (f_{k})_{x_0 , \rho}|^p dx\right)^{\frac{1}{p}}\\
&=& \sup_{x_{0}\in \mathrm{B}^{+}_{1}, r>0} \left(\intav{\mathrm{B}_{r\rho^{k} }(y+\rho^{k}x_{0}) \cap \Omega} |f(z) - (f)_{y+\rho^{k}x_0, r\rho^{k}}|^p dx\right)^{\frac{1}{p}}.
\end{eqnarray*}
As a result,
\begin{eqnarray}
\|f_{k}\|_{p-\text{BMO}(\mathrm{B}^{+}_{1})}\leq\|f\|_{p-\text{BMO}(\mathrm{B}^{+ }_{1})}\leq \eta.
\end{eqnarray}

Additionally, we observe that $\beta_{k},\gamma_{k}\in C^{1,\alpha}(\overline{\mathrm{T}_{1}})$, since $\beta,\gamma\in C^{1,\alpha}(\overline{\mathrm{T}_{1}})$  (by the hypothesis $({\bf B})$) and $\rho\in\left(0,\frac{1}{2}\right]$. On other the hand, for each $i=1,\cdots, n$ we have that
\begin{eqnarray}\label{estimateofgk}
[D_{i}g_{k}]_{0,\alpha,\mathrm{T}_{1}}&\leq&[D_{i}g]_{0,\alpha, \mathrm{T}_{\rho^{k}}(y)}+[D_{i}\beta(y+\rho^{k}\cdot)\cdot D\mathfrak{p}_{k}(y+\rho^{k}\cdot)]_{0,\alpha,\mathrm{T}_{1}}+\nonumber\\
&+&[\beta(y+\rho^{k}\cdot)\cdot D_{i}(D\mathfrak{p}_{k}(y+\rho^{k}\cdot))]_{0,\alpha,\mathrm{T}_{1}}+[D_{i}\gamma(y+\rho^{k}\cdot)\mathfrak{p}_{k}(y+\rho^{k}\cdot)]_{0,\alpha,\mathrm{T}_{1}}+\nonumber\\
&+&[\gamma(y+\rho^{k}\cdot)D_{i}\mathfrak{p}_{k}(y+\rho^{k}\cdot)]_{0,\alpha,\mathrm{T}_{1}}.
\end{eqnarray}

Next, we will analyze each term in the right-hand of \eqref{estimateofgk} separately. In effect, by $g\in C^{1,\alpha}(\overline{\mathrm{T}_{1}})$, it follows that
$$
[D_{i}g]_{0,\alpha, \mathrm{T}_{\rho^{k}}(y)}\leq \|g\|_{C^{1,\alpha}(\overline{\mathrm{T}_{1}})}<\infty.
$$
Furthermore,
\begin{eqnarray*}
[D_{i}\beta(y+\rho^{k}\cdot)\cdot D\mathfrak{p}_{k}(y+\rho^{k}\cdot)]_{0,\alpha,\mathrm{T}_{1}}&\leq& [D_{i}\beta]_{0,\alpha,\mathrm{T}_{\rho^{k}}(y)}\left(\rho^{k\alpha}\|D\mathfrak{p}_{k}\|_{L^{\infty}(\mathrm{T}_{\rho^{k}}(y))}\right)+\\
&+&2\|D\mathfrak{p}_{k}\|_{L^{\infty}(\mathrm{T}_{\rho^{k}}(y))}\sup_{x,z\in \mathrm{T}_{1}\atop x\neq z}\frac{|D_{i}\beta(y+\rho^{k}z)|}{|x-z|^{\alpha}}\\
&=&[D_{i}\beta]_{0,\alpha,\mathrm{T}_{\rho^{k}}(y)}\left(\rho^{k\alpha}\|D\mathfrak{p}_{k}\|_{L^{\infty}(\mathrm{T}_{\rho^{k}}(y))}\right)+\\
&+&2\|D\mathfrak{p}_{k}\|_{L^{\infty}(\mathrm{T}_{\rho^{k}}(y))}\rho^{k\alpha}\sup_{x,z\in \mathrm{T}_{\rho^{k}}(y)\atop \overline{x} \neq \overline{z}}\frac{|D_{i}\beta(\overline{z})|}{|\overline{x}-\overline{z}|^{\alpha}}\\
&<&\infty,
\end{eqnarray*}
since, by induction hypothesis, see item $(iii)$,
$$
\|D\mathfrak{p}_{k}\|_{L^{\infty}(\mathrm{T}_{\rho^{k}}(y))}\leq \frac{1}{1-\rho}\mathrm{C}+\mathrm{C}\text{o}(k)\quad  \text{as} \quad k\to \infty
$$
and
$$
\displaystyle\sup_{x,z\in \mathrm{T}_{\rho^{k}}(y)\atop \overline{x} \neq \overline{z}}\frac{|D_{i}\beta(\overline{z})|}{|\overline{x}-\overline{z}|^{\alpha}}\leq \mathrm{C}_{\beta}\rho^{-k\alpha_{\beta}} \quad \mbox{for} \quad k\gg 1
$$
by \eqref{hypothesisunderbetaandgamma}.

Similarly, we obtain
\begin{eqnarray*}
[\beta(y+\rho^{k}\cdot)\cdot D_{i}(D\mathfrak{p}_{k}(y+\rho^{k}\cdot))]_{0,\alpha,\mathrm{T}_{1}}\leq \rho^{k\alpha}[\beta]_{0,\alpha,\mathrm{T}_{\rho^{k}}(y)}\|\mathrm{M}_{k}\| \leq 2\|\beta\|_{C^{1,\alpha}(\overline{ \mathrm{T}_{1}})}\text{o}(k), \ \mbox{as} \ k\to \infty,
\end{eqnarray*}
since by item (iii), it follows that $\|\mathrm{M}_{k}\|\leq \mathrm{C}k+\|\mathrm{M}_{0}\|$.

Next, we analyze third term on the right side in \eqref{estimateofgk},
\begin{eqnarray*}
[D_{i}\gamma(y+\rho^{k}\cdot)\mathfrak{p}_{k}(y+\rho^{k}\cdot)]_{0,\alpha,\mathrm{T}_{1}}&\leq&\rho^{k\alpha}[D_{i}\gamma]_{0,\alpha,\mathrm{T}_{\rho^{k}}(y)}\|\mathfrak{p}_{k}\|_{L^{\infty}(\mathrm{T}_{\rho^{k}}(y))}+\\
&+&2\|\mathfrak{p}_{k}\|_{L^{\infty}(\mathrm{T}_{\rho^{k}}(y))}\sup_{x,z\in \mathrm{T}_{1}\atop x\neq z}\frac{|D_{i}\gamma(y+\rho^{k}x)|}{|x-z|^{\alpha}}\\
&\leq&\|\gamma\|_{C^{1,\alpha}(\overline{ \mathrm{T}_{1}})}+2\mathrm{C}_{\gamma}\rho^{k(\alpha-\alpha_{\gamma})}\quad \mbox{for}\quad k\gg 1 \quad \mbox{(by \eqref{hypothesisunderbetaandgamma})}.
\end{eqnarray*}
And finally,
\begin{eqnarray*}
[\gamma(y+\rho^{k}\cdot)D_{i}\mathfrak{p}_{k}(y+\rho^{k}\cdot)]_{0,\alpha,\mathrm{T}_{1}}&\leq&\rho^{k\alpha}[\gamma]_{0,\alpha,\mathrm{T}_{\rho^{k}}(y)}\|D_{i}\mathfrak{p}_{k}\|_{L^{\infty}(\mathrm{T}_{\rho^{k}}(y))}+\\
&+&\rho^{k(1-\alpha)}\|\gamma\|_{C^{1,\alpha}(\overline{\mathrm{T}_{1}})}\mathrm{C}(n)\|\mathrm{M}_{k}\|\\
&\leq& 2\|\gamma\|_{C^{1,\alpha}(\overline{ \mathrm{T}_{1}})}\|D\mathfrak{p}_{k}\|_{L^{\infty}(\mathrm{T}_{\rho^{k}}(y))}+\\
&+&\|\gamma\|_{C^{1,\alpha}(\overline {\mathrm{T}_{1}})}\mathrm{C}(n)\text{o}(k), \ \mbox{as} \ k\to \infty.
\end{eqnarray*}

Therefore, $g_{k}\in C^{1,\alpha}(\overline{ \mathrm{T}_{1}})$. Hence, we are under the hypothesis of Lemma \ref{lemma4.4.4}, and thus, there exists a quadratic polynomial $\tilde{\mathfrak{p}}$ of the form $\tilde{\mathfrak{p}}(x)=\tilde{\mathbf{a}}+\tilde{\mathbf{b}} \cdot x+\frac{1}{2}x^{t}\tilde{\mathrm{M}}x$ such that
$$
|\tilde{\mathbf{a}}| + \|\tilde{\mathbf{b}}\| + \|\tilde{\mathrm{M}}\|\le  \mathrm{C}(\verb"universal")
$$
and
\begin{eqnarray}\label{estimate1theorem4.4.5}
\sup_{\mathrm{B}^{+}_{\rho}}|v_{k}-\tilde{\mathfrak{p}}|\leq\rho^{2}
\end{eqnarray}

Now, by defining
$$
\mathbf{a}_{k+1}=\mathbf{a}_{k}+\rho^{2k}\tilde{\mathbf{a}},\,\, \mathbf{b}_{k+1}=\mathbf{b}_{k}+\rho^{k}\tilde {\mathbf{b}} \quad \text{and} \quad \mathrm{M}_{k+1}=\mathrm{M}_{k}+\tilde{\mathrm{M}},
$$
and
$$
\mathfrak{p}_{k+1}(x) \defeq \mathbf{a}_{k+1} + \mathbf{b}_{k+1}\cdot x + \frac{1}{2}x^t\cdot \mathrm{M}_{k+1}\cdot x.
$$
Then,  by \eqref{estimate1theorem4.4.5}, it follows that
\begin{eqnarray*}
\sup_{\mathrm{B}^{+}_{\rho^{k+1}}(y)}|u-\mathfrak{p}_{k+1}|\leq \rho^{2(k+1)},
\end{eqnarray*}
which establishes the condition (ii) for $k+1$. Furthermore, note that condition (i), it is also guaranteed by Lemma \ref{lemma4.4.4} (see, \textit{e.g.} last statement). Finally,
\begin{eqnarray*}
|\mathbf{a}_{k}-\mathbf{a}_{k+1}|+\rho^{k}|\mathbf{b}_{k}-\mathbf{b}_{k+1}|+\rho^{2k}|\mathrm{M}_{k}-\mathrm{M}_{k+ 1}|\leq \rho^{2k}|\tilde{\mathbf{a}}|+\rho^{k}\rho^{k}|\tilde{\mathbf{b}}|+\rho^{2k}|\tilde{\mathrm{M}} |\leq \mathrm{C}\rho^{2k},
\end{eqnarray*}
thereby guaranteeing the condition (iii) for $k+1$.  This completes  the proof of statement via induction.

Now, note that condition (iii) ensures that the sequences $(\mathbf{a}_{k})_{k\in\mathbb{N}}$ and $(\mathbf{b}_{k})_{k\in\mathbb{N }}$ are Cauchy. Thus, we may consider
$$
\displaystyle \mathbf{a}_{\infty}=\lim_{k \to \infty} \mathbf{a}_{k} \quad \text{and} \quad \mathbf{b}_{\infty}=\lim_{k \to \infty} \mathbf{b}_ {k}.
$$
Moreover, it is easy to see $\mathbf{a}_{\infty}=u(y)$.

On the other hand, from condition (iii), we have the following rate of convergence of the sequences $(\mathbf{a}_{k})$ and $(\mathbf{b}_{k})$
\begin{eqnarray}\label{estimate2theorem4.4.5}
|u(y)-\mathbf{a}_{k}|\leq \frac{\mathrm{C}}{1-\rho^{2}}\rho^{2k} \quad \mbox{and} \quad \|\mathbf{b}_{ \infty}-\mathbf{b}_{k}\|\leq\frac{\mathrm{C}}{1-\rho}\rho^{k}
\end{eqnarray}
for all $k\in\mathbb{N}$. Furthermore, although we have no guarantee of convergence of the sequence $(\mathrm{M}_{k})_{k\in\mathbb{N}}$, observe that condition (iii) still ensures that
\begin{eqnarray}\label{estimate3theorem4.4.5}
\|\mathrm{M}_{k}\|\leq \mathrm{C}k \quad \text{for all} \,\,\,k\in\mathbb{N}.
\end{eqnarray}
Therefore, fixing $r\in(0,\rho)$ (thus $\rho\leq {1/2} <\sqrt{1/e}$), we can choose $k\in\mathbb{N}$ in such a way that $\rho^{k+1}<r\leq\rho^{k}$. Thus, by (ii), \eqref{estimate2theorem4.4.5} and \eqref{estimate3theorem4.4.5}, we get that
\begin{eqnarray*}
\sup_{x\in\mathrm{B}^{+}_{r}(y)}|u(x)-u(y)-b_{\infty}\cdot(x-y)|&\leq&\sup_{x \in \mathrm{B}^{+}_{r}(y)}|u-\mathfrak{p}_{k}|+|u(y)-\mathbf{a}_{k}|+\nonumber\\
&+&\sup_{x\in\mathrm{B}^{+}_{r}(y)}|(\mathbf{b}_{k}-\mathbf{b}_{\infty})\cdot(x-y)|+\sup_{x \in \mathrm{B}^{+}_{r}}|\mathrm{M}_{k}(x-y)\cdot (x-y)|\nonumber\\
&\leq&\sup_{\mathrm{B}^{+}_{\rho^{k}}(y)}|u-\mathbf{p}_{k}|+\frac{\mathrm{C}}{1-\rho^{2 }}\rho^{2k}+\|\mathbf{b}_{k}-\mathbf{b}_{\infty}\|r+\|\mathrm{M}_{k}\|r^{2}\nonumber\\
&\leq& \rho^{2k}+\frac{\mathrm{C}}{1-\rho^{2}}\rho^{2k}+\frac{\mathrm{C}}{1-\rho }\rho^{k}r+\mathrm{C}kr^{2}\nonumber
\end{eqnarray*}
Finally, by using $r\le \rho^{k}$, it follows that
\begin{eqnarray}\label{estimate4theorem4.4.5}
\sup_{x\in\mathrm{B}^{+}_{r}(y)}|u(x)-u(y)-\mathbf{b}_{\infty}\cdot(x-y)|&\leq&\rho^{ 2k}+\frac{\mathrm{C}}{1-\rho^{2}}\rho^{2k}+\frac{\mathrm{C}}{1-\rho}\rho^{2k} +\mathrm{C}k\rho^{2k}\nonumber\\
&\leq&\left(1+\frac{\mathrm{C}}{1-\rho^{2}}+\frac{\mathrm{C}}{1-\rho}\right)\rho^{ 2k}+\mathrm{C}k\rho^{2k}\nonumber\\
&\leq& \mathrm{C}(\rho^{2k}+k\rho^{2k})=\frac{\mathrm{C}}{\rho^{2}}\left(\frac{1} {k}+1\right)k\rho^{2(k+1)}\nonumber\\
&\leq&\mathrm{C}k\rho^{2(k+1)}\leq -\mathrm{C}r^{2}\ln(r).
\end{eqnarray}

Moreover, as $y\in\mathrm{T}_{\frac{1}{2}}$ is arbitrary, 
it follows from \eqref{estimate4theorem4.4.5} that $u\in C^{1,Log-Lip}\left(\overline{\mathrm{B}^{+}_{\frac{1}{2}}}\right)$
with the desired estimate.
\end{proof}

\subsection*{p-BMO estimates for Hessian of solutions}

In this final part, we will revisit the previous theorem and, to some extent improve it, by assuming a sort of $W^{2, p}$ estimates for solutions of certain class fully nonlinear models of oblique boundary problems.

In effect, following the ideas of Theorem \ref{theorem4.4.5} and \cite[Theorem 1.4]{Bessa}, we can obtain $p$-BMO estimates for the Hessian of solutions to the problem \eqref{coeficientesconstantes} under suitable \textit{a priori} estimates. More precisely, we say that the problem
\begin{equation}\label{EqW2,p}
\left\{
\begin{array}{rclcl}
F(D^{2}\mathfrak{h}) & = & f(x) &\mbox{in}&   \mathrm{B}^{+}_{1} \\
\beta\cdot D\mathfrak{h} & = & g_{0}(x)  &\mbox{on}&  \mathrm{T}_{1},
\end{array}
\right.
\end{equation}
\textbf{enjoys $W^{2,p}$ estimates}, when $f\in L^{p}(\mathrm{B}^{+}_{1})$ (for some $n\leq p<\infty$) and $g_{0}\in C^{1 ,\alpha}(\overline{\mathrm{T}_{1}})$ (for some $\alpha\in(0,1)$), we have that $\mathfrak{h}\in W^{2,p}\left(\overline{\mathrm{B}^{+}_{\frac{1}{2}}}\right)$ with the following estimate
\begin{eqnarray*}
\|\mathfrak{h}\|_{W^{2,p}\left(\overline{\mathrm{B}^{+}_{\frac{1}{2}}}\right)}\leq\mathrm{C}(\verb"universal")\left(\|\mathfrak{h}\|_{L^{\infty}\left(\mathrm{B}^{+}_{1}\right)}+\|f\|_{L^{p}(\mathrm{B}^{+}_{1})}+\|g\|_{C^{1,\alpha}\left(\overline{\mathrm{T}_{1}}\right)}\right).
\end{eqnarray*}

We must highlight that $W^{2,p}$ estimates are available, for instance, for the class of convex operators (see \cite{BH20}) and asymptotically convex operators (see \cite[Proposition 3.4]{Bessa} for more details).

Taking these conditions into account, we have the following result.

\begin{theorem}[{\bf p-BMO estimates for Hessian}]

Let $u$ be a viscosity solution for \eqref{coeficientesconstantes}. Assume the conditions of Theorem \ref{theorem4.4.5} are in force, with $n\leq p<\infty$ and $\gamma=0$. Further suppose that the problem \eqref{coeficientesconstantes} enjoys $W^{2,p}$ estimates. Then, $D^{2}u\in \mbox{p-BMO}\left(\overline{\mathrm{B}^{+}_{\frac{1}{2}}}\right)$ with the following estimate
\begin{eqnarray*}
\|D^2 u\|_{p-\textrm{BMO}\left(\overline{\mathrm{B}^{+}_{\frac{1}{2}}}\right)} \le \mathrm{C}\left(\|u\|_{L^{\infty}(\mathrm{B}^+_1)} +\|f\|_{\textrm{p-BMO}(\mathrm{B}^+_1)}+\|g\|_{C^{1, \alpha}(\overline{\mathrm{T}_1})} \right),
\end{eqnarray*}
where $\mathrm{C}>0$ depends only on $n$, $\lambda$,  $\Lambda$, $p$,  $\mu_0$, $\mathrm{C}_{\beta\gamma}$, $p$, $\mathrm{C}^{*}$, $\|\beta\|_{C^{1, \alpha}(\overline{\mathrm{T}_1})}$, $\|\gamma\|_{C^{1, \alpha}(\overline{\mathrm{T}_1})}$ and $W^{2,p}$ a priori regularity estimates.
\end{theorem}

\begin{proof}
By Theorem \ref{theorem4.4.5}, we observe that there is a sequence of polynomials $(\mathfrak{p}_{k})_{k\in\mathbb{N}}$ of the form $\mathfrak{p}_{k}(x)=\mathbf{a}_{k}+\mathbf{b}_{k}\cdot x+\frac{1}{2}x^{t}\mathrm{M}_{k}x $ such that
\begin{enumerate}
\item [(i)] $F(\mathrm{M}_{k})=(f)_{1}$,
\item[(ii)] $\displaystyle\sup_{\mathrm{B}^{+}_{\rho^{k}}(y)}|u-\mathfrak{p}_{k}|\leq \rho^{2k}$,
\item[(iii)] $|\mathbf{a}_{k-1}-\mathbf{a}_{k}|+\rho^{k-1}|\mathbf{b}_{k-1}-\mathbf{b}_{k}|+\rho^{2( k-1)}|\mathrm{M}_{k-1}-\mathrm{M}_{k}|\leq \mathrm{C}\rho^{2(k-1)}$,
\end{enumerate}
for all $k\geq 0$, where $\mathfrak{p}_{-1}=\mathfrak{p}_{0}=\dfrac{1}{2}x^{t}\mathrm{M}_{0}x$ for $\mathrm{M}_{0}\in Sym(n)$ in  such a way that $F(\mathrm{M}_{0})=(f)_{1}$, and $\rho$ is the radius coming from Lemma \ref{lemma4.4.4}. Now, the auxiliary function
\begin{eqnarray*}
v_{k}(x)=\frac{(u-\mathfrak{p}_{k})(\rho^{k}x)}{\rho^{2k}}
\end{eqnarray*}
is a normalized (by (ii)) viscosity solution of
\begin{eqnarray*}
\left\{
\begin{array}{rclcl}
F(D^{2}v_{k}+\mathrm{M}_{k}) & = & f_{k}(x) &\mbox{in}& \mathrm{B}^{+}_{1} \\
\beta_{k}(x)\cdot Dv_{k}(x)& = & g_{k}(x) &\mbox{on}& \mathrm{T}_{1},
\end{array}
\right.
\end{eqnarray*}
where
$$
\left\{
\begin{array}{rcl}
f_{k}(x) & \defeq & f(\rho^{k}x) \\
\beta_{k}(x) & \defeq & \beta(\rho^{k}x) \\
g_{k}(x) & \defeq & \rho^{-k}( g(\rho^{k}x)-\beta(\rho^{k}x)\cdot D\mathfrak{p}_{k}( \rho^{k}x)),
\end{array}
\right.
$$
with $\beta_{k}, g_{k}\in C^{1,\alpha}(\overline{\mathrm{T}_{1}})$ (see the proof of Theorem \ref{theorem4.4.5}). Thus, from the available $W^{2,p}$ a priori estimates, $v_{k}\in W^{2,p}\left(\overline{\mathrm{B}^{+}_{\frac{1}{2}}}\right)$. Furthermore, given $0<r<\rho$, there exists an integer $k$ such that $\rho^{k+1}<r\leq\rho^{k}$, we obtain
$$
\begin{array}{rcl}
\displaystyle \sup_{r \in (0, 1/2)}\left(\intav{\mathrm{B}^+_{r} } |D^2 u(z) - \mathrm{M}_{k}|^p dz\right)^{\frac{1}{p}} & \le &  \displaystyle \sup_{r \in (0, 1/2)} \left(\frac{\rho^{kn}}{r^n}.\intav{\mathrm{B}^+_{\rho^k} } |D^2 u(z) - \mathrm{M}_k|^p dz\right)^{\frac{1}{p}}\\
& =&  \displaystyle \sup_{r \in (0, 1/2)} \left(\frac{1}{\rho^n|\mathrm{B}^{+}_{1}|}.\int_{\mathrm{B}^+_1 } |D^2 v_k|^p dx\right)^{\frac{1}{p}} \\
& \le & \mathrm{C}(\verb"universal").
\end{array}
$$

Now, recall the general inequality
$$
\displaystyle \intav{\mathrm{B}^{+}_{r} } \left|D^2 u - \intav{\mathrm{B}^{+}_{r}} D^2 u \ dy\right|^p dx \le   2^{p} \displaystyle  \intav{\mathrm{B}^{+}_{r} } |D^2 u - \mathrm{M}_k|^p dx.
$$
In effect, it is straightforward
$$
\begin{array}{rcl}
\displaystyle \left(\intav{\mathrm{B}^{+}_{r} } \left|D^2 u - \intav{\mathrm{B}^{+}_{r}} D^2 u \ dy\right|^p dx\right)^{\frac{1}{p}} & \le &  \displaystyle  \left(\intav{\mathrm{B}^{+}_{r} } |D^2 u - \mathrm{M}_k|^p dx\right)^{\frac{1}{p}}+\left|\intav{\mathrm{B}^{+}_{r} } D^2 u \ dy - \mathrm{M}_k\right|\\
& \le &   \displaystyle  \left(\intav{\mathrm{B}^{+}_{r} } |D^2 u - \mathrm{M}_k|^p dx\right)^{\frac{1}{p}}+  \displaystyle  \intav{\mathrm{B}^{+}_{r} } |D^2 u - \mathrm{M}_k| dx\\
& \le & 2 \displaystyle  \left(\intav{\mathrm{B}^{+}_{r} } |D^2 u - \mathrm{M}_k|^p dx\right)^{\frac{1}{p}},
\end{array}
$$
	
Therefore, by combining the above inequalities
$$
\displaystyle \|D^{2}u\|_{p-BMO\left(\overline{\mathrm{B}^{+}_{\frac{1}{2}}}\right)} \defeq \sup_{r \in (0, 1/2)}\left(\intav{\mathrm{B}^{+}_{r} } \left|D^2 u - \intav{\mathrm{B}^{+}_{r}} D^2 u \ dy\right|^p dx\right)^{\frac{1}{p}} \leq \mathrm{C}(\verb"universal"),
$$
thereby finishing the proof of the theorem.
\end{proof}

\begin{remark}
As a final remark, since viscosity solutions to \eqref{EqW2,p} enjoy a Hessian bound in the $L^p-$ average sense, then using the embedding result from \cite[Lemma 1]{AZBED}, we can obtain $C^{1, \text{Log-Lip}}$ type estimates. In fact, if $u: \mathrm{B}^{+}_{1} \rightarrow \mathbb{R}$ is such that $D_{ij}u \in \text{p-BMO}(\mathrm{B}_{1/2}^{+})$ for all $1 \le i, j \le n$, then, $u \in C^{1, \text{Log-Lip}}(\mathrm{B}_{1/4}^{+})$. Specifically,
{\scriptsize{
$$
\displaystyle \sup_{\rho \in (0, 1/4)} \sup_{z \in \partial \Omega}\sup_{\mathrm{B}_{\rho}(z)\cap\Omega} \frac{|u(x)-[u(z)+D u(z)\cdot (x-z)]|}{\rho^2\ln (\rho^{-1})} \leq \hat{\mathrm{C}}\cdot \left(\|u\|_{L^{\infty}(\Omega)} + \|g\|_{C^{1, \alpha}(\partial \Omega)} + \|f\|_{\textrm{p-BMO}(\Omega)}\right),
$$}}
for some constant $\hat{\mathrm{C}} = \hat{\mathrm{C}}(n, \lambda, \Lambda, p, \mu_0, \|\beta\|_{C^{1, \alpha}(\overline{\mathrm{T}_1})},\|\partial \Omega\|_{C^{1,1}})$.
\end{remark}

\section{Schauder-type estimates}\label{SecSchauder}

In this final section, we will address $C^{2,\alpha}$ estimates for solutions of fully nonlinear elliptic equations, such as \eqref{1}. Our purpose is to ensure, under certain conditions, Schauder-type estimates for the problem \eqref{1}. Such estimates have been well-explored in the context of fully nonlinear elliptic equations. The celebrated works of Evans \cite{Evans} and Krylov in \cite{Kry82} and \cite{Kry83} in the 1980s provided the starting point in this direction for interior $C^{2,\alpha}$ estimates for classical solutions of homogeneous fully nonlinear elliptic PDEs with constant coefficients, expressed in terms of the $C^{2}$ norms.

An essential ingredient for obtaining the Schauder estimates was the assumption that the governing operator of the equation has a convex structure. Advancing along this line, and with the formalization of the viscosity solution concept by Crandall and Lions in \cite{CL1983}, Caffarelli in \cite{Caff1} (see also \cite{CC}) further ensured interior H\"{o}lder estimates for the second derivatives of viscosity solutions to fully nonlinear equations with variable coefficients.

Finally, in the context of regularity for fully nonlinear equations with oblique boundary conditions, we must initially mention the work of Milakis and Silvestre in \cite{MilSil06}, which obtained $C^{2,\alpha}$ estimates for viscosity solutions with Neumann boundary data on flat boundaries when $F$ is convex. Last but not least, Li and Zhang, in \cite{LiZhang}, recently addressed Schauder-type estimates for equations like \eqref{1} in the setting of convex operators and constant coefficients.

Therefore, the purpose of this section will be to extend the recent results addressed by Li and Zhang in \cite{LiZhang} regarding Schauder-type estimates for models like \eqref{1}.

\bigskip

From now on, we will need to define the following function, which measures the oscillation of the coefficients of the operator $F$ around $x_{0}$:
\begin{eqnarray*}
\tilde{\Phi}_{F}(x;x_{0})=\sup_{\mathrm{M}\in \text{Sym}(n)}\frac{|F(\mathrm{M},x)-F(\mathrm{M},x_{0})|}{1+\|\mathrm{M}\|}.
\end{eqnarray*}	
Moreover, for the sake of simplicity, we denote $\tilde{\Phi}_{F}(x)=\Phi_{F}(x;0)$.

For our purpose, it is necessary to make the following assumptions:
\begin{itemize}
	\item [{\bf (\#)}]({\bf }$C^{2,\alpha_0}$ a priori estimates) Given $g_{0}\in C^{1,\alpha_{0}}(\overline{\mathrm{T}_{1}})$, we will assume that the problem
	\begin{eqnarray*}
		\left\{
		\begin{array}{rclcl}
			F(D^{2}\mathfrak{h},0) & = & 0 &\mbox{in}&   \mathrm{B}^{+}_{1} \\
			\beta\cdot D\mathfrak{h}+\gamma \mathfrak{h} & = & g_{0}  &\mbox{on}&  \mathrm{T}_{1},
		\end{array}
		\right.
	\end{eqnarray*}
	satisfies $\mathfrak{h}\in C^{2,\alpha_{0}}\left(\overline{\mathrm{B}^{+}_{\frac{2}{3}}}\right)$  with the following estimate
	\begin{eqnarray*}
		\|\mathfrak{h}\|_{C^{2,\alpha_{0}}\left(\overline{\mathrm{B}^{+}_{\frac{2}{3}}}\right)}\leq\mathrm{C}^{\ast}\left(\|h\|_{L^{\infty}\left(\mathrm{B}^{+}_{1}\right)}+\|g_{0}\|_{C^{1,\alpha_{0}}\left(\overline{\mathrm{T}_{1}}\right)}\right)
	\end{eqnarray*}
	for a universal constant $\mathrm{C}^{\ast}>0$.
\end{itemize}

In the following, we will present an approximation result that, through an iterative process, yields the desired Schauder-type estimates for solutions of \eqref{1}. The proof is inspired by \cite[Lemma 7.9]{CC} and \cite[Lemma 3.5]{BH20}.

\begin{lemma}[{\bf Approximation Lemma III}]\label{lemadeaprox}
	Let $\varepsilon\in(0,1)$, and $u$ be a normalized viscosity solution for the problem \eqref{1}, where $\beta,\gamma,g\in C^{1,\alpha_{0}}(\overline{\mathrm{T}_{1}})$. Further assume that  $\|\tilde{\Phi}_{F}\|_{L^{n}(\mathrm{B}^{+}_{1})}\leq \varepsilon$, and {\bf (\#)} holds. Then, there exist a function $\mathfrak{h}\in C^{2}\left(\overline{\mathrm{B}^{+}_{\frac{3}{4}}}\right)$ and $\varphi\in C^0\left(\overline{\mathrm{B}^{+}_{\frac{3}{4}}}\right)$ such that
	$$
	\|\mathfrak{h}\|_{C^{2}(\overline{\mathrm{B}^{+}_{\frac{3}{4}}})}\leq \mathrm{C} \quad \textrm{and} \quad u-\mathfrak{h}\in \mathcal{S}\left(\frac{\lambda}{n},\Lambda,\varphi\right),
	$$
	such that
	\begin{eqnarray*}		\|u-\mathfrak{h}\|_{L^{\infty}\left(\mathrm{B}^{+}_{\frac{3}{4}}\right)}+\|\varphi\|_{L^{n}\left(\mathrm{B}^{+}_{\frac{3}{4}}\right)}\le \mathrm{C}^{\prime}\left(\varepsilon^{\theta}+\|f\|_{L^{n}\left(\mathrm{B}^{+}_{1}\right)}+\|g\|_{L^{\infty}(\mathrm{B}^{+}_{1})}\right)
	\end{eqnarray*}
	for constants $\mathrm{C}, \mathrm{C}^{\prime}>0$ depending only on $n$, $\lambda$,$\Lambda$, $\mu_{0}$, $\alpha$, $\mathrm{C}^{\ast}$, $\|\beta\|_{C^{1,\alpha_{0}}(\overline{\mathrm{T}_{1}})}$ and $\|\gamma\|_{C^{1,\alpha_{0}}(\overline{\mathrm{T}_{1}})}$, and $\theta\in(0,1)$ depending only on $n$, $\lambda$, $\Lambda$ and $\mu_{0}$.
\end{lemma}

\begin{proof}
Let $\mathfrak{h} \in C^0\left(B^+_{7/8}\right)$ be the viscosity solution to
\begin{eqnarray}\label{problemalimite}
\left\{
\begin{array}{rclcl}
F(D^{2}\mathfrak{h},0) & = & 0 &\mbox{in}&   \mathrm{B}^{+}_{\frac{7}{8}} \\
\mathfrak{h} & = & u &\mbox{on}& \partial \mathrm{B}^{+}_{\frac{7}{8}}\setminus \mathrm{T}_{\frac{7}{8}}\\
\beta\cdot D\mathfrak{h}+\gamma \mathfrak{h} & = & 0  &\mbox{on}&  \mathrm{T}_{\frac{7}{8}}.
\end{array}
\right.
\end{eqnarray}
By hypothesis {\bf(\#)}, it follows that $\mathfrak{h} \in C^{2,\alpha_{0}}$, and for $\vartheta\in\left(0,\frac{7}{8}\right)$ and scaling properties, we have
\begin{eqnarray}\label{estimativa1} \|\mathfrak{h}\|_{L^{\infty}\left(\mathrm{B}^{+}_{\frac{7}{8}-\vartheta}\right)}+\vartheta\|D\mathfrak{h}\|_{L^{\infty}\left(\mathrm{B}^{+}_{\frac{7}{8}-\vartheta}\right)}+\vartheta^{2}\|D^{2}\mathfrak{h}\|_{L^{\infty}\left(\mathrm{B}^{+}_{\frac{7}{8}-\vartheta}\right)}\leq \mathrm{C},
\end{eqnarray}
where $\mathrm{C}=\mathrm{C}(n, \lambda,\Lambda,\mu_{0},\alpha_{0},\mathrm{C}^{\ast},\|\beta\|_{C^{1,\alpha_{0}}(\overline{\mathrm{T}_{1}})},\|\gamma\|_{C^{1,\alpha_{0}}(\overline{\mathrm{T}_{1}})}) >0$. Now, by considering $w=u-\mathfrak{h}$, we see that $w$ satisfies in the viscosity sense
\begin{eqnarray*}
\left\{
\begin{array}{rclcl}
w\in \mathcal{S}\left(\frac{\lambda}{n},\Lambda,\varphi\right) &\mbox{in}&   \mathrm{B}^{+}_{\frac{7}{8}} \\
w= 0 &\mbox{on}& \partial \mathrm{B}^{+}_{\frac{7}{8}}\setminus \mathrm{T}_{\frac{7}{8}}\\
\beta\cdot Dw+\gamma w= 0  &\mbox{on}&  \mathrm{T}_{\frac{7}{8}},
\end{array}
\right.
\end{eqnarray*}
 where $\varphi(x)=f(x)-F(D^{2}\mathfrak{h}(x),x)$. Then, by the  A.B.P. estimate \ref{ABP},
\begin{eqnarray}\label{estimativa2}
\|w\|_{L^{\infty}\left(\mathrm{B}^{+}_{\frac{7}{8}-\vartheta}\right)}&\leq& \|w\|_{\left(\partial\mathrm{B}^{+}_{\frac{7}{8}-\vartheta}\setminus \mathrm{T}_{\frac{7}{8}-\vartheta}\right)}+\mathrm{C}\left(\|\varphi\|_{L^{n}\left(\mathrm{B}^{+}_{\frac{7}{8}-\vartheta}\right)}\right)\nonumber\\
&\leq&\|w\|_{\left(\partial\mathrm{B}^{+}_{\frac{7}{8}-\vartheta}\setminus \mathrm{T}_{\frac{7}{8}-\vartheta}\right)}+\mathrm{C}\Bigg(\|f\|_{L^{n}\left(\mathrm{B}^{+}_{\frac{7}{8}-\vartheta}\right)}+\|F(D^{2}h(\cdot),\cdot)\|_{L^{n}\left(\mathrm{B}^{+}_{\frac{7}{8}-\vartheta}\right)}\Bigg),
\end{eqnarray}
for some constant $\mathrm{C}=\mathrm{C}(n,\lambda,\Lambda,\mu_{0})>0$.

Next, we will study the second term of the estimate \eqref{estimativa2}. In fact, as $\mathfrak{h}$ is a viscosity solution of \eqref{problemalimite}, and by the hypothesis on the oscillation of the coefficients, we obtain in $\mathrm{B}^{+}_{\frac{7}{8}-\vartheta}$
\begin{eqnarray*}
\|F(D^{2}\mathfrak{h}(\cdot),\cdot)\|_{L^{n}\left(\mathrm{B}^{+}_{\frac{7}{8}-\vartheta}\right)}&=&\left(\int_{\mathrm{B}^{+}_{\frac{7}{8}-\vartheta}}|F(D^{2}\mathfrak{h}(x),x)-\underbrace{F(D^{2}\mathfrak{h}(x),0)}_{=0}|^{n}dx\right)^{\frac{1}{n}}\\
&\leq& \left(\int_{\mathrm{B}^{+}_{\frac{7}{8}-\vartheta}}|\tilde{\Phi}_{F}(x)(1+\|D^{2}\mathfrak{h}(x)\|)|^{n}dx\right)^{\frac{1}{n}}\\
&\leq&\left(1+\|D^{2} \mathfrak{h}(x)\|_{L^{\infty}\left(\mathrm{B}_{\frac{7}{8}-\vartheta}\right)}\right) \|\tilde{\Phi}_{F}\|_{L^{n}\left(\mathrm{B}_{\frac{7}{8}-\vartheta}\right)}\\
&\leq& \left(1+\|D^{2}\mathfrak{h}(x)\|_{L^{\infty}\left(\mathrm{B}_{\frac{7}{8}-\vartheta}\right)}\right)\varepsilon.
\end{eqnarray*}
Consequently, by \eqref{estimativa1}, we have
\begin{eqnarray}
\|F(D^{2}h(\cdot),\cdot)\|_{L^{n}\left(\mathrm{B}^{+}_{\frac{7}{8}-\vartheta}\right)}\leq\mathrm{C} \varepsilon(1+\vartheta^{-2})\leq 2\mathrm{C}\varepsilon\vartheta^{-2},
\end{eqnarray}
since $\vartheta\in (0, 1)$, and $0<\mathrm{C}=\mathrm{C}\left(n, \lambda,\Lambda,\mu_{0},\alpha_{0},\mathrm{C}^{*},\|\beta\|_{C^{1,\alpha_{0}}(\overline{\mathrm{T}_{1}})},\|\gamma\|_{C^{1,\alpha_{0}}(\overline{\mathrm{T}_{1}})}\right)$.

On the other hand, by \cite[Theorem 1.1]{LiZhang}, we have that $w\in C^{\alpha^{\prime}}(\overline{\mathrm{B}^{+}_{\frac{7}{8}-\vartheta}})$, for some $\alpha^{\prime}\in(0,1)$ depending only on $n$, $\lambda$, $\Lambda$, and $\mu_{0}$. Therefore, from $w=0$ on $\partial\mathrm{B}^{+}_{\frac{7}{8}}\setminus \mathrm{T}_{\frac{7}{8}}$, it follows that
\begin{eqnarray}\label{estimativa3}
\|w\|_{\left(\partial\mathrm{B}^{+}_{\frac{7}{8}-\vartheta}\setminus \mathrm{T}_{\frac{7}{8}-\vartheta}\right)}&\leq&\left([w]_{\alpha',\overline{\mathrm{B}^{+}_{\frac{7}{8}-\vartheta}}}\right)\vartheta^{\alpha'}\nonumber\\
&\leq& \mathrm{C}\vartheta^{\alpha'}(1+\|f\|_{L^{n}(\mathrm{B}^{+}_{1})}+\|g\|_{L^{\infty}(\mathrm{T}_{1})}),
\end{eqnarray}
where $0<\mathrm{C}=\mathrm{C}(n,\lambda,\Lambda,\mu_{0})$ (here, we used  \cite[Proposition 4.14]{CC}).

Finally, setting $\vartheta=\varepsilon^{\frac{1}{2+\alpha^{\prime}}}$ and $\theta=\frac{\alpha^{\prime}}{2+\alpha^{\prime}}$, and using \eqref{estimativa2} and \eqref{estimativa3},
\begin{eqnarray*}
\|w\|_{L^{\infty}\left(\mathrm{B}^{+}_{\frac{7}{8}-\vartheta}\right)}+\|\varphi\|_{L^{n}\left(\mathrm{B}^{+}_{\frac{7}{8}-\vartheta}\right)}&\leq&\mathrm{C}\Big(\|f\|_{L^{n}(\mathrm{B}^{+}_{1})}+\|g\|_{L^{\infty}(\mathrm{T}_{1})}+\\
&+&\vartheta^{-2}\varepsilon+\vartheta^{\alpha^{\prime}}(1+\|f\|_{L^{n}(\mathrm{B}^{+}_{1})}+\|g\|_{L^{\infty}(\mathrm{T}_{1})})\Big)\\
&\leq&\mathrm{C}^{\prime}\left(\varepsilon^{\theta}+\|f\|_{L^{n}(\mathrm{B}^{+}_{1})}+\|g\|_{L^{\infty}(\mathrm{T}_{1})}\right),
\end{eqnarray*}
for some $\mathrm{C}^{\prime}=\mathrm{C}^{\prime}\left(n, \lambda,\Lambda,\mu_{0},\alpha_{0},\mathrm{C}^{\ast},\|\beta\|_{C^{1,\alpha_{0}}(\overline{\mathrm{T}_{1}})},\|\gamma\|_{C^{1,\alpha_{0}}(\overline{\mathrm{T}_{1}})}\right)>0$.  This concludes the proof of the Theorem.
\end{proof}

To obtain our main result, we will need the following structural assumption:
\begin{itemize}
	\item [{\bf (H1)}]({\bf $C^{2,\alpha_{0}}$ type estimates for translated problems }) Given $g_{0}\in C^{1,\alpha_{0}}(\overline{\mathrm{T}_{1}})$ for some $\alpha_{0}\in(0,1)$, and $\mathrm{M}\in Sym(n)$ such that $F(\mathrm{M},0)=0$, we will assume that the problem	
	\begin{eqnarray*}
		\left\{
		\begin{array}{rclcl}
			F(D^{2}\mathfrak{h}+\mathrm{M},0) & = & 0 &\mbox{in}&   \mathrm{B}^{+}_{1} \\
			\beta\cdot D\mathfrak{h}+\gamma \mathfrak{h} & = & g_{0}  &\mbox{on}&  \mathrm{T}_{1},
		\end{array}
		\right.
	\end{eqnarray*}
	admits solutions $\mathfrak{h}\in C^{2,\alpha_{0}}\left(\overline{\mathrm{B}^{+}_{\frac{2}{3}}}\right)$ with the following estimate
	\begin{eqnarray*} \|\mathfrak{h}\|_{C^{2,\alpha_{0}}\left(\overline{\mathrm{B}^{+}_{\frac{2}{3}}}\right)}\leq\mathrm{C}^{\sharp}\left(\|\mathfrak{h}\|_{L^{\infty}\left(\mathrm{B}^{+}_{1}\right)}+\|g_{0}\|_{C^{1,\alpha_{0}}\left(\overline{\mathrm{T}_{1}}\right)}\right)
	\end{eqnarray*}
	for a universal constant $\mathrm{C}^{\sharp}>0$.

	\item[{\bf (H2)}]({\bf Regularity on the data of the problem \eqref{1}}) Assume that $f\in C^{0,\alpha}(\mathrm{B}^{+}_{1})$ for some $\alpha\in(0,1)$. In addition, assume that $\beta\in C^{1,\alpha_{\beta}}(\overline{\mathrm{T}_{1}})$, $\gamma\in C^{1,\alpha_{\gamma}}(\overline{\mathrm{T}_{1}})$, and $g\in C^{1,\alpha_{g}}(\overline{\mathrm{T}_{1}})$ fulfilling the following relation for the H"{o}lder exponents continuity   $\alpha< \min\{\alpha_{\beta}, \alpha_{\gamma}, \alpha_{g}\}$. Furthermore, assume that $\beta$ and $\gamma$ are such that for every $r\in(0,1)$
	\begin{eqnarray*}
		\|\beta\|_{L^{\infty}(\mathrm{T}_{r})}\leq \mathrm{C}_{\beta}r^{1+\alpha_{\beta}} \quad \mbox{and} \quad \|\gamma\|_{L^{\infty}(\mathrm{T}_{r})}\leq \mathrm{C}_{\gamma}r^{1+\alpha_{\gamma}},
	\end{eqnarray*}
	where $\mathrm{C}_{\beta}$ and $\mathrm{C}_{\gamma}$ are positive constants.
\end{itemize}

The next result is a key tool in establishing our Schauder-type estimates. It provides a geometric decay of $(2+\alpha)$-order (at the origin) for solutions of \eqref{1} , thereby addressing the desired higher regularity estimates.

\begin{theorem}[{\bf Point-wise $C^{2, \alpha}$ estimates}]\label{ThmSchauder}
	Let $u$ be a viscosity solution to
	\begin{eqnarray*}
		\left\{
		\begin{array}{rcl}
			F(D^2u,x) &=& f(x) \quad \mbox{in} \,\,   \mathrm{B}^{+}_{r_{0}}\\
			\beta \cdot Du + \gamma \, u&=& g(x)  \quad \mbox{on} \,\,  \mathrm{T}_{r_{0}}.
		\end{array}
		\right.,
	\end{eqnarray*}
	where $F$ satisfies the assumption {\bf(\text{A1})}, with $F(0,0)=f(0)=0$, and consider $0<\alpha<\min\{\alpha_{0},\alpha_{\beta},\alpha_{\gamma}, \alpha_{g}\}$. Further assume the structural hypotheses {\bf(H1)-(H2)} are in force, and that there exist constants $\mathrm{C}_{0}>0$ and $\mathrm{C}_{1}>0$ such that
	\begin{eqnarray*}
		\left(\intav{\mathrm{B}^{+}_{r}}|\tilde{\Phi}_{F}(x)|^{n}dx\right)^{\frac{1}{n}}\leq \mathrm{C}_{0}\left(\frac{r}{r_{0}}\right)^{\alpha} \quad \mbox{and} \quad \left(\intav{\mathrm{B}^{+}_{r}}|f(x)|^{n}dx\right)^{\frac{1}{n}}\leq \mathrm{C}_{1}\left(\frac{r}{r_{0}}\right)^{\alpha}, \forall r\in(0,r_{0}].
	\end{eqnarray*}
	Then, $u \in C^{2,\alpha}$ at the origin. More precisely, there exists a quadratic polynomial  $\mathfrak{p}$  such that:
	\begin{itemize}
		\item[(i)] $\|u-\mathfrak{p}\|_{L^{\infty}(\overline{\mathrm{B}^{+}_{1}}\cap \mathrm{B}_{r})}\leq \mathrm{C}^{\prime \prime}\left(\frac{r}{r_{0}}\right)^{2+\alpha}$, \, $ \forall \, r\in(0,r_{1}]$.
		\item[(ii)] $\|D\mathfrak{p}(0)\|+\|D^{2}\mathfrak{p}(0)\|\leq \mathrm{C}^{\prime \prime}$.
		\item[(iii)] $\mathrm{C}^{\prime \prime}\leq \mathrm{C}\left(\|u\|_{L^{\infty}(\mathrm{B}^{+}_{1})}+\mathrm{C}_{1}+\|g\|_{C^{1,\alpha}(\overline{\mathrm{T}_{1}})}\right)$,
	\end{itemize}
	where $1<\mathrm{C}=\mathrm{C}(n,\lambda,\Lambda,\mu_{0},\mathrm{C}^{\sharp},\alpha,\|\beta\|_{C^{1,\alpha}(\overline{\mathrm{T}_{1}})},\|\gamma\|_{C^{1,\alpha}(\overline{\mathrm{T}_{1}})})$ and
	$r_{1}=\check{\mathrm{C}}=\mathrm{C}^{-1}r_{0}$
\end{theorem}

\begin{remark}
We must emphasize that the assumptions $F(0,0)=f(0)=0$ are not restrictive. Indeed, we can use the uniform ellipticity and a suitable translation in the source term in such a way that these hypotheses are satisfied. For more details, see \cite[Chapter 8]{CC}.
\end{remark}

\begin{proof}

As addressed in \cite[Lemma 6.3]{LiZhang}, we can assume, without loss of generality, that $g(0)=0$ and $Dg(0)=0$. Moreover, by scaling reasoning (cf. \cite[Theorem 8.1]{CC}), it is enough to prove that there exist constants $\varepsilon\in (0,1)$ and $\delta>0$ depending only on $n$, $\lambda$, $\Lambda$, $\mathrm{C}_{\gamma}$, $\mathrm{C}_{\beta}$, $\mathrm{C}^{\sharp}$, $\alpha$, $\|\beta\|_{C^{1,\alpha}(\overline{\mathrm{T}_{1}})}$ and $\|\gamma\|_{C^{1,\alpha}(\overline{\mathrm{T}_{1}})}$ such that if $u$ is a normalized viscosity  solution of \eqref{1}, $\|g\|_{C^{1,\alpha}(\overline{\mathrm{T}_{1}})}\leq \varepsilon$, and
	\begin{eqnarray}\label{condicao1}
		\left(\intav{\mathrm{B}^{+}_{r}}|\tilde{\Phi}_{F}(x)|^{n}dx\right)^{\frac{1}{n}}\leq \delta r^{\alpha}, \quad  \left(\intav{\mathrm{B}^{+}_{r}}|f(x)|^{n}dx\right)^{\frac{1}{n}}\leq \delta r^{\alpha}, \forall r\in(0,1],
	\end{eqnarray}
	then there exists a quadratic polynomial function $\mathfrak{p}$ satisfying $(i)$ and $(ii)$ with $r_{1}=1$.

	Firstly, let $\rho\in(0,1)$ be a radius such that
	\begin{equation} \label{Esc1}
		\rho  \defeq \min\left\{ \left(\frac{1}{2}\right)^{\frac{1}{\alpha}}, \,\,\frac{49}{64}, \,\, \left(\frac{1}{3\mathrm{C}^{\sharp}}\right)^{\frac{1}{\alpha_{0}-\alpha}}\right\}.
	\end{equation}
	For such a fixed $\rho>0$, we choose $\varepsilon\in(0,1)$, such that
	\begin{equation}\label{Esc2}
		10\mathrm{C}^{\prime}\varepsilon^{\theta}\leq \rho^{2+\alpha},
	\end{equation}
	where $\mathrm{C}^{\prime}$ and $\theta$ are as in Lemma \ref{lemadeaprox}. Now,  we select $\delta>0$ such that
$$
\delta\leq \frac{\varepsilon}{2^{1-\frac{1}{n}}(1+\tilde{\mathrm{C}})\omega^{\frac{1}{n}}_{n}},
$$
 where $\omega_{n}$ denotes the volume of the unit ball $\mathrm{B}_{1} \subset \mathbb{R}^{n}$, and the constant $\tilde{\mathrm{C}}>0$ will be taken in such a way that
 $$
 20\tilde{\mathrm{C}}\max\{\mathrm{C}_{\beta},\mathrm{C}_{\gamma}\}\leq \varepsilon.
 $$

Such choices determine universal parameters in our approach. In such a context, it is enough to prove the following statement:

\begin{statement} For every $k \ge 1$, there exist quadratic polynomials
	\begin{eqnarray*}
		\mathfrak{p}_{k}(x)= \frac{1}{2}x^{t} \cdot \mathrm{M}_{k} \cdot x + \mathbf{b}_k \cdot x + \mathbf{a}_k,
	\end{eqnarray*}
	where $\mathrm{M}_{k} \in \textrm{Sym}(n)$, $\mathbf{b}_k \in \mathbb{R}^n$, and $\mathbf{a}_k \in \mathbb{R}$, such that
	\begin{itemize}
		\item[{\bf (I)}] $F(\mathrm{M}_{k},0)=0$.
		\item[{\bf (II)}] $\|u-\mathfrak{p}_{k}\|_{L^{\infty}(\overline{\mathrm{B}^{+}_{\rho^{k}}})}\leq \rho^{k(2+\alpha)}$.
		\item[{\bf (III)}] $|\mathbf{a}_{k}-\mathbf{a}_{k-1}|+\rho^{k-1}\|\mathbf{b}_{k}-\mathbf{b}_{k-1}\|+\rho^{2(k-1)}\|\mathrm{M}_{k}-\mathrm{M}_{k-1}\|\leq \tilde{\mathrm{C}}\rho^{2(k-1)(2+\alpha)}$,
	\end{itemize}
	where $\mathfrak{p}_{0}\equiv \mathfrak{p}_{-1}\equiv 0$, and $\rho\in(0,1)$ is a universal constant.

\end{statement}

	We will prove such a statement by induction on $k$. For the case $k=0$, we set $\mathfrak{p}_{-1}\equiv \mathfrak{p}_{0}\equiv 0$, $F(0,0)=0$ and $u$ is normalized. Now, suppose that we have already constructed $\mathfrak{p}_{0}, \cdots, \mathfrak{p}_{k}$ satisfying statements $(\mathrm{I})-(\mathrm{III})$. Thus, we must show that there exists a quadratic polynomial $\mathfrak{p}_{k+1}$ satisfying such conditions. For this purpose, we define the auxiliary function given by
	\begin{eqnarray*}
		v_{k}(x)=\frac{(u-\mathfrak{p}_{k})(\rho^{k}x)}{\rho^{k(2+\alpha)}}, \ x\in\mathrm{B}^{+}_{1}\cup \mathrm{T}_{1}.
	\end{eqnarray*}
	Thus, we will prove that $v_k$ satisfies the assumptions of Lemma \ref{lemadeaprox}. In fact, note that by the induction hypothesis, $\|v_{k}\|_{L^{\infty}(\mathrm{B}^{+}_{1})}\leq 1$. Now, note that $v_{k}$ is a viscosity solution of
	\begin{eqnarray*}
		\left\{
		\begin{array}{rcl}
			F_{k}(D^{2}v_{k},x) &=& f_{k}(x) \quad \mbox{in} \,\,   \mathrm{B}^{+}_{1}\\
			\beta_{k} \cdot Dv_{k} + \gamma_{k} v_{k}&=& g_{k}(x)  \quad \mbox{on} \,\,  \mathrm{T}_{1},
		\end{array}
		\right.,
	\end{eqnarray*}
	where
	\begin{eqnarray*}
		\left\{
		\begin{array}{rcl}
			F_{k}(\mathrm{M},x)& \defeq& \frac{1}{\rho^{k\alpha}}\left( F\left(\rho^{k\alpha}\mathrm{M}+\mathrm{M}_k, \rho^{k}x\right)-F(\mathrm{M}_{k},\rho^{k}x)\right),\\
			f_{k}(x)&\defeq& \frac{1}{\rho^{k\alpha}}\left( f(\rho^{k\alpha}x)-F(\mathrm{M}_{k},\rho^{k}x)\right),\\
			\beta_{k}(x) &\defeq& \beta(\rho^{k}x), \\
			\gamma_{k}(x)&\defeq&\rho^{k} \gamma(\rho^{k}x),\\
			g_{k}(x)&\defeq& \frac{1}{\rho^{k(1+\alpha)}}[g(\rho^{k}x)-\beta_{k}(x)\cdot D\mathfrak{p}_{k}(\rho^{k}x)-\gamma(\rho^{k}x)\mathfrak{p}_{k}(\rho^{k}x)].
		\end{array}
		\right.
	\end{eqnarray*}

	Now, by construction, and by the structural hypothesis {\bf (H1)}, we have that $F_{k}(0,x)=0$ for all $x\in\mathrm{B}^{+}_{1}\cup \mathrm{T}_{1}$, and the associated problem to the operator $F_{k}$ also satisfies the condition {\bf (H1)}, with the same constant $\mathrm{C}^{\sharp}$ (in particular, the condition {\bf ($\sharp$)} holds true for $F_{k}$, with the same constant $\mathrm{C}^{\sharp}$). Moreover, as in \cite[Theorem 8.1]{CC}, we can see that
	\begin{eqnarray*}
\tilde{\Phi}_{F_{k}}(x)\leq 2\rho^{-k\alpha}\tilde{\Phi}_{F}(\rho^{k}x)(1+\|\mathrm{M}_{k}\|)
\end{eqnarray*}
Now, by $(\mathrm{III})$, we have that for $i\geq k$
\begin{eqnarray*}
\|\mathrm{M}_{k}\|\leq \frac{\tilde{\mathrm{C}}}{1-\rho^{\alpha}}\leq \tilde{\mathrm{C}}.
\end{eqnarray*}
Hence, by combining the two facts above, it follows that
\begin{eqnarray}\label{condicao2}
\|\tilde{\Phi}_{F_{k}}\|_{L^{n}(\mathrm{B}^{+}_{1})}\leq 2(1+\tilde{\mathrm{C}})\rho^{-k\alpha}\rho^{-k}\|\tilde{\Phi}_{F}\|_{L^{n}(\mathrm{B}^{+}_{\rho^{k}})}\leq 2^{1-\frac{1}{n}}(1+\tilde{\mathrm{C}})\delta \omega_{n}^{\frac{1}{n}}\leq \varepsilon,
\end{eqnarray}
where we used \eqref{condicao1}. Similarly, by using the induction hypothesis $(\mathrm{I})$, we have
\begin{eqnarray}\label{condicao3}
\|f_{k}\|_{L^{n}(\mathrm{B}^{+}_{1})}&\leq& \rho^{-k(1+\alpha)}\left(\|f\|_{L^{n}(\mathrm{B}^{+}_{\rho^{k}})}+(1+\tilde{\mathrm{C}})\|\tilde{\Phi}_{F}\|_{L^{n}(\mathrm{B}^{+}_{\rho^{k}})}\right)\nonumber\\
&\leq&\omega_{n}^{\frac{1}{n}}\delta2^{1-\frac{1}{n}}(1+\tilde{\mathrm{C}})\leq \varepsilon.
	\end{eqnarray}

	Finally, it is clear that $\beta_{k}, \gamma_{k}, g_{k} \in C^{1,\alpha}(\overline{\mathrm{T}_{1}})$, and by the conditions on $\beta$, $\gamma$, and $g$, we have
	\begin{eqnarray}\label{condicao4}
		\|g_{k}\|_{L^{\infty}(\mathrm{T}_{1})}&\leq&\rho^{-k(1+\alpha)}[\|g\|_{L^{\infty}(\mathrm{T}_{\rho^{k}})}+\|\gamma\|_{L^{\infty}(\mathrm{T}_{\rho^{k}})}\|\mathfrak{p}_{k}\|_{L^{\infty}(\mathrm{T}_{\rho^{k}})}+\nonumber\\
		&+&\|\beta\|_{L^{\infty}(\mathrm{T}_{\rho^{k}})}\|D\mathfrak{p}_{k}\|_{L^{\infty}(\mathrm{T}_{\rho^{k}})}] \nonumber\\
		&\leq& \|g\|_{C^{1,\alpha}(\overline{\mathrm{T}_{1}})}+\mathrm{C}_{\gamma}\rho^{k(\alpha_{\gamma}-\alpha)}\|\mathfrak{p}_{k}\|_{L^{\infty}(\mathrm{T}_{\rho^{k}})}+\nonumber\\
		&+&\mathrm{C}_{\beta}\rho^{k(\alpha_{\beta}-\alpha)}\|D\mathfrak{p}_{k}\|_{L^{\infty}(\mathrm{T}_{\rho^{k}})}.
	\end{eqnarray}

	Now, since hypothesis $(\mathrm{III})$ holds, we can bound the $L^{\infty}$-norm of $\mathfrak{p}_{k}$, as well as the corresponding norm of its gradient, and obtain that
	\begin{eqnarray*}
		\|\mathfrak{p}_{k}\|_{L^{\infty}(\mathrm{T}_{\rho^{k}})}\leq \frac{3\tilde{\mathrm{C}}}{1-\rho^{2+\alpha}} \quad \mbox{and} \quad \|D\mathfrak{p}_{k}\|_{L^{\infty}(\mathrm{T}_{\rho^{k}})}\leq \frac{2\tilde{\mathrm{C}}}{1-\rho^{1+\alpha}}.
	\end{eqnarray*}

Thus, it follows from \eqref{condicao4}, the assumption $\|g\|_{C^{1,\alpha}(\overline{\mathrm{T}_{1}})}\leq \varepsilon$, and the choice of the constant $\tilde{\mathrm{C}}$, that the following holds
	\begin{eqnarray}\label{condicao5}
		\|g_{k}\|_{L^{\infty}(\mathrm{T}_{1})}\leq \varepsilon+40\tilde{\mathrm{C}}\max\{\mathrm{C}_{\beta},\mathrm{C}_{\gamma}\}\leq \varepsilon+2\varepsilon=3\varepsilon.
	\end{eqnarray}

	Therefore, under the hypotheses of Lemma \ref{lemadeaprox}, we obtain the existence of constants  $\mathrm{C}^{\prime}$ and $\theta$, and a function $\mathfrak{h}\in C^{2}(\overline{\mathrm{B}^{+}_{\frac{3}{4}}})$ such that
	\begin{eqnarray}\label{condicao6}
		\|v_{k}-\mathfrak{h}\|_{L^{\infty}(\mathrm{B}^{+}_{\frac{3}{4}})}\leq \mathrm{C}^{\prime}(\varepsilon^{\theta}+4\varepsilon)\leq 5\mathrm{C}^{\prime}\varepsilon^{\theta}\leq \frac{1}{2}\rho^{2+\alpha},
	\end{eqnarray}
	where we used the assumption \eqref{Esc2}, as well as the fact that $\varepsilon, \theta\in(0,1)$.

Now, remember that, by Lemma \ref{lemadeaprox}, $\mathfrak{h}$ satisfies, in viscosity sense
	\begin{eqnarray}\label{problemaenvolvendoh}
		\left\{
		\begin{array}{rclcl}
			F_{k}(D^{2}\mathfrak{h},0) & = &  0 &\mbox{in}&   \mathrm{B}^{+}_{\frac{7}{8}} \\
			\mathfrak{h} & = &  v_{k} &\mbox{on}& \partial \mathrm{B}^{+}_{\frac{7}{8}}\setminus \mathrm{T}_{\frac{7}{8}}\\
			\beta_{k}\cdot D\mathfrak{h}+\gamma_{k} \mathfrak{h} & = &  0  &\mbox{on}&  \mathrm{T}_{\frac{7}{8}},
		\end{array}
		\right.
	\end{eqnarray}
	Furthermore, we can check that $F_k$ satisfies the assumption required in {\bf(H1)}, thus
	\begin{eqnarray*}
		\|\mathfrak{h}\|^{\ast}_{C^{2,\alpha_{0}}\left(\overline{\mathrm{B}^{+}_{\frac{49}{64}}}\right)}\leq \mathrm{C}^{\sharp}\|\mathfrak{h}\|_{L^{\infty}\left(\mathrm{B}^{+}_{\frac{7}{8}}\right)}\leq\mathrm{C}^{\sharp}.
	\end{eqnarray*}

	Now, consider $\overline{\mathfrak{p}}(x)= \frac{1}{2}x^{t} \cdot \overline{\mathrm{M}} \cdot x + \overline{\mathbf{b}}\cdot x + \overline{\mathbf{a}}$, where
$$
\overline{\mathbf{a}}=\mathfrak{h}(0),\quad  \overline{\mathbf{b}}=D\mathfrak{h}(0) \quad  \text{and} \quad \overline{\mathrm{M}}=D^{2}\mathfrak{h}(0).
$$
With the choices performed in \eqref{Esc1} and \eqref{Esc2}, we know by Taylor formula with the Lagrange remainder
\begin{eqnarray} \label{condicao7}
\|\mathfrak{h}-\overline{\mathfrak{p}}\|_{L^{\infty}(\mathrm{B}^{+}_{\rho})}&\leq&\frac{1}{2}\rho^{2}\sup_{x, y \in \overline{\mathrm{B}^{+}_{\rho}} \atop{x \neq y}}\frac{\|D^{2}h(x)-D^{2}h(x)\|}{|x-y|^{\alpha_{0}}}\sup_{x, y \in \overline{\mathrm{B}^{+}_{\rho}} \atop{x \neq y}}|x-y|^{\alpha_{0}}\nonumber
\\  &\leq&\frac{1}{2}\mathrm{C}^{\sharp}\left(\frac{64}{49}\right)^{2+\alpha_{0}}\rho^{2+\alpha_{0}}\nonumber\\
&\leq& \frac{3}{2}\mathrm{C}^{\sharp}\rho^{2+\alpha_{0}}\nonumber\\
&\leq& \frac{1}{2}\rho^{2+\alpha}.
\end{eqnarray}

	Therefore, from \eqref{condicao6} and \eqref{condicao7}, it follows that
	\begin{eqnarray}
		\|v_{k}-\overline{\mathfrak{p}}\|_{L^{\infty}(\mathrm{B}^{+}_{\rho})}&\leq& \|v_{k}-\mathfrak{h}\|_{L^{\infty}(\mathrm{B}^{+}_{\rho})} + \|\mathfrak{h}-\overline{\mathfrak{p}}\|_{L^{\infty}(\mathrm{B}^{+}_{\rho})} \nonumber \\ &\le& \frac{1}{2} \rho^{2+\alpha} + \frac{1}{2} \rho^{2+\alpha} = \rho^{2+\alpha}.\label{condicao8}
	\end{eqnarray}
	Hence, in view of \eqref{condicao8}, for all $x \in B^+_{\rho}$, we obtain
	\begin{equation} \label{Esc3}
		\left | \frac{u(\rho^{k}x)-\mathfrak{p}_{k}(\rho^{k}x)}{\rho^{k(2+\alpha)}} -\overline{\mathfrak{p}}(x)  \right | \le \rho^{2+\alpha},
	\end{equation}
	or equivalently, if $y=\rho^k x$, then for any $y \in B^+_{\rho^{k+1}}$,
	$$
	|u(y) - \mathfrak{p}_k(y) - \rho^{k(2+\alpha)} \overline{\mathfrak{p}}(\rho^{-} y)| \le \rho^{(k+1)(2+\alpha)}.
	$$
Now, we	consider the quadratic polynomial function as follows
	$$
\mathfrak{p}_{k+1}(y)=\mathfrak{p}_{k}(y)+\rho^{k(2+\alpha)}\overline{\mathfrak{p}}(\rho^{-k}y).
$$
Thus, by \eqref{Esc3}, we conclude that
	$$
	\|u-\mathfrak{p}_{k+1}\|_{L^{\infty}\left(\overline{B^+_{\rho^{k+1}}}\right)} \le \rho^{(k+1)(2+\alpha)},
	$$
thereby establishing the condition $(\mathrm{II})$ via induction. Moreover, note that, by the construction of such polynomials, the condition $(\mathrm{I})$ is naturally satisfied, as well as the coefficients satisfy the condition $(\mathrm{III})$. This proves the desired statement. With such an assertion, analogously to \cite[Theorem 8.1]{CC}, the proof is completed.
	\end{proof}

In conclusion, by combining the above theorem with \cite[Theorem 8.1]{CC}, we obtain the following result:
\begin{theorem}[{\bf Schauder estimates under oblique boundary conditions}]
	Let $F$ be a uniformly elliptic operator satisfying the condition ${\bf (H1)}$. Suppose further  $f, \beta, \gamma, g$ satisfy the structural condition {\bf (H2)}, and consider $0<\alpha<\min\{\alpha_{0},\alpha_{\beta},\alpha_{\gamma},\alpha_{g}\}$. Then, $u\in C^{2,\alpha}\left(\overline{\mathrm{B}^{+}_{\frac{1}{2}}}\right)$ with the following estimate
	\begin{eqnarray*} \|u\|_{C^{2,\alpha}\left(\overline{\mathrm{B}^{+}_{\frac{1}{2}}}\right)}\leq\mathrm{C}\left(\|u\|_{L^{\infty}(\mathrm{B}^{+}_{1})}+\|f\|_{C^{0, \alpha}(\mathrm{B}^{+}_{1})}+\|g\|_{C^{1,\alpha_g}(\overline{\mathrm{T}_{1}})}\right),
	\end{eqnarray*}
	for some universal constant $\mathrm{C}>1$.
\end{theorem}

\subsection*{Acknowledgments}

\hspace{0.4cm} This manuscript is part of the first author's Ph.D. thesis. He would like to express gratitude to the Department of Mathematics at the Universidade Federal do Cear\'{a} for fostering a pleasant and productive scientific atmosphere, which contributed to the successful outcome of this project. J.S. Bessa was partially supported by CAPES-Brazil under Grant No. 88887.482068/2020-00. J.V. da Silva and G.C. Ricarte have been partially supported by CNPq-Brazil under Grant No. 307131/2022-0, and No. 304239/2021-6.  J.V. da Silva has been partially supported by  FAEPEX-UNICAMP 2441/23 Editais Especiais - PIND - Projetos Individuais (03/2023). Part of this work was developed during the \textit{Fortaleza Conference on Analysis and PDEs} (2022) at the Universidade Federal do Cear\'{a} (UFC-Brazil). J.V. da Silva would like to thank to UFC's Department of Mathematics for fostering a pleasant scientific and research atmosphere during his visits in the Summer of 2022 and Spring of 2023.

\vspace{1cm}
\noindent  \textsc{Junior da S. Bessa} \hfill  \\
\hfill Universidade Federal Cear\'a  \\
\hfill Department of Mathematics \\
\hfill Fortaleza, CE-Brazil 60455-760\\
\hfill \texttt{junior.bessa@alu.ufc.br}\\
\vspace{0.5cm}

\noindent  \textsc{Jo\~{a}o Vitor  da Silva} \hfill  \\
\hfill  Universidade Estadual de Campinas - UNICAMP \\
\hfill Instituto de Matem\'{a}tica, Estat\'{i}stica e Computa\c{c}\~{a}o Cient\'{i}fica - IMECC\\
\hfill Departamento de Matem\'{a}tica \\
\hfill Rua S\'{e}rgio Buarque de Holanda, 651 \\
\hfill Campinas - SP, Brazil 13083-859\\
\hfill \texttt{jdasilva@unicamp.br}\\
\vspace{0.5cm}

\noindent  \textsc{Gleydson C. Ricarte} \hfill  \\
\hfill  Universidade Federal Cear\'a  \\
 \hfill Department of Mathematics \\
\hfill Fortaleza, CE-Brazil 60455-760\\
 \hfill \texttt{ricarte@mat.ufc.br}\\

\end{document}